\documentclass[10pt,draft,twoside]{article}

\usepackage{etex}                                               
\usepackage[english]{babel}                                     
\usepackage{amsmath,amsthm,amssymb,amsfonts,latexsym,cancel,centernot}    
\usepackage{enumerate}                                          
\usepackage{enumitem} \setlist[itemize]{noitemsep,topsep=0pt}   
\usepackage{pifont}                                             
\usepackage[latin1]{inputenc}                                   
\usepackage{lmodern}                                            
\usepackage[T1]{fontenc}                                        
\usepackage[none]{hyphenat}                                     
\usepackage{ragged2e} \justifying                               
\usepackage{marvosym,manfnt,calligra}                           
\usepackage[mathscr]{euscript}                                  
\usepackage{makeidx} \makeindex                                 
\usepackage[dvips]{graphicx}                                    
\usepackage{graphicx}                                           
\usepackage{color}                                              
\usepackage{verbatim}                                           
\usepackage{amssymb}                                            
\usepackage{amsfonts}                                           
\usepackage{amsmath}                                            
\usepackage{latexsym}                                           
\usepackage{amsthm}                                             
\usepackage{dsfont}                                             
\usepackage{float}                                              
\usepackage{wrapfig}                                            
\usepackage[rflt]{floatflt}                                     
\usepackage{nccmath}                                            
\usepackage{booktabs,longtable,multirow}
\usepackage[x11names,table]{xcolor}                             
\usepackage{multicol} 
\usepackage{tikz}                                               
\usetikzlibrary{intersections}                                  
\usetikzlibrary{mindmap,backgrounds,shapes,arrows,chains}       
\usepackage[all,matrix,arrow]{xy}                               
\usepackage{pgf,tikz-cd} \usetikzlibrary{arrows}                   
\usepackage{lipsum}
\usepackage{titling}
\usepackage[hidelinks]{hyperref}

\newtheoremstyle{newstyle}   
{} 			 				 
{12pt plus 1pt} 			 
{\mdseries} 				 
{} 
{\bfseries} 				 
{.} 						 
{ } 						 
{} 							 

\theoremstyle{newstyle}
\newtheorem{definition}{Definition}[section]
\newtheorem{theorem}[definition]{Theorem}

\newtheorem{lemma}[definition]{Lemma}
\newtheorem{proposition}[definition]{Proposition}
\newtheorem{fact}[definition]{Fact}

\newtheorem{question}[definition]{Question}

\allowdisplaybreaks
\DeclareMathAlphabet{\mathcalligra}{T1}{calligra}{m}{n}
\DeclareMathAlphabet{\mathpzc}{OT1}{pzc}{m}{it}
		
\def\f{\mathcal F}

\def\i{\mathcal I}

\title{ \textbf{Combinatorics of Ramsey ideals} }
\author{ Juli\'{a}n C. Cano $\;\;\;\;\;\;\;\;$ Carlos A. Di Prisco $\;\;\;\;\;\;\;\;$ Carlos Uzc\'{a}tegui-Aylwin}
\date{}
		
\newcommand{\Addresses}{{
		\bigskip
		\footnotesize
		
		Juli\'{a}n C. Cano, \textsc{Universidad de Los Andes (Bogot\'a).}\par\nopagebreak
		\textit{E-mail address}, J.C.~Cano: \texttt{jc.canor@uniandes.edu.co}
		
		\medskip
		
		Carlos A. Di Prisco, \textsc{Universidad de Los Andes (Bogot\'a), Instituto Venezolano de Investigaciones Cient\'{\i}ficas (Caracas), and Universidad Nebrija (Madrid).}\par\nopagebreak
		\textit{E-mail address}, C.A.~Di Prisco: \texttt{ca.di@uniandes.edu.co}
		
		\medskip
		
		Carlos Uzc\'{a}tegui-Aylwin, \textsc{Universidad Industrial de Santander (Bucaramanga).}\par\nopagebreak
		\textit{E-mail address}, C.~Uzc\'{a}tegui-Aylwin: \texttt{cuzcatea@saber.uis.edu.co}
		
}}
		
\begin{document}

\maketitle
			
\sloppy
			
\kern-2em			
\begin{abstract}
\noindent
We primarily study several combinatorial properties of Ramsey-type ideals on countably infinite sets. Specifically, we show new combinatorial characterizations of Ramsey ideals through various partition and convergence properties. Furthermore, we analyze ideal versions of some relevant high-dimensional Ramsey-type theorems, in order to research ideals related to finite colorings of barriers on the natural numbers as well as ideals associated with finite partitions of any family of finite subsets of the natural numbers. In particular, Galvin ideals are introduced as an intermediate combinatorial concept between Ramsey ideals and semiselective ideals. 

\smallskip

\noindent \textit{Key words and phrases:} Ideals of sets, Ramsey ideals, Galvin ideals.

\smallskip

\noindent \textit{2020 Mathematics Subject Classification:} 03E02, 03E05, 03E15, 05D10. 
\end{abstract}


\section{Introduction}

Combinatorics of ideals and filters of sets and its interaction with aspects of Ramsey theory constitutes a rich area of research. A central theme in this direction is the analysis of combinatorial and topological properties of ideals that reflect, extend, or localize classical partition principles. Within this framework, the development of ideal versions of Ramsey-type theorems has led to a variety of deep notions of ideals, which have been extensively studied from different perspectives in combinatorial set theory (see, for example, \cite{CGH2025, Hrusak2011, Uzc}). In this article, we continue this line of research by investigating new combinatorial aspects of Ramsey ideals, with particular emphasis on ideal versions of certain classical Ramsey-type results that have not been studied previously.

\medskip

Ramsey's theorem is a fundamental combinatorial principle asserting that, for any way of coloring the family of all finite subsets of $\omega$ of a fixed size with finitely many colors, there exists an infinite subset of $\omega$ such that its finite subsets of that size have the same color (see \cite{Ramsey}). In this spirit, the development of Ramsey-type theorems concerns the study of optimal conditions on certain families of subsets of $\omega$ endowed with some combinatorial structure that ensure the existence of infinite homogeneous sets for specific colorings of that structure. For instance, Erd\H{o}s--Rado's theorem deals with canonical colorings of the family of finite subsets of $\omega$ of a fixed size (see \cite{Erdos-Rado}), Nash-Williams' theorem deals with finite colorings of fronts and barriers on $\omega$ (see \cite{NashWilliams}), Galvin's lemma deals with finite colorings of the family of all finite subsets of $\omega$ (see \cite{Galvin}), and Ellentuck's theorem deals with Baire-measurable finite colorings of the family of all infinite subsets of $\omega$ (see \cite{Ellentuck}).

\medskip

Local Ramsey theory studies the combinatorial conditions on ideals on $\omega$ that lead to localized versions of Ramsey-type theorems, where homogeneous sets can be found within the associated coideals. In this vein, selective ideals were introduced by Mathias in \cite{Mathias} as a combinatorial property that implies an ideal version of Ellentuck's theorem (see also \cite[Section 9]{Todorcevic1997}). Subsequently, semiselective ideals were introduced by Farah in \cite{Farah1998} as the precise combinatorial notion required to obtain exactly an ideal version of Ellentuck's theorem (see also \cite[Section 7.1]{Todorcevic2010}). 

\medskip

On the other hand, Ramsey ideals are characterized as those special families that give rise to a localized version of Ramsey's theorem, in the sense that homogeneous sets whose existence is guaranteed by this theorem can be restricted to the corresponding coideals. Roughly speaking, an ideal $\mathcal{I}$ on $\omega$ is Ramsey if, for every infinite set outside $\mathcal{I}$ and every coloring of its pairs into two colors, there exists an infinite subset outside $\mathcal{I}$ that is homogeneous for the coloring. Ramsey ideals and their diverse combinatorial properties were initially studied by Booth and Grigorieff in the 1970s (see \cite{Booth, Grigorieff}), and they continue to be an active area of research at present (see, for example, \cite{FMRS07, FMRS11, GrebikUzca2019, Hrusak2011, Hrusak2017, Meza, PG, Todorcevic2010, Uzc}). In fact, an open question of particular interest, posed by Hru\v{s}\'{a}k, is whether analytic Ramsey tall ideals exist (see \cite[Question 5.19]{Hrusak2011} and \cite[Question 6.4]{Hrusak2017}). Consequently, the study of definable ideals and filters on countable sets remains of central importance, as does the interplay between their topological and combinatorial aspects.

\medskip

Our research deals with combinatorial aspects of Ramsey ideals, with an exploratory focus on ideal versions of certain Ramsey-type theorems. The material presented throughout this article is, as far as possible, self-contained and is organized as follows.

\medskip

In Section \ref{Ideals on countable sets}, we begin with a brief introduction to ideals of sets, along with some of their combinatorial classifications and main features, providing a general overview of the preliminaries needed for the development of this article.

\medskip

In Section 3, we review some recent results on Ramsey-type properties of ideals, and then we propose new characterizations of Ramsey ideals in terms of partition and convergence properties. In particular, we analyze an ideal version of the Erd\H{o}s--Rado canonization theorem in relation to Ramsey ideals. Furthermore, we answer a question posed in \cite[Last Remark of Section 4]{FMRS11} and \cite[Question 5.6]{HMTU2017} by showing that, for any member of a Ramsey coideal, every finite coloring of the family of its subsets of a fixed finite size admits a homogeneous set belonging to the corresponding coideal.

\medskip

In Section \ref{An ideal version of Galvin's lemma}, we introduce the notion of Galvin ideal, which is the corresponding ideal version of Galvin's lemma. We show that Galvin ideals constitute an intermediate combinatorial notion between Ramsey ideals and semiselective ideals, and we provide an example of a Ramsey tall ideal that is not Galvin as well as an example of a Galvin tall ideal that is not semiselective. Furthermore, we analyze ideals related to Nash-Williams' theorem on finite colorings of barriers and leave some questions on this topic open.

\medskip

In Section \ref{Remarks on non-hereditary Ramsey ideals}, we conclude with a discussion of non-hereditary Ramsey-type properties for tall ideals, including some remarks, combinatorial facts, and examples in this context. Specifically, we investigate the interplay among the non-hereditary Ramsey, Nash-Williams, and Galvin properties for tall ideals.

\medskip

In our view, we remark that the results of this article deepen the understanding of Ramsey-type properties of ideals and their ideal-theoretic counterparts. We hope that the approach developed here will prove useful for further investigations into the theory of ideals and filters as well as their connections with Ramsey theory and combinatorial set theory.

\section{Ideals on countable sets} 
\label{Ideals on countable sets}

In this section, we collect the preliminary material needed for the development of the
rest of the article, by introducing the
basic notation, definitions, and facts concerning combinatorial properties of ideals on countable sets.

\subsection{Some well-known combinatorial properties of ideals}

We use standard set-theoretic notation. In particular, for any set $X$ and $n \in \omega$, we denote by $[X]^{n}$ the family of subsets of $X$ of size $n$; similarly, $[X]^{<\omega}$ represents the family of finite subsets of $X$, and $[X]^{\omega}$ represents the family of countably infinite subsets of $X$. Moreover, for $s,t \in [\omega]^{<\omega}$ and $A\in [\omega]^{\omega}$, we write $s \sqsubset A$ if $s$ is an initial segment of $A$ in its increasing enumeration; analogously, we write $s \sqsubseteq t$ if $s$ is an initial segment of $t$, and we put $s \sqsubset t$ whenever $s \sqsubseteq t$ and $s\neq t$. Also, for $s\in [\omega]^{<\omega}$ and $n\in\omega$, we write $s<n$ to mean that either $s = \emptyset$ or $\max s < n$. Finally, given $A\in [\omega]^{\omega}$ and $n\in\omega$, the tail of $A$ above $n$ is the set $A/n = \{ m\in A : n<m \}$.

\begin{definition}
Let $\mathcal{I} \subseteq \wp(X)$ be a collection of subsets of a countably infinite set $X$. This collection $\mathcal{I}$ is said to be an \textit{ideal} on $X$ if it satisfies the following conditions:
\setlist{nolistsep}
\begin{enumerate}
\setlength{\itemsep}{0pt} 
    \item[i.] $[X]^{<\omega} \subseteq \mathcal{I}$ and $X \notin \mathcal{I}$.
    \item[ii.] If $A\subseteq B \subseteq X$ and $B \in \mathcal{I}$ then $A\in\mathcal{I}$.
    \item[iii.] If $A,B\in \mathcal{I}$ then $A\cup B \in \mathcal{I}$.
\end{enumerate}
\end{definition}

Given an ideal $\mathcal{I}$ on $X$, we write $\mathcal{I}^{+}$ to denote the family $\mathcal{I}^{+} = \{Y \in [X]^{\omega}:  Y \notin \mathcal{I}\}$, in which case we say that $\mathcal{I}^{+}$ is a coideal on $X$, and its elements are said to be $\i$-positive. If $Y\in \mathcal{I}^{+}$ is any $\mathcal{I}$-positive set, we define $\mathcal{I} \!\restriction\! Y = \{ A\in \mathcal{I} : A\subseteq Y \}$, which is an ideal on $Y$, called the restriction of $\mathcal{I}$ to $Y$. Similarly, we write $\mathcal{I}^{+} \!\restriction\! Y$ to indicate the collection $\mathcal{I}^{+} \!\restriction\! Y = \mathcal{I}^{+} \cap [Y]^{\omega}$. Also, for any family $\mathcal{K} \subseteq \wp(X)$ of subsets of $X$, we denote by $\mathcal{I}_{\mathcal{K}}$ the ideal on $X$ generated by $\mathcal{K}$, which is defined by $\mathcal{I}_{\mathcal{K}} = \{ A \subseteq X : (\exists \, \mathcal{S} \in [\mathcal{K}]^{<\omega} ) (A \subseteq^{*} \bigcup \mathcal{S}) \}$. 

\medskip

We adopt the standard descriptive set-theoretic notions and conventions. In particular, we identify $\wp(X)$ with $2^{X}$ through the natural homeomorphism that assigns to each subset of $X$ its characteristic function. Accordingly, we say that an ideal $\mathcal{I}$ on $X$ is $F_{\sigma}$, Borel, analytic, coanalytic, etc., whenever it is respectively $F_{\sigma}$, Borel, analytic, coanalytic, etc., as a subset of the Cantor cube $2^{X}$. Similarly, the same convention applies to any family $\mathcal{K} \subseteq \wp(X)$ of subsets of $X$.

\medskip

Fix an ideal $\mathcal{I}$ on $X$, and fix a quasi-order\footnote{A quasi-order, also known as a preorder, is a binary relation that is reflexive and transitive.} relation $\preceq$ on $\mathcal{I}^{+}$; also, let $\mathcal{D} \subseteq \mathcal{I}^{+}$ be given. We say that $\mathcal{D}$ is \textit{dense} on $(\mathcal{I}^{+} , \preceq  )$ whenever for every $N \in \mathcal{I}^{+}$ there is $M \in \mathcal{D}$ such that $M \preceq  N$. We say that $\mathcal{D}$ is \textit{open} on $(\mathcal{I}^{+} , \preceq )$ whenever for each $N \in \mathcal{I}^{+}$ and each $M \in \mathcal{D}$, if $N \preceq  M$ then $N \in \mathcal{D}$.  We say that $\mathcal{D}$ is \textit{dense-open} on $(\mathcal{I}^{+} , \preceq )$ if $\mathcal{D}$ is both dense and open. 

\medskip

For an ideal $\mathcal{I}$ on $X$ and a quasi-order relation $\preceq$ on $\mathcal{I}^{+}$, we say that:
\setlist{nolistsep}
\begin{itemize}
\setlength{\itemsep}{0pt} 

\item $(\mathcal{I}^{+} , \preceq  )$ is $\omega$-closed if for every decreasing sequence $\{A_{n}\}_{n\in \omega}$ on $(\mathcal{I}^{+} , \preceq  )$, there is $A\in \mathcal{I}^{+}$ such that $A \preceq A_{n}$ for each $n\in\omega$.

\item $(\mathcal{I}^{+} , \preceq  )$ is $\omega$-distributive if the intersection of countably many dense-open sets on $(\mathcal{I}^{+} , \preceq  )$ is also a dense-open set on $(\mathcal{I}^{+} , \preceq  )$.

\item $(\mathcal{I}^{+} , \preceq)$ is $(\omega,2)$-distributive if for every $B\in \mathcal{I}^{+}$ and every sequence $\{A_{n}\}_{n\in \omega}$ on $(\mathcal{I}^{+} , \preceq  )$, there is $A\in \mathcal{I}^{+}$ with $A \preceq B$ such that for each $n\in\omega$ either $A \preceq A_{n}$ or $A \preceq X \setminus A_{n}$.
\end{itemize}

\medskip

The \textit{Kat\v{e}tov (pre)order} $\leq_{K}$ on the class of all ideals on countable sets is defined as follows: for ideals $\mathcal{I}$ and $\mathcal{J}$ on $\omega$, we write $\mathcal{I} \leq_{K} \mathcal{J}$ if there exists a function $f\in \omega^{\omega}$ such that $f^{-1} [A] \in \mathcal{J}$ for all $A\in\mathcal{I}$. Moreover, the ideals $\mathcal{I}$ and $\mathcal{J}$ are said to be Kat\v{e}tov equivalent if $\mathcal{I} \leq_{K} \mathcal{J}$ and $\mathcal{J} \leq_{K} \mathcal{I}$.

\medskip

An ideal $\mathcal{I}$ on $X$ is said to be \textit{tall} if for every $A\in \mathcal{I}^{+}$ there exists $B\in [A]^{\omega}$ such that $B\in \mathcal{I}$. More generally, a family $\mathcal{K} \subseteq \wp(X)$ of subsets of $X$ is said to be \textit{tall} if $[A]^{\omega} \cap \mathcal{K} \neq \emptyset$ for every $A\in [X]^{\omega}$.

\medskip

On the other hand, a collection $\mathcal{U} \subseteq [X]^{\omega}$ of infinite subsets of $X$ is an \textit{ultrafilter} if it is a maximal coideal on $X$, or equivalently, a coideal closed under taking finite intersections of its elements.

\medskip

The central theme of this article is the combinatorial notion of a Ramsey ideal, which may be viewed as an ideal-theoretic counterpart of the Ramsey theorem. We now introduce the corresponding definition.

\begin{definition} \label{def_Ramsey_ideal*}
    An ideal $\mathcal I$ on $\omega$ is \textit{Ramsey} if for every $A\in \mathcal{I}^{+}$ and every coloring $\pi:[A]^{2} \rightarrow 2$ there exists $H\in \mathcal{I}^{+} \!\restriction\! A$ such that $\pi \,\text{''}\, [H]^{2} = \{i\}$ for some $i\in 2$.
\end{definition}

We review several well-known combinatorial properties of ideals that are closely related to the notion of a Ramsey ideal and also play an important role in the study of ideals and filters in combinatorial set theory (see, for example, \cite{CGH2025, Hrusak2011, Hrusak2017, HMTU2017, Todorcevic2010, Uzc}).

\begin{definition}
Given an ideal $\mathcal{I}$ on $\omega$, we say that:

\setlist{nolistsep}
\begin{itemize}
\setlength{\itemsep}{0pt} 

\item $\mathcal{I}$ is $p^{+}$ if for every sequence $\{A_{n}\}_{n\in\omega} \subseteq \mathcal{I}^{+}$ for which $A_{n+1} \subseteq^{*} A_{n}$ for each $n\in\omega$, there exists some $B\in \mathcal{I}^{+}$ such that $B\subseteq^{*} A_{n}$ for all $n\in \omega$.

\item $\mathcal{I}$ is $p^{w}$ if for every $A\in \mathcal{I}^{+}$ and every sequence $\{ \mathcal{D}_{n} \}_{n\in\omega} \subseteq \wp(\mathcal{I}^{+})$ of dense-open sets on $(\mathcal{I}^{+}, \subseteq)$ or $(\mathcal{I}^{+}, \subseteq^{*})$, there exists some $B\in \mathcal{I}^{+} \!\restriction\! A$ with the property that for each $n\in\omega$ there is $A_{n}\in \mathcal{D}_{n}$ such that $B \subseteq^{*} A_{n}$.

\item $\mathcal{I}$ is $p^{-}$ if for every $A\in \mathcal{I}^{+}$ and every infinite partition $A = \bigcup_{k\in\omega} f_{k}$, where each $f_{k} \in \mathcal{I}$, there exists some $B\in \mathcal{I}^{+} \!\restriction\! A$ such that for each $k\in \omega$ the set $f_{k}\cap B$ is finite. 

\item $\mathcal{I}$ is $q^{+}$ if for every $A\in \mathcal{I}^{+}$ and every infinite partition $A = \bigcup_{k\in\omega} f_{k}$, where each $f_{k}$ is finite, there exists some $B\in \mathcal{I}^{+} \!\restriction\! A$ such that for each $k\in \omega$ the set $f_{k}\cap B$ has at most one element.

\item $\mathcal{I}$ is $\textrm{h-FinBW}$\footnote{The notation $\textrm{h-FinBW}$ abbreviates the hereditary finite Bolzano--Weierstrass property (see \cite[Section 2]{FMRS07}).} if for every $A\in\mathcal{I}^{+}$ and every bounded sequence of real numbers $\{ x_{n} \}_{n\in A} \subseteq \mathbb{R}$ indexed by $A$, there exists some $B\in \mathcal{I}^{+} \!\restriction\! A$ such that the subsequence $\{ x_{n} \}_{n\in B}$ is convergent.

\item $\mathcal{I}$ is $\textrm{h-Mon}$\footnote{The notation $\textrm{h-Mon}$ abbreviates the hereditary monotone property (see \cite[Subsection 3.2]{FMRS11}).} if for every $A\in\mathcal{I}^{+}$ and every sequence of real numbers $\{ r_{n} \}_{n\in A} \subseteq \mathbb{R}$ indexed by $A$, there exists some $B\in \mathcal{I}^{+} \!\restriction\! A$ such that the subsequence $\{ r_{n} \}_{n\in B}$ is monotone.

\item $\mathcal{I}$ is \textit{selective} if for every sequence $\{ {A}_{n} \}_{n\in \omega} \subseteq \mathcal{I}^{+}$ for which $A_{n+1} \subseteq^{*} A_{n}$ for each $n\in\omega$, there exists some $A_{\infty} \in \mathcal{I}^{+}$ such that $A_{\infty}/m \subseteq {A}_{m}$ for all $m\in A_{\infty}$. In this case $A_\infty$ is said to be a diagonalization of the sequence  $\{ {A}_{n} \}_{n\in \omega}$.

\item $\mathcal{I}$ is \textit{semiselective} if for each $A\in \mathcal{I}^{+}$ and for every sequence $\{ \mathcal{D}_{n} \}_{n\in \omega} \subseteq \wp(\mathcal{I}^{+})$ of dense-open sets on $(\mathcal{I}^{+}, \subseteq)$, there exists some $D_{\infty} \in \mathcal{I}^{+} \!\restriction\! A$ such that $D_{\infty}/m \in \mathcal{D}_{m}$ for all $m\in D_{\infty}$. In this case $D_\infty$ is said to be a diagonalization of the sequence $\{\mathcal {D}_{n} \}_{n\in \omega}$.

\item $\mathcal{I}$ is \textit{weakly selective} if for every $A\in \mathcal{I}^{+}$ and every infinite partition $A = \bigcup_{k\in\omega} f_{k}$, where each $f_{k} \in \mathcal{I}$, there exists some $B\in \mathcal{I}^{+} \!\restriction\! A$ such that for each $k\in \omega$ the set $f_{k}\cap B$ has at most one element.

\item $\mathcal{I}$ is \textit{bisequential} if for every ultrafilter $\mathcal{U} \subseteq \mathcal{I}^{+}$ there exists a sequence $\{A_{n}\}_{n\in\omega} \subseteq \mathcal{U}$ with the property that for each $I\in \mathcal{I}$ there is some $n\in\omega$ such that $I\cap A_{n} = \emptyset$.  

\item $\mathcal{I}$ is \textit{Fr\'{e}chet} if for every $A\in \mathcal{I}^{+}$ there is $B\in \mathcal{I}^{+} \!\restriction\! A$ such that $[B]^{\omega} \subseteq \mathcal{I}^{+}$.

\end{itemize}    
\end{definition}

The relationships and interactions among the combinatorial properties of ideals introduced above are summarized in the following diagram, which displays the strict implications between these properties.

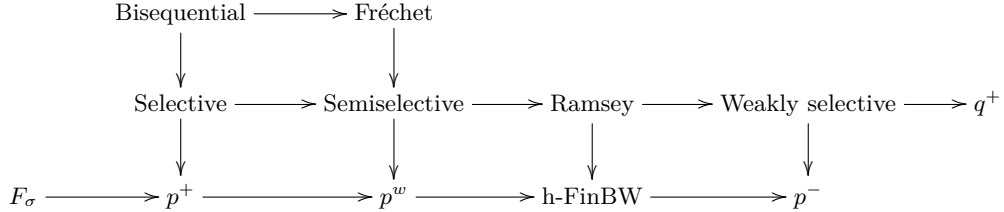
\begin{equation*}
\small
\xymatrix@C=2.5em@R=2.1em{
 & \text{Bisequential} \ar@{->}[r] \ar@{->}[d] & \text{Fr\'echet} \ar@{->}[d] & & &
 \\
& \text{Selective} \ar@{->}[r] \ar@{->}[d] & \text{Semiselective} \ar@{->}[r] \ar@{->}[d] & \text{Ramsey} \ar@{->}[r] \ar@{->}[d] & \text{Weakly selective} \ar@{->}[r] \ar@{->}[d] & q^{+}
 \\
F_{\sigma} \ar@{->}[r] & p^{+} \ar@{->}[r] & p^{w} \ar@{->}[r] & \textrm{h-FinBW} \ar@{->}[r] & p^{-} &
}
\end{equation*}

\medskip

It is worth mentioning several remarks concerning the combinatorial properties of ideals introduced above. First, the property $\textrm{h-Mon}$ does not appear in the diagram because, as shown in \cite[Theorem 3.16]{FMRS11}, it is equivalent to the ideal Ramsey property. Furthermore, it is straightforward to verify that an ideal is weakly selective if and only if it is both $q^{+}$ and $p^{-}$. Moreover, note that an ideal $\mathcal{I}$ is Fr\'{e}chet if and only if, for every $A \in \mathcal{I}^{+}$, the restricted ideal $\mathcal{I} \!\restriction\! A$ is not tall. In particular, there are no tall Fr\'{e}chet ideals, and consequently no tall bisequential ideals. 

\medskip

On the other hand, Todor\v{c}evi\'{c} proved in \cite[Theorem 7.53]{Todorcevic2010} that every analytic selective ideal is bisequential, and hence Fr\'{e}chet. Earlier, Mathias proved in \cite[Proposition 4.6]{Mathias} that there are no analytic selective tall ideals, and this result was later strengthened by Farah in \cite{Farah1998}, who showed that there are no analytic semiselective tall ideals (see also \cite[Corollary 7.20]{Todorcevic2010}). Finally, Mazur proved in \cite{Mazur} that every $F_{\sigma}$ ideal is $\mathrm{p}^{+}$ (see also \cite[Lemma 3.3]{HMTU2017}), from which it follows that there are no $F_{\sigma}$ Ramsey tall ideals. Nevertheless, it remains an open question, posed by Hru\v{s}\'{a}k in \cite[Question 5.19]{Hrusak2011} and \cite[Question 6.4]{Hrusak2017}, whether there exists an analytic Ramsey tall ideal.

\medskip

We conclude this brief overview of some combinatorial properties of ideals and their interrelations by presenting two important characterizations of selective and semiselective ideals, respectively (see, for example, \cite[Lemmas 7.4 and 7.9]{Todorcevic2010}).

\begin{proposition}[\cite{Todorcevic2010, Todorcevic1997}] \label{selective=(p)+(q)}
Let $\mathcal I$ be an ideal on $\omega$. Then, the following statements are equivalent:
\setlist{nolistsep}
\begin{itemize}
\setlength{\itemsep}{0pt} 
\item[(a)] $\mathcal I$ is selective.
\item[(b)] $\mathcal I$ is $q^{+}$ and $p^{+}$.
\item[(c)] $\mathcal I$ is $q^{+}$ and $(\mathcal{I}^{+}, \subseteq^{*})$ is $\omega$-closed.
\end{itemize}
\end{proposition}

\begin{proposition}[\cite{Farah1998, Todorcevic2010}] \label{semiselective=(p^w)+(q)}
Let $\mathcal I$ be an ideal on $\omega$. Then, the following statements are equivalent:
\setlist{nolistsep}
\begin{itemize}
\setlength{\itemsep}{0pt} 
\item[(a)] $\mathcal I$ is semiselective.
\item[(b)] $\mathcal I$ is $q^{+}$ and $p^{w}$.
\item[(c)] $\mathcal I$ is $q^{+}$ and $(\mathcal{I}^{+}, \subseteq^{*})$ is $\omega$-distributive.
\end{itemize}
\end{proposition}

We point out that Ramsey ideals also admit a combinatorial description analogous to the two preceding characterizations of selective and semiselective ideals, but this result will be presented at the beginning of Section \ref{An ideal version of Ramsey's theorem}.

\subsection{Some classical examples of ideals}

We conclude this section by presenting several well-known examples of ideals that are relevant to the main topics covered and developed in this article, some of which will reappear throughout the paper.

\medskip

\setlist{nolistsep}
\begin{itemize}
\setlength{\itemsep}{0pt} 

    \item Note that the collection $\mathrm{FIN} = [\omega]^{<\omega}$ of all finite subsets of $\omega$ is an ideal on $\omega$, and it is contained in every other ideal on $\omega$.

    \medskip

    \item Recall that an infinite family  $\mathcal{A}$ of infinite subsets of $\omega$ is \textit{almost disjoint} if $|A\cap B|<\omega$ for every $A,B\in\mathcal{A}$ with $A \neq B$; moreover, an almost disjoint family $\mathcal{A}$ is \textit{maximal} if for every $X\in [\omega]^{\omega}$ there exists $A\in\mathcal{A}$ such that $|X\cap A| = \omega$. A known fact is that the ideal $\mathcal{I}_{\mathcal{A}}$ generated by an almost disjoint family $\mathcal{A}$ on $\omega$ is the quintessential example of a selective ideal on $\omega$, as Mathias showed in \cite[Proposition 0.7]{Mathias} (see also \cite[Example 7.1.2]{Todorcevic2010}). Furthermore, $\mathcal{I}_{\mathcal{A}}$ is tall if and only if the almost disjoint family $\mathcal{A}$ is maximal.

    \medskip
    
    \item Denote by $\emptyset \times \mathrm{FIN}$ the ideal of all subsets of $\omega \times \omega$ whose vertical sections are finite. Farah shows in \cite[Example 2.1]{Farah1998} that $\emptyset \times \mathrm{FIN}$ is semiselective but not selective. Moreover, it is easy to check that this ideal is also Fr\'{e}chet.

    \medskip

    \item Let $\mathcal{A}$ be a maximal almost disjoint family on $\omega$, and let $\emptyset \times \mathcal{I}_{\mathcal{A}}$ denote the ideal of all subsets of $\omega \times \omega$ whose vertical sections belong to $\mathcal{I}_{\mathcal{A}}$, that is 

     \kern-0.5em
    \begin{equation*}
        \emptyset \times \mathcal{I}_{\mathcal{A}} = \{ X \subseteq \omega \times \omega : (\forall\, n\in\omega) ( \{ m\in\omega : (n,m)\in X \} \in \mathcal{I}_{\mathcal{A}})\}.
    \end{equation*}

     \kern-0.5em
   Observe that $\emptyset \times \mathcal{I}_{\mathcal{A}}$ is a tall semiselective ideal which is neither Fr\'{e}chet nor selective. To the best of our knowledge, this simple example is new.
    
    \medskip
    
    \item Let $\mathbf{nwd} (\mathbb{Q})$ denote the ideal of all nowhere dense subsets of $\mathbb{Q}$ with respect to its usual topology. This ideal is weakly selective but not Ramsey (see, for example, \cite[Section 3, Item 1]{Hrusak2017}).

    \medskip
    
    \item Let $\pi: [\omega]^{2} \rightarrow 2$ be a coloring for which $\omega$ is not cover by finitely many homogeneous sets for $\pi$, and let $\mathrm{hom}(\pi)$ denote the family of all infinite homogeneous sets of $\pi$. The ideal $\mathcal{I}_{\mathrm{hom}(\pi)}$ generated by $\mathrm{hom}(\pi)$ is $F_{\sigma}$ and tall; in particular, this ideal is $\textrm{h-FinBW}$ but not Ramsey (see, for example, \cite[Proposition 3.4]{FMRS07}). 

    \medskip

    \item Let $u:[\omega]^2 \to 2$ denote the universal coloring associated with the \textit{random graph} (see \cite[Section 3]{Erdos-Renyi}); namely, $u$ satisfies the property that for every pair of disjoint finite sets $f_0,f_1 \subseteq \omega$, there exists $\ell \in \omega$ such that $u(\{m,\ell\}) = i$ for every $m\in f_{i}$ and every $i\in 2$ (see also \cite[Section 1]{Cameron}). The \textit{random graph ideal} $\mathcal{R}$ is the ideal generated by the infinite homogeneous sets of $u$, that is $\mathcal{R} = \mathcal{I}_{\textrm{hom}(u)}$ (see \cite[Subsection 3.5]{Hrusak2011}).
    
    \medskip

    \item The random graph ideal admits natural higher-dimensional extensions. Let $n,k\in \omega$ with $n,k\geq 2$, and let $u^{n}_{k} : [\omega]^{n} \rightarrow k$ be a coloring satisfying the property that for every collection of pairwise disjoint finite sets $F_{0}, \ldots, F_{k-1} \subseteq [\omega]^{n-1}$, there exists $\ell\in\omega$ such that $u^{n}_{k} (x\cup \{l\}) =i$ for every $x\in F_{i}$ and every $i\in k$. Then, the ideal $\mathcal{R}_k^n$, which generalizes the random graph ideal $\mathcal{R}$, is defined by $\mathcal{R}^n_k = \mathcal{I}_{\textrm{hom}(u^n_k)}$ (see \cite[Definition 3.4]{PG}).
    
    \medskip

    \item The ideal $\mathbf{conv}$ is defined as the ideal on $\mathbb{Q} \cap [0,1]$ generated by sequences in $\mathbb{Q}\cap [0,1]$ that converges to a point in $[0,1]$ (see, for example, \cite[Section 3, Item 5]{Hrusak2017}).

    \medskip

    \item The summable ideal $\mathcal{I}_{1/n}$ is defined by $\mathcal{I}_{1/n} = \{ X\subseteq \omega : \sum_{n\in X} \frac{1}{n} < \infty\}$ (see, for example, \cite[Section 3, Item 8]{Hrusak2017}).

    \medskip

    \item The ideal $\mathcal{WR}$ is defined as the ideal on $\omega \times \omega$ generated by vertical lines and sets $G$ such that for every $(i,j) , (k,l) \in G$ either $i>k+l$ or $k>i+j$ (see \cite[Definition 1.1]{Kwela}). 

\end{itemize}

\section{An ideal version of Ramsey's theorem}
\label{An ideal version of Ramsey's theorem}

In this section, we undertake a study of Ramsey ideals from several complementary perspectives, with particular emphasis on combinatorial characterizations of this concept involving partition and convergence properties.

\subsection{Ramsey ideals}

We begin this section by introducing the notion of a Ramsey ideal, which may be viewed as an ideal-theoretic version of the Ramsey theorem. Recall that, given a coloring $\pi:[A]^{n} \rightarrow k$, a set $H \subseteq A$ is said to be \textit{homogeneous} for $\pi$ if $\pi$ is constant on $[H]^{n}$. In this regard, the Ramsey theorem asserts the following.

\begin{proposition}[\cite{Ramsey}]
    For all integers $n,k \geq 2$, every $A\in [\omega]^{\omega}$, and every coloring $\pi:[A]^{n} \rightarrow k$, there exists $H\in [A]^{\omega}$ such that $H$ is homogeneous for $\pi$.
\end{proposition}

Let $\mathcal{I}$ be an ideal on $\omega$, and fix $n,k\in \omega$ with $n,k \geq 2$. The ideal $\mathcal I$ satisfies

\kern-0.5em
\begin{equation*}
    \mathcal{I}^+ \longrightarrow (\mathcal{I}^+)^{n}_{k}
\end{equation*}

\kern-0.5em
if, for every $A\in \mathcal{I}^+$ and every coloring $\pi:[A]^{n} \rightarrow k$ of the collection $[A]^{n}$ into $k$ colors, there exists some $B\in \mathcal{I}^+ \!\restriction\! A$ such that $B$ is homogeneous for $\pi$.

\medskip

It is worth mentioning that for each fixed $n\in\omega$ with $n\geq 2$, the property $\mathcal{I}^+ \longrightarrow (\mathcal{I}^+)^{n}_{k}$ for ideals is independent of the number of colors $k\in \omega$ with $k\geq 2$.

\begin{fact} \label{fact_colors}
    Let $\mathcal{I}$ be an ideal on $\omega$ and let $n \geq 2$. Then, the following statements are equivalent:
    \setlist{nolistsep}
\begin{itemize}
\setlength{\itemsep}{0pt} 
\item[(a)] $\mathcal{I}^+ \longrightarrow (\mathcal{I}^+)^{n}_{2}$ holds.
\item[(b)] $\mathcal{I}^+ \longrightarrow (\mathcal{I}^+)^{n}_{k}$ holds for all $k\geq 2$.
\end{itemize}
\end{fact}

\begin{proof}
    $[\text{(b)} \Longrightarrow \text{(a)}].$ Straightforward.

    $[\text{(a)} \Longrightarrow \text{(b)}].$ By induction on $k$, suppose that $\mathcal{I}^+ \longrightarrow (\mathcal{I}^+)^{n}_{k}$ holds for a fixed $k \geq 2$, and consider a coloring $\pi: [A]^{n} \rightarrow k+1$ with $A\in \mathcal{I}^{+}$. Define $\phi: [A]^{n} \rightarrow 2$ by $\phi(x) = 0$ if and only if $\pi(x) = k$. Then, there is $H\in \mathcal{I}^{+} \!\restriction\! A$ such that $\phi$ is constant on $[H]^{n}$. Clearly, if $\phi \,\text{''}\, [H]^{n} = \{0\}$ then $\pi \,\text{''}\, [H]^{n} = \{k\}$. Otherwise, if $\phi \,\text{''}\, [H]^{n} = \{1\}$ then $\pi(x) \neq k$ for each $x\in [H]^{n}$; so, take $\widetilde{\pi} : [H]^{n} \rightarrow k$ given by $\widetilde{\pi} = \pi \!\restriction\![H]^{n}$. The induction hypothesis $\mathcal{I}^+ \longrightarrow (\mathcal{I}^+)^{n}_{k}$, implies that there are $i<k$ and $\widetilde{H} \in \mathcal{I}^{+} \!\restriction\! H \subseteq \mathcal{I}^{+} \!\restriction\! A$ such that $\widetilde{\pi} \,\text{''}\, [\widetilde{H}]^{n} = \{i\}$, and hence $\pi \,\text{''}\, [\widetilde{H}]^{n} = \{i\}$. Thus $\mathcal{I}^+ \longrightarrow (\mathcal{I}^+)^{n}_{k+1}$ also holds.
\end{proof}

\smallskip

Ramsey ideals provide a localized version of the Ramsey theorem by ensuring the existence of a positive homogeneous subset for every coloring of pairs of a positive set, as formulated in Definition \ref{def_Ramsey_ideal*}. We emphasize that most of the results in this article concern the ideal Ramsey property, whose definition we recall below.

\begin{definition} \label{def_Ramsey_ideal}
An ideal $\mathcal I$ on $\omega$ is \textit{Ramsey} if it satisfies
$\mathcal{I}^+ \longrightarrow (\mathcal{I}^+)^{2}_{2}$.
\end{definition}

In \cite[Section 2]{Farah1998}, Farah claims that every semiselective ideal is Ramsey, and provides a counterexample in \cite[Example 2.2]{Farah1998} showing that the converse does not hold in general. Moreover, this result was later refined in \cite[Proposition 5.4]{HMTU2017}, where it is shown that every semiselective ideal
$\mathcal{I}$ satisfies $\mathcal{I}^+ \longrightarrow (\mathcal{I}^+)^{n}_{k}$ for all $n,k \geq 2$.

\medskip

Ramsey ideals can be characterized in various ways through different combinatorial properties that are interconnected. The following proposition corresponds to an important description of those ideals $\mathcal{I}$ satisfying the combinatorial property $\mathcal{I}^+ \longrightarrow (\mathcal{I}^+)^{2}_{2}$, as presented by Filip\'{o}w, Mro\.{z}ek, Rec{\l}aw, and Szuca in \cite[Theorem 3.16]{FMRS11}.

\begin{proposition} [\cite{FMRS11}]  \label{Ramsey = (q^{+})+(h-FinBW)}
    Let $\mathcal I$ be an ideal on $\omega$. Then, the following statements are equivalent:
\setlist{nolistsep}
\begin{itemize}
\setlength{\itemsep}{0pt} 
\item[(a)] $\mathcal I$ is Ramsey.
\item[(b)] $\mathcal I$ is $\textrm{h-Mon}$.
\item[(c)] $\mathcal I$ is $q^{+}$ and $\textrm{h-FinBW}$.
\item[(d)] For every $A\in \mathcal{I}^{+}$ along with every collection $\{ A_{s}  :  s\in 2^{<\omega} \} \subseteq \wp(A)$, with $A_{\emptyset} = A$, such that $A_{s} = A_{s^{\smallfrown}0} \cup A_{s^{\smallfrown}1}$ and $A_{s^{\smallfrown}0} \cap A_{s^{\smallfrown}1} = \emptyset$ for each $s\in 2^{<\omega}$, there exist $D \in \mathcal{I}^{+} \!\restriction\! A$ and $\alpha \in 2^{\omega}$ such that $D/m \subseteq A_{\alpha \restriction (m+1)}$ for all $m\in D$. 
\end{itemize}
\end{proposition}

Another valuable characterization of Ramsey ideals is analogous to the corresponding ones for selective and semiselective ideals presented in Propositions \ref{selective=(p)+(q)} and \ref{semiselective=(p^w)+(q)}, respectively, as established by Hru\v{s}\'{a}k, Meza-Alc\'{a}ntara, Th\"{u}mmel, and Uzc\'{a}tegui in \cite[Proposition 4.10]{HMTU2017}.

\begin{proposition} [\cite{HMTU2017}]  \label{Ramsey = (q^{+})+(w2-dist)}
    Let $\mathcal I$ be an ideal on $\omega$. Then, the following statements are equivalent:
\setlist{nolistsep}
\begin{itemize}
\setlength{\itemsep}{0pt} 
\item[(a)] $\mathcal I$ is Ramsey.
\item[(b)] $\mathcal I$ is $q^{+}$ and $(\mathcal{I}^{+}, \subseteq^{*})$ is $(\omega,2)$-distributive.
\end{itemize}
\end{proposition}

One of the most notable characterizations of Ramsey ideals, due to Hru\v{s}\'ak in \cite[Proposition 3.11]{Hrusak2011} (see also \cite[Corollary 4.11]{HMTU2017}), is expressed in terms of the Kat\v{e}tov order $\leq_{K}$ and its interaction with the random graph ideal $\mathcal{R}$.

\begin{proposition} [\cite{Hrusak2011}]  \label{Ramsey_and_RandomGraph}
    Let $\mathcal I$ be an ideal on $\omega$. Then, the following statements are equivalent:
\setlist{nolistsep}
\begin{itemize}
\setlength{\itemsep}{0pt} 
\item[(a)] $\mathcal I$ is Ramsey.
\item[(b)] $\mathcal{R} \not\leq_{K} \mathcal{I} \!\restriction\! X$ for every $X\in\mathcal{I}^{+}$.
\end{itemize}
\end{proposition}

On the other hand, bringing together some results that were independently stated in \cite[Subsection 2.3]{FMRS07}, \cite[Subsection 2.2]{FMRS11}, \cite[Subsection 2.2]{FMRS12}, \cite[Subsection 3.4]{Hrusak2011}, and \cite[Subsection 5.1]{HMTU2017}, we highlight the following characterizations of the property $\textrm{h-FinBW}$ for ideals, which constitutes a combinatorial notion strictly weaker than the ideal Ramsey property. 

\begin{proposition} [\cite{FMRS11, HMTU2017}]  \label{(h-FinBW)}
    Let $\mathcal I$ be an ideal on $\omega$. Then, the following statements are equivalent:
\setlist{nolistsep}
\begin{itemize}
\setlength{\itemsep}{0pt}
\item[(a)] $\mathcal I$ is $\textrm{h-FinBW}$.
\item[(b)] For some uncountable compact metrizable space $\mathcal{X}$, and for every $A\in\mathcal{I}^{+}$ along with every sequence $\{ x_{n} \}_{n\in A} \subseteq \mathcal{X}$ indexed by $A$, there exists some $B\in \mathcal{I}^{+} \!\restriction\! A$ such that the subsequence $\{ x_{n} \}_{n\in B}$ is convergent.
\item[(c)] For each uncountable compact metrizable space $\mathcal{X}$, and for every $A\in\mathcal{I}^{+}$ along with every sequence $\{ x_{n} \}_{n\in A} \subseteq \mathcal{X}$ indexed by $A$, there exists some $B\in \mathcal{I}^{+} \!\restriction\! A$ such that the subsequence $\{ x_{n} \}_{n\in B}$ is convergent.
\item[(d)] For every $A\in \mathcal{I}^{+}$ along with every collection $\{ A_{s} : s\in 2^{<\omega} \} \subseteq \wp(A)$, with $A_{\emptyset} = A$, such that $A_{s} = A_{s^{\smallfrown}0} \cup A_{s^{\smallfrown}1}$ and $A_{s^{\smallfrown}0} \cap A_{s^{\smallfrown}1} = \emptyset$ for each $s\in 2^{<\omega}$, there exist $B \in \mathcal{I}^{+} \!\restriction\! A$ and $\alpha \in 2^{\omega}$ such that $B \subseteq^{*} A_{\alpha \restriction n}$ for all $n\in \omega$.
\item[(e)] $(\mathcal{I}^{+}, \subseteq^{*})$ is $(\omega,2)$-distributive.
\item[(f)] $\mathbf{conv} \not\leq_{K} \mathcal{I} \!\restriction\! X$ for every $X\in\mathcal{I}^{+}$.
\end{itemize}
\end{proposition}

Finally, we present a diagram illustrating the role of Ramsey ideals within in a relevant portion of the hierarchy of combinatorial properties of ideals. The diagram specifies the implications among selective, semiselective, Ramsey, and weakly selective ideals; moreover, all these combinatorial notions coincide when the corresponding coideal is an ultrafilter (see, for example, \cite[Proposition 11.7]{Halbeisen}), in which case one obtains the classical result given by Kunen asserting that an ultrafilter is Ramsey if and only if it is both a q-point and a p-point (see, for example, \cite[Fact 11.11]{Halbeisen}).

\smallskip

\begin{equation*}
    \small
		\xymatrix@!0@C=3.7em@R=3.7em
		{ 
			\text{Selective} \ar@2{->}[d] & \ar@2{<->}[rr]  &  & \;\; q^{+} \ar@{}[d]|{\;\;\parallel} \ar@{}[r]|{\&} & p^{+} \ar@{->}[d]
			\\ 
			\text{Semiselective} \ar@2{->}[d] & \ar@2{<->}[rr]  &  & \;\; q^{+} \ar@{}[d]|{\;\;\parallel} \ar@{}[r]|{\&} & p^{w} \ar@{->}[d]
			\\
			\text{Ramsey} \ar@2{->}[d] & \ar@2{<->}[rr]  &  & \;\; q^{+} \ar@{}[d]|{\;\;\parallel} \ar@{}[r]|{\&} & \textrm{ \small \;\;\;\; h-FinBW} \ar@{->}[d]
			\\ 
			\text{Weakly selective} \;\;\;\; \ar@{-->}@/^{16mm}/[uuu]^{\textstyle\textit{\rotatebox{90}{Ultrafilter}}} & \ar@2{<->}[rr] &  & \;\; q^{+} \ar@{}[r]|{\&} & p^{-} \ar@{-->}@/^{-16mm}/[uuu]_{\textstyle\textit{\rotatebox{270}{Ultrafilter}}}
		}
	\end{equation*}

\smallskip

\subsection{Systems with diagonalizations}

In \cite[Last Remark of Section 4]{FMRS11} and \cite[Question 5.6]{HMTU2017}, the following problem about Ramsey ideals has been essentially asked: Is there a Ramsey ideal $\mathcal{I}$ satisfying $\mathcal{I}^+ \centernot \longrightarrow (\mathcal{I}^+)^{n}_{2}$ for some $n>2$? 

\medskip

The main goal of this section is to provide a negative answer to this question. Specifically, in Theorem \ref{Ramsey_ideal_theorem}, we aim to prove that an ideal $\mathcal{I}$ is Ramsey if and only if it satisfies $\mathcal{I}^+ \longrightarrow (\mathcal{I}^+)^{n}_{k}$ for all $n,k \geq 2$. 

\medskip

We begin by introducing the notions of systems with pseudo-intersections and diagonalizations relative to an ideal $\mathcal{I}$ on $\omega$. These combinatorial concepts for ideals are inspired by related properties previously considered in \cite[Definition 3.1]{Farah02}, \cite[Proposition 3.3]{FMRS07}, \cite[Theorem 3.16]{FMRS11}, and \cite[Proposition 2.10]{FMRS12}; however, in our case, we work with finitely branching trees without any bound on the branching.

\begin{definition} \label{def_system_1}
    Let $\mathcal{I}$ be an ideal on $\omega$. A system of subsets of a $\mathcal{I}$-positive set $A\in \mathcal{I}^{+}$ is a collection $\{ A_{s} : s\in \omega^{<\omega} \} \subseteq \wp(A)$, with $A_{\emptyset} = A$, for which there exists a sequence $\{ a_{n} \}_{n\in \omega} \subseteq \omega$, with each $a_{n} \geq 2$, such that for every $s\in \omega^{<\omega}$ both $A_s=\bigcup_{i<a_{|s|}} A_{s^{\smallfrown} i}$ and $A_{s^{\smallfrown} i} \cap A_{s^{\smallfrown} j}=\emptyset$ if $i< j < a_{|s|}$.
\end{definition}

In other words, a system of subsets of a $\mathcal{I}$-positive set $A\in \mathcal{I}^{+}$ can be viewed as a finitely branching tree of height $\omega$, formed by subsets of $A$, whose root is $A$ and in which each node is finitely partitioned by its immediate successors.

\begin{definition} \label{def_system_2}
    Given an ideal $\mathcal{I}$ on $\omega$, we say that:

    
    \setlist{nolistsep}
    \begin{itemize}
    \setlength{\itemsep}{0pt}

    \item $\mathcal{I}$ satisfies the property of systems with pseudo-intersections, or simply $\mathcal{I}$ is $\textrm{SPI}$, if for each $A\in\mathcal{I}^{+}$ and each system $\{ A_{s} : s\in \omega^{<\omega} \}$ of subsets of $A$, there exist $B\in \mathcal{I}^{+} \!\restriction\! A$ and $x \in \prod_{n\in\omega} a_{n}$ such that $B$ is a pseudo-intersection of $\{ A_{x\restriction n}\}_{n\in \omega}$, that is, $B \subseteq^{*} A_{x\restriction n}$ for all $n\in \omega$. 

    
    \item $\mathcal{I}$ satisfies the property of systems with diagonalizations, or simply $\mathcal{I}$ is $\textrm{SD}$, if for each $A\in\mathcal{I}^{+}$ and each system $\{ A_{s} : s\in \omega^{<\omega} \}$ of subsets of $A$, there exist $D\in \mathcal{I}^{+} \!\restriction\! A$ and $x \in \prod_{n\in\omega} a_{n}$ such that $D$ is a diagonalization of $\{ A_{x\restriction n}\}_{n\in \omega}$, that is, $D/m \subseteq A_{x\restriction (m+1)}$ for all $m\in D$.
 
    \end{itemize}  
\end{definition}

The following result is a slight generalization of \cite[Proposition 2.8]{FMRS11}, which characterizes the property $\textrm{h-FinBW}$ in terms of systems with pseudo-intersections. For the sake of completeness, we provide a sketch of its proof.

\begin{lemma}  \label{SPI}
    Let $\mathcal I$ be an ideal on $\omega$. Then, the following statements are equivalent:
\setlist{nolistsep}
\begin{itemize}
\setlength{\itemsep}{0pt}
\item[(a)] $\mathcal I$ is $\textrm{h-FinBW}$.
\item[(b)] $\mathcal I$ is $\textrm{SPI}$.
\end{itemize}
\end{lemma}

\begin{proof}
    $[\text{(a}) \Longrightarrow \text{(b)}].$ Suppose that $\mathcal{I}$ is $\textrm{h-FinBW}$. Let $A\in \mathcal{I}^{+}$ and let $\{ A_{s} : s\in \omega^{<\omega} \} \subseteq \wp(A)$ be as in Definition \ref{def_system_1}; so, note that for every $k\in A$ there exists a unique $x_{k} \in \prod_{n\in\omega} a_{n}$ such that $k\in \bigcap_{n\in\omega} A_{x_{k} \restriction n}$. Consider the sequence $\{ x_{k}\}_{k\in A}$ in the uncountable compact metrizable space $\prod_{n\in\omega} a_{n}$, then there exist $B\in \mathcal{I}^{+} \!\restriction\! A$ and $x\in \prod_{n\in\omega} a_{n}$ such that the subsequence $\{ x_{k}\}_{k\in B}$ converges to $x$. Finally, conclude that $B \subseteq^{*} A_{x \restriction n}$ for all $n\in\omega$. Therefore $\mathcal{I}$ is $\textrm{SPI}$.

    \medskip
    
    $[\text{(b}) \Longrightarrow \text{(a)}].$ Suppose that $\mathcal I$ is $\textrm{SPI}$. Let $A\in\mathcal{I}^{+}$ and consider a sequence $\{x_{n}\}_{n\in A} \subseteq 2^{\omega}$. For each $s\in 2^{<\omega}$ let $A_{s}=\{ n\in A : s \sqsubset x_{n} \}$, so that $\{ A_{s} : s\in 2^{<\omega} \} \subseteq \wp(A)$ satisfies the conditions of Definition \ref{def_system_1}. Now, take $B \in \mathcal{I}^{+} \!\restriction\! A$ and $\alpha \in 2^{\omega}$ such that $B \subseteq^{*} A_{\alpha \restriction n}$ for all $n\in\omega$. Finally, conclude that the subsequence $\{ x_{n} \}_{n\in B}$ converges to $\alpha$. Therefore $\mathcal I$ is $\textrm{h-FinBW}$.
\end{proof}    

\smallskip

We present the following result as an interesting application of the previous lemma, where systems with pseudo-intersections are used as a tool to establish the existence of simultaneously ideal convergent subsequences for a countable family of sequences in the Cantor space $2^{\omega}$.

\begin{lemma} \label{h-FinBW Lemma}
    Let $\mathcal I$ be an ideal on $\omega$. Then, the following statements are equivalent:
\setlist{nolistsep}
\begin{itemize}
\setlength{\itemsep}{0pt} 
\item[(a)] $\mathcal I$ is $\textrm{h-FinBW}$.
\item[(b)] For every $A\in\mathcal{I}^{+}$ and every countable collection $\{ \{x^{m}_{k}\}_{k\in A} \}_{m\in\omega}$  of sequences in $2^{\omega}$ indexed by $A$, there exists $B\in \mathcal{I}^{+} \!\restriction\! A$ such that for each $m\in\omega$ the subsequence $\{x^{m}_{k}\}_{k\in B}$ is convergent.   
\end{itemize}
\end{lemma}

\begin{proof}
    $[\text{(b)} \Longrightarrow \text{(a)}].$ Straightforward.

    \medskip
    
    $[\text{(a)} \Longrightarrow \text{(b)}].$ Suppose that $\mathcal{I}$ is $\textrm{h-FinBW}$, so it is $\textrm{SPI}$ according to Lemma \ref{SPI}. Let $A\in\mathcal{I}^{+}$ and consider a countable collection $\{ \{x^{m}_{k}\}_{k\in A} \}_{m\in\omega}$  of sequences in $2^{\omega}$ indexed by $A$. We recursively construct a system $\{ A_{s} : s\in \omega^{<\omega} \} \subseteq \wp(A)$ of subsets of $A$ as follows:

    \medskip

     Let $\{a_{n}\}_{n\in\omega} \subseteq \omega$ be the sequence defined by $a_{n}=2^{n+1}$, where each $a_{n}$ represents the set of all sequences of length $n+1$ consisting of zeros and ones. Now, put $A_{\emptyset} = A$, and for each $\ell \in \omega$ and $s\in \prod_{n<\ell} 2^{n+1}$, where $s = \{ s_{n} \}_{n<\ell}$ is a sequence of length $\ell$ with each $s_{n}\in 2^{n+1}$, suppose inductively that the set $A_{s}$ has been constructed. Then, for every $\gamma \in 2^{\ell+1}$ we define the set $A_{s^{\smallfrown}\gamma}$ by
    
    \kern-0.5em
    \begin{equation*}
       A_{s^{\smallfrown}\gamma} = \{ k\in A_{s} : ( \forall\, m \leq \ell ) ( {\{s_{i}(m) \}_{m\leq i <\ell}}^{\smallfrown} \gamma(m) \sqsubset x^{m}_{k}) \}. 
    \end{equation*}

    \kern-0.5em
    In this way, we obtain that $A_{s} = \bigcup_{\gamma \in a_{|s|}} A_{s^{\smallfrown}\gamma}$ and $A_{s^{\smallfrown}\gamma} \cap A_{s^{\smallfrown}\beta} = \emptyset$ if $\gamma,\beta\in a_{|s|}$ with $\gamma \neq \beta$.

    \medskip
    
    Since the ideal $\mathcal{I}$ is $\textrm{SPI}$, there exist $B\in \mathcal{I}^{+} \!\restriction\! A$ and $\alpha \in \prod_{n\in\omega} 2^{n+1}$, where $\alpha = \{\alpha_{n}\}_{n\in\omega}$ and each $\alpha_{n} \in 2^{n+1}$, such that $B$ is a pseudo-intersection of $\{ A_{\alpha\restriction n}\}_{n\in \omega}$, then $B \subseteq^{*} A_{\alpha \restriction n}$ for all $n\in \omega$.

    \medskip

    Finally, we conclude that for each $m\in \omega$ the subsequence $\{ x^{m}_{k} \}_{k\in B}$ converges to $\alpha^{m} \in 2^{\omega}$, where $\alpha^{m} = \{ \alpha_{i}(m) \}_{m \leq i}$. Indeed, fix $m\in\omega$ and note that for each $n > m$ there is $j_{n}\in\omega$ such that $B/j_{m} \subseteq A_{\alpha \restriction n}$, thus for every $k\in B/j_{m}$ we have $k\in A_{\alpha \restriction n}$, then $\{ \alpha_{i} (m)\}_{m\leq i\leq n} \sqsubset x^{m}_{k}$ and hence $\alpha^{m} \!\restriction\!(n-m) = x^{m}_{k} \!\restriction\!(n-m)$; therefore, we deduce that $\alpha^{m}$ is the limit of the subsequence $\{ x^{m}_{k} \}_{k\in B}$.
\end{proof}

\smallskip

We remark that the previous fact remains valid when the sequences are taken in an arbitrary uncountable compact metrizable space $\mathcal{X}$, rather than only in $2^{\omega}$, as will be established in a more general setting in Proposition \ref{h-FinBW_n Lemma}.

\medskip

We continue by presenting a useful combinatorial characterization of Ramsey ideals in terms of systems with diagonalizations, thereby generalizing \cite[Theorem 3.16, (1) $\Leftrightarrow$ (4)]{FMRS11}. For the sake of completeness, we include a proof.

\begin{proposition} \label{SD}
    Let $\mathcal I$ be an ideal on $\omega$. Then, the following statements are equivalent:
\setlist{nolistsep}
\begin{itemize}
\setlength{\itemsep}{0pt}
\item[(a)] $\mathcal I$ is Ramsey.
\item[(b)] $\mathcal I$ is $\textrm{SD}$.
\end{itemize}
\end{proposition}

\begin{proof}    
    $[\text{(a}) \Longrightarrow \text{(b)}].$ Suppose that $\mathcal{I}$ is Ramsey, so it is both $q^{+}$ and $\textrm{SPI}$, according to Proposition \ref{Ramsey = (q^{+})+(h-FinBW)} and Lemma \ref{SPI}. Let $\{ A_{s} : s\in \omega^{<\omega} \} \subseteq \wp(A)$ be a system of subsets of $A\in \mathcal{I}^{+}$, and let $\{ a_{n} \}_{n\in\omega} \subseteq \omega$ be a sequence such that for every $s\in \omega^{<\omega}$ both $A_s=\bigcup_{i<a_{|s|}} A_{s^{\smallfrown} i}$ and $A_{s^{\smallfrown} i} \cap A_{s^{\smallfrown} j}=\emptyset$ if $i< j < a_{|s|}$. 

    \medskip
    
    Since the ideal $\mathcal{I}$ is $\textrm{SPI}$, there exist $B\in \mathcal{I}^{+} \!\restriction\! A$ and $x \in \prod_{n\in\omega} a_{n}$ such that $B \subseteq^{*} A_{x\restriction n}$ for all $n\in \omega$. Taking this into account, consider a strictly increasing sequence $\{n_{k}\}_{k\in\omega} \subseteq \omega$ such that $n_{0}=0$ and $n_{k+1} > \max( (B\setminus A_{x \restriction n_{k}})\cup\{n_{k}\} )$ for each $n\in \omega$.
    
    \medskip
    
    Consider now the infinite partition $B= \bigcup_{k\in\omega} f_{k}$, where each finite piece $f_{k}$ is defined by $f_{k} = B \cap \{ n_{k}, \ldots, n_{k+1}-1 \}$. Since the ideal $\mathcal{I}$ is $q^{+}$, there exists $D \in \mathcal{I}^{+} \!\restriction\! B \subseteq \mathcal{I}^{+} \!\restriction\! A$ such that $|D\cap f_{k}|\leq 1$ for all $k\in\omega$. Finally, write $D= D_{0}\cup D_{1}$ as a disjoint union, where $D_{i} = \{d\in D : (\exists\, j\in\omega)(d\in f_{2j+i}) \}$ for $i\in 2$, then at least one of these sets belongs to $\mathcal{I}^{+}$; however, in either case $D_{i}$ is a diagonalization of $\{ A_{x \restriction n} \}_{n\in\omega}$.
        
    \medskip
    
    Indeed, for every $m \in D_{i}$ there is a unique $k\in\omega$ such that $m\in f_{k}$; thus, it holds that

    \kern-0.5em
    \begin{equation*}
        D_{i} /m = D_{i}/(n_{k+2}-1) \subseteq B/(n_{k+2}-1) \subseteq A_{x\restriction n_{k+1}} \subseteq A_{x\restriction(m+1)}.
    \end{equation*}

    \kern-0.5em    
    Consequently, we deduce that $D_{i} /m \subseteq A_{x \restriction (m+1)}$ for all $m\in D_{i}$. Therefore, we conclude that $\mathcal{I}$ is $\textrm{SD}$.

    \medskip
    
    $[\text{(b}) \Longrightarrow \text{(a)}].$ Suppose that $\mathcal I$ is $\textrm{SD}$. Let $A\in \mathcal{I}^{+}$ and consider any coloring $\pi: [A]^{2} \rightarrow 2$. Recursively construct a collection $\{ A_{s} : s\in 2^{<\omega} \} \subseteq \wp(A)$, with $A_{\emptyset}= A$, such that $A_{s} = A_{s^{\smallfrown}0} \cup A_{s^{\smallfrown}1}$ and $A_{s^{\smallfrown}0} \cap A_{s^{\smallfrown}1} = \emptyset$ for all $s\in 2^{<\omega}$, as follows: 

    \medskip

    For $n\in \omega$ and $s\in 2^{<\omega}$ such that $|s|=n$, assume that $A_{s}$ has already been constructed. Then, if $n\notin A$, set $A_{s^{\smallfrown}0}= A_{s}$ and $A_{s^{\smallfrown}1}= \emptyset$. Otherwise, if $n\in A$, define $A_{s^{\smallfrown}i} = \{ m\in A_{s} : \pi(\{n,m\})=i \}$ for each $i\in 2$; additionally, join $\{n\}$ with any $A_{s^{\smallfrown}i}$ as long as $n\in A_{s}$. 

    \medskip

    Since the ideal $\mathcal{I}$ is $\textrm{SD}$, there exist $H\in \mathcal{I}^{+} \!\restriction\! A$ and $x \in 2^{\omega}$ such that $H/m \subseteq A_{x \restriction (m+1)}$ for all $m\in H$. Next, split $H= H_{0}\cup H_{1}$ into two disjoint sets, where $H_{i} = \{m\in H : x(m)=i \}$ for $i\in 2$, then at least one of these sets belongs to $\mathcal{I}^{+}$; nevertheless, in either case $H_{i}$ is homogeneous for $\pi$. Indeed, for every $m,m^{\prime} \in H_{i}$ with $m<m^{\prime}$ we have $m^{\prime}\in H_{i}/m$ and hence $m^{\prime}\in A_{x\restriction(m+1)}$, thus $\pi(\{m,m^{\prime}\}) = x(m)=i$, which implies that $\pi \,\text{''}\, [H_{i}]^{2} = \{i\}$. Therefore, we conclude that $\mathcal{I}$ is Ramsey.
\end{proof}

\smallskip

Finally, we apply systems with diagonalizations to obtain the following result, which answers the question raised in \cite[Last Remark of Section 4]{FMRS11} and \cite[Question 5.6]{HMTU2017}. Specifically, we prove that an ideal $\mathcal{I}$ is Ramsey if and only if it satisfies $\mathcal{I}^+ \longrightarrow (\mathcal{I}^+)^{n}_{k}$ for all $n,k \geq 2$.

\begin{theorem} \label{Ramsey_ideal_theorem}
    Let $\mathcal I$ be an ideal on $\omega$. Then, the following statements are equivalent:
\setlist{nolistsep}
\begin{itemize}
\setlength{\itemsep}{0pt}
\item[(a)] $\mathcal I$ is Ramsey.
\item[(b)] $\mathcal{I}^+ \longrightarrow (\mathcal{I}^+)^{n}_{k}$ holds for all $n,k\geq 2$.
\end{itemize}
\end{theorem}

\begin{proof}
    $[\text{(b)} \Longrightarrow \text{(a)}].$ Straightforward.

    \medskip
    
    $[\text{(a)} \Longrightarrow \text{(b)}].$ Suppose that $\mathcal{I}$ is Ramsey, so it is $\textrm{SD}$ according to proposition \ref{SD}. We are going to prove that $\mathcal{I}^+ \longrightarrow (\mathcal{I}^+)^{n}_{2}$ holds for all $n\geq 2$ by induction on $n$, and obviously the result is true for the base case $n=2$.

    \medskip

    Inductively, suppose that $\mathcal{I}^+ \longrightarrow (\mathcal{I}^+)^{n}_{2}$ holds for a fixed $n \geq 2$, and consider a coloring $\pi: [A]^{n+1} \rightarrow 2$ with $A\in \mathcal{I}^{+}$. We recursively construct a system $\{ A_{s} : s\in \omega^{<\omega} \} \subseteq \wp(A)$ of subsets of $A$ as follows: 

    \medskip

    Let $\{a_{j}\}_{j\in\omega} \subseteq \omega$ be the sequence defined by $a_{j}=2^{\binom{n+j}{n}}$, where each $a_{j}$ is interpreted as the set of all functions from $[r_{n+j}(A)]^{n}$ to $2$, with $r_{n+j}(A)$ denoting the initial segment of $A$ of size $n+j$ in its increasing enumeration. Now, put $A_{\emptyset} = A$, and for each $i\in 2$, define the set $A_{i} = \{k \in A \setminus r_{n}(A) : \pi(r_{n}(A) \cup \{k\} ) =i \}$, and later include the elements of $r_{n}(A)$ in any of these sets.
    
    \medskip
    
    Continuing with the process, for each $\ell \in \omega$ and $s\in \prod_{j<\ell} 2^{\binom{n+j}{n}}$ with $|s|=\ell$, suppose inductively that the set $A_{s}$ has been constructed. We partition $A_{s} \setminus r_{n+\ell}(A)$ into finitely many pieces using the equivalence relation $\sim_{s}$ defined as follows: for $k,k^{\prime} \in A_{s} \setminus r_{n+\ell}(A)$, we have $k \sim_{s} k^{\prime}$ if and only if $\pi(t\cup\{k\}) = \pi(t\cup\{k^{\prime}\})$ for every $t\in [r_{n+\ell}(A)]^{n}$. Thus, there are at most $2^{\binom{n+\ell}{n}}$ equivalence classes in $A_{s} \setminus r_{n+\ell}(A)$ under $\sim_{s}$. In other words, if $\{ f^{s}_{i} : i< 2^{\binom{n+\ell}{n}} \}$ is a finite list of all functions from $[r_{n+\ell}(A)]^{n}$ to $2$, then each of the above equivalence classes is determined by exactly one of these functions. Taking this into account, for every $i< 2^{\binom{n+\ell}{n}}$ we define the set $A_{s^{\smallfrown}i}$ by
    
    \kern-0.5em
    \begin{equation*}
       A_{s^{\smallfrown}i} = \{ k\in A_{s} \setminus r_{n+\ell}(A) : ( \forall\, t\in [r_{n+\ell}(A)]^{n} ) ( \pi(t\cup \{k\}) = f^{s}_{i}(t) ) \}. 
    \end{equation*}

    \kern-0.5em
    In addition, for each $k\in r_{n+\ell}(A)$, we join $\{k\}$ with any $A_{s^{\smallfrown}i}$ whenever $k\in A_{s}$. In this way, we obtain that $A_{s} = \bigcup_{i<a_{|s|}} A_{s^{\smallfrown}i}$ and $A_{s^{\smallfrown}i} \cap A_{s^{\smallfrown}j} = \emptyset$ if $i<j<a_{|s|}$.

    \medskip

     By virtue of the fact that the ideal $\mathcal{I}$ is $\textrm{SD}$, there exist $E\in \mathcal{I}^{+} \!\restriction\! A$ and $x \in \prod_{j\in\omega} 2^{\binom{n+j}{n}}$ such that $E$ is a diagonalization of $\{ A_{x\restriction j}\}_{j\in \omega}$, meaning that $E/m \subseteq A_{x\restriction (m+1)}$ for all $m\in E$. Now, define the coloring $\psi: [E]^{2} \rightarrow 2$ as follows: for $m,m^{\prime}\in E$ with $m<m^{\prime}$, let $\psi(\{m,m^{\prime}\}) = 0$ if and only if $m^{\prime} \in r_{n+m}(A)$. Since the ideal $\mathcal{I}$ is Ramsey, there exists $D\in \mathcal{I}^{+} \!\restriction\! E \subseteq \mathcal{I}^{+} \!\restriction\! A$ such that $\psi$ is constant on $[D]^{2}$, and necessarily $\psi \,\text{''}\, [D]^{2} = \{1\}$. As a result, we have both $D/m \cap r_{n+m}(A) = \emptyset$ and $D/m \subseteq A_{x\restriction (m+1)}$ for all $m\in D$.  
    \medskip

    Next, given any $t\in [D]^{n}$, if $m_{t}=\max t$ then $t\in [r_{n+m_{t}}(A)]^{n}$, and for each $k\in D/m_{t}$ we have that $k \in A_{x\restriction (m_{t}+1)} \setminus r_{n+m_{t}}(A)$; thus, we deduce that $\pi (t\cup \{k\}) = f^{x \restriction m_{t}}_{x(m_{t})} (t)$ whenever $k\in D/m_{t}$, which means that the coloring $\pi$ is constant on the set $\{ t\cup \{k\} : k\in D/m_{t} \}$.

    \medskip
    
    Bearing this in mind, we define the coloring $\phi : [D]^{n} \rightarrow 2$ by $\phi(t) = f^{x \restriction m_{t}}_{x(m_{t})} (t)$ for every $t\in [D]^{n}$. By inductive hypothesis, $\mathcal{I}^+ \longrightarrow (\mathcal{I}^+)^{n}_{2}$ holds, so there exists $H \in \mathcal{I}^{+} \!\restriction\! D \subseteq \mathcal{I}^{+} \!\restriction\! A$ such that $\phi$ takes a constant value on $[H]^{n}$.
    
    \medskip

     Finally, we claim that $H$ is also a homogeneous set for $\pi$. Indeed, if $h\in [H]^{n+1}$ then $h=t\cup \{k\}$ for some $t\in [H]^{n}$ and $k\in H/ m_{t}$, consequently $\pi(h) = \pi (t\cup\{k\}) = f^{x \restriction m_{t}}_{x(m_{t})} (t) = \phi(t)$, and thus $\pi$ takes a constant value on $[H]^{n+1}$. Due to this fact, we deduce that $\mathcal{I}^+ \longrightarrow (\mathcal{I}^+)^{n+1}_{2}$ also holds. 
    
    \medskip
    
    Therefore, we conclude that $\mathcal{I}^+ \longrightarrow (\mathcal{I}^+)^{n}_{2}$ holds for all $n\geq 2$, and hence, by Fact \ref{fact_colors}, we infer that $\mathcal{I}^+ \longrightarrow (\mathcal{I}^+)^{n}_{k}$ holds for all $n,k\geq 2$.
\end{proof}

\subsection{Almost homogeneous sets for countably-many colorings}

Let $f\in\omega^{\omega}$ be such that $f\geq 2$, and for each $m\in \omega$, consider a coloring $\pi_{m} : [\omega]^{f(m)} \rightarrow 2$. A known fact is that, in general, this arbitrary family of colorings $\{ \pi_{m} : [\omega]^{f(m)} \rightarrow 2 \}_{m\in\omega}$ does not necessarily admit an infinite set that is simultaneously homogeneous for each $\pi_{m}$. For instance, for each $m\in\omega$, define $\pi_{m} : [\omega]^{f(m)} \rightarrow 2$ as follows: for $s\in [\omega]^{f(m)}$, let $\pi_{m}(s) = 0$ if and only if $m\in s$. It is easy to check that there is no infinite set that is simultaneously homogeneous for each $\pi_{m}$.

\medskip

Despite this fact, in \cite[Section 2]{GrebikHrusak}, an interesting characterization of $F_{\sigma}$ ideals is given in terms of families of sets that are simultaneously homogeneous for suitable countable collections of colorings. Furthermore, it should be noted that for any countable collection of colorings $\{ \pi_{m} : [A]^{f(m)} \rightarrow 2 \}_{m\in\omega}$, one can always find an infinite set that is almost homogeneous for each $\pi_{m}$.

\begin{definition} \label{def_almosthomog}
    Given $n,k\in \omega$ with $n,k \geq 2$, let $A$ be a countably infinite set, and let $\pi:[A]^n \rightarrow k$ be a finite coloring of $[A]^n$. An infinite subset $H$ of $A$ is almost homogeneous for $\pi$ if there is $m\in\omega$ such that $H/m$ is homogeneous for $\pi$.
\end{definition}

We now present a characterization of Ramsey ideals in terms of positive almost homogeneous sets associated with countable sequences of colorings of pairs from a given positive set.

\begin{proposition} \label{countable_colorings_of_pairs}
    Let $\mathcal I$ be an ideal on $\omega$. Then, the following statements are equivalent:
\setlist{nolistsep}
\begin{itemize}
\setlength{\itemsep}{0pt}
\item[(a)] $\mathcal I$ is Ramsey.
\item[(b)] For every $A\in \mathcal{I}^{+}$ together with every countable collection $\{ \pi_{m}: [A]^{2} \rightarrow 2 \}_{m\in\omega}$ of colorings of $[A]^{2}$, there exists $H \in \mathcal{I}^{+} \!\restriction\! A$ that is almost homogeneous for each $\pi_{m}$. 
\end{itemize}
\end{proposition}

\begin{proof}
    $[\text{(b)} \Longrightarrow \text{(a)}].$ Straightforward.

    \medskip
    
    $[\text{(a)} \Longrightarrow \text{(b)}].$ Suppose that $\mathcal{I}$ is a Ramsey ideal. Let $A\in \mathcal{I}^{+}$, and for every $m\in \omega$, let $\pi_{m}: [A]^{2} \rightarrow 2$ be a coloring of $[A]^{2}$. We recursively construct a system $\{ A_{s} : s\in \omega^{<\omega} \} \subseteq \wp(A)$ of subsets of $A$ as follows:

    \medskip
    
    Let $\{a_{j}\}_{j\in\omega} \subseteq \omega$ be the sequence defined by $a_{j}=2^{j+1}$, where each $a_{j}$ represents the set of all sequences of length $j+1$ consisting of zeros and ones. Now, put $A_{\emptyset} = A$, and for each $\ell \in {\omega}$ and $s\in \prod_{j<\ell} 2^{j+1}$ with $|s|=\ell$, suppose inductively that the set $A_{s}$ has been constructed. 

    \medskip
    
    When $\ell \notin A$, set $A_{s^{\smallfrown}\gamma_{0}} = A_{s}$ for the null sequence $\gamma_{0}$ in $2^{\ell+1}$, and take $A_{s^{\smallfrown}\gamma} = \emptyset$ for all other sequences $\gamma \neq \gamma_{0}$ in $2^{\ell+1}$. Otherwise, if $\ell \in A$, then for every $\gamma \in 2^{\ell+1}$ we define the set $A_{s^{\smallfrown}\gamma}$ by

    \kern-0.5em
    \begin{equation*}
       A_{s^{\smallfrown}\gamma} = \{ k\in A_{s} / \ell : ( \forall\, m\leq\ell ) ( \pi_{m}(\{\ell,k\}) = \gamma(m) ) \}. 
    \end{equation*}

    \kern-0.5em
     In this last case, we include each $m\leq \ell$ that belongs to $A$ in any $A_{s^{\smallfrown}\gamma}$ whenever $m\in A_{s}$. In this way, we obtain that $A_{s} = \bigcup_{\gamma \in a_{|s|}} A_{s^{\smallfrown}\gamma}$ and $A_{s^{\smallfrown}\gamma} \cap A_{s^{\smallfrown}\beta} = \emptyset$ if $\gamma,\beta\in a_{|s|}$ with $\gamma \neq \beta$.

    \medskip

    Since the ideal $\mathcal{I}$ is Ramsey, Proposition \ref{SD} ensures the existence of $D \in \mathcal{I}^{+} \!\restriction\! A$ and $\alpha \in \prod_{j\in\omega} 2^{j+1}$, where $\alpha = \{ \alpha_{j} \}_{j\in \omega}$ and each $\alpha_{j} \in 2^{j+1}$, such that $D$ is a diagonalization of $\{ A_{\alpha \restriction n} \}_{n\in\omega}$, which means that $D/n \subseteq A_{\alpha \restriction (n+1)}$ for all $n\in D$. Thus, for any $n,k\in D$ such that $k \in D/n$, we have $k \in A_{\alpha \restriction (n+1)}$ and hence $\pi_{m} (\{n,k\}) = \alpha_{n} (m)$ for each $m\leq n$. Consequently, for every $n\in D$ and every $m\leq n$, the coloring $\pi_{m}$ takes the constant value $\alpha_{n}(m)$ on the set of pairs $\{ \{n,k\} : k\in D/n \}$.
        
    \medskip

    Next, we recursively construct a new system $\{ D_{s} : s \in 2^{<\omega} \} \subseteq \wp(D)$ of subsets of $D$, with $D_{\emptyset} = D$, as follows: for every $m \in \omega$ and $s \in 2^{<\omega}$ such that $|s| = m$, if the set $D_{s}$ has already been constructed, then for each $i \in 2$ we define the set $D_{s^{\smallfrown}i}$ by

    \kern-0.5em
    \begin{equation*}
       D_{s^{\smallfrown}i} = \{ n\in D_{s} / m : \alpha_{n}(m) = i \}. 
    \end{equation*}

    \kern-0.5em
    As expected, we include each $n\leq m$ that belongs to $D$ in any $D_{s^{\smallfrown}i}$ as long as $n\in D_{s}$. In this way, we get that $D_{s} = D_{s^{\smallfrown}0} \cup D_{s^{\smallfrown}1}$ and $D_{s^{\smallfrown}0} \cap D_{s^{\smallfrown}1} = \emptyset$.

    \medskip

    Since the ideal $\mathcal{I}$ is Ramsey, by Proposition \ref{SD} we deduce that there exist $H \in \mathcal{I}^{+} \!\restriction\! D \subseteq \mathcal{I}^{+} \!\restriction\! A$ and $\xi\in 2^{\omega}$ such that $H$ is a diagonalization of $\{ D_{\xi \restriction h} \}_{h\in\omega}$, so $H/h \subseteq D_{\xi \restriction (h+1)}$ for all $h\in H$.

    \medskip

    Finally, we claim that $H$ is an almost homogeneous set for each coloring $\pi_{m}$. Indeed, given $m\in \omega$, take any $h_{m}\in H$ such that $m\leq h_{m}$; then, for every $n,k \in H/ h_{m}$ with $n<k$, we have $n\in D_{\xi \restriction (h_{m}+1)} \subseteq D_{\xi \restriction (m+1)}$, and thus $\pi_{m} (\{n,k\}) = \alpha_{n}(m) =\xi(m)$. Therefore, we conclude that $\pi_{m} \,\text{''}\, [H/h_{m}]^{2} = \{\xi(m)\}$, showing that $H/h_{m}$ is homogeneous for $\pi_{m}$.
\end{proof}

\smallskip

The previous proposition admits a higher-dimensional extension obtained by considering sequences of finite colorings of $[\omega]^n$ for arbitrary $n \in \omega$ with $n\geq 2$.

\begin{theorem} \label{countable_colorings_of_n-sets}
    Let $\mathcal I$ be an ideal on $\omega$. Then, the following statements are equivalent:
\setlist{nolistsep}
\begin{itemize}
\setlength{\itemsep}{0pt}
\item[(a)] $\mathcal I$ is Ramsey.
\item[(b)] For all $n\in \omega$ with $n\geq 2$, and for every $A\in \mathcal{I}^{+}$ together with every countable collection $\{ \pi_{m}: [A]^{n} \rightarrow 2 \}_{m\in\omega}$ of colorings of $[A]^{n}$, there exists $H \in \mathcal{I}^{+} \!\restriction\! A$ that is almost homogeneous for each $\pi_{m}$. 
\item[(c)] For all $n\in \omega$ with $n\geq 2$, for any sequence $\{k_{m}\}_{m\in\omega} \subseteq \omega$ such that each $k_{m} \geq 2$, and for every $A\in \mathcal{I}^{+}$ together with every countable collection $\{ \pi_{m}: [A]^{n} \rightarrow k_{m} \}_{m\in\omega}$ of colorings of $[A]^{n}$, there exists $H \in \mathcal{I}^{+} \!\restriction\! A$ that is almost homogeneous for each $\pi_{m}$. 
\end{itemize}
\end{theorem}

\begin{proof}
    $[\text{(c)} \Longrightarrow \text{(b)} \Longrightarrow \text{(a)}].$ Straightforward.

    \medskip
    
    $[\text{(a)} \Longrightarrow \text{(b)}].$ Suppose that $\mathcal{I}$ is a Ramsey ideal. To prove this result, we proceed by induction on $n$, noting that the base case $n=2$ was established in Proposition \ref{countable_colorings_of_pairs}.

    \medskip

    Inductively, fix $n\geq 2$ and suppose that for every $A\in\mathcal{I}^{+}$, along with every countable collection of colorings of $[A]^{n}$ into two colors, there exists some $\mathcal{I}$-positive subset of $A$ that is almost homogeneous for each of these colorings. 

    \medskip
    
    As a starting point, let $A\in \mathcal{I}^{+}$ and let $\{ \pi_{m}: [A]^{n+1} \rightarrow 2\}_{m\in \omega}$ be a countable collection of colorings of $[A]^{n+1}$. We recursively construct a system $\{ A_{s} : s\in\omega^{<\omega} \} \subseteq \wp(A)$ of subsets of $A$ as follows:

    \medskip

    Let $\{a_{j}\}_{j\in\omega} \subseteq \omega$ be the sequence defined by $a_{j} = {2^{\binom{n+j}{n}}}^{j+1}$, where each $a_{j}$ is interpreted as the set of all sequences of length $j+1$ consisting of functions from $[r_{n+j}(A)]^{n}$ to $2$, with $r_{n+j}(A)$ denoting the initial segment of $A$ of size $n+j$ in its increasing enumeration. 

    \medskip
    
    Now, put $A_{\emptyset} = A$, and for each $\ell \in {\omega}$ and $s\in \prod_{j<\ell} {2^{\binom{n+j}{n}}}^{j+1}$ with $|s|=\ell$, suppose inductively that the set $A_{s}$ has been constructed. We split $A_{s} \setminus r_{n+\ell}(A)$ into finitely many pieces through the equivalence relation $\sim_{s}$ defined as follows: for $k,k^{\prime} \in A_{s} \setminus r_{n+\ell}(A)$, we have $k \sim_{s} k^{\prime}$ if and only if $\pi_{m}(t\cup\{k\}) = \pi_{m}(t\cup\{k^{\prime}\})$ for every $m\leq \ell$ and every $t\in [r_{n+\ell}(A)]^{n}$. Thus, there are at most $2^{\binom{n+\ell}{n}(\ell +1)}$ equivalence classes in $A_{s} \setminus r_{n+\ell}(A)$ under $\sim_{s}$. In other words, if $\{ f^{s}_{i} : i< 2^{\binom{n+\ell}{n}} \}$ is a finite list of all functions from $[r_{n+\ell}(A)]^{n}$ to $2$, then each of the above equivalence classes is determined by exactly one sequence of length $\ell+1$ formed by some of these functions.

    \medskip

    Based on this observation, for every $\gamma \in  {2^{\binom{n+\ell}{n}}}^{\ell+1}$ we define the set $A_{s^{\smallfrown}\gamma}$ by

    \kern-0.5em
    \begin{equation*}
       A_{s^{\smallfrown}\gamma} = \{ k\in A_{s} \setminus r_{n+\ell}(A) : ( \forall\, m\leq \ell ) ( \forall\, t\in [r_{n+\ell}(A)]^{n} ) ( \pi_{m}(t\cup \{k\}) = f^{s}_{\gamma(m)}(t) ) \}. 
    \end{equation*}

    \kern-0.5em
    Naturally, for each $k\in r_{n+\ell}(A)$, we join $\{k\}$ with any $A_{s^{\smallfrown}\gamma}$ whenever $k\in A_{s}$. In this way, we obtain that $A_{s} = \bigcup_{\gamma \in a_{|s|}} A_{s^{\smallfrown}\gamma}$ and $A_{s^{\smallfrown}\gamma} \cap A_{s^{\smallfrown}\beta} = \emptyset$ if $\gamma, \beta \in a_{|s|}$ with $\gamma \neq \beta$.

    \medskip
    
    Since the ideal $\mathcal{I}$ is Ramsey, Proposition \ref{SD} implies that there are $D \in \mathcal{I}^{+} \!\restriction\! A$ and $\alpha \in \prod_{j\in\omega} {2^{\binom{n+j}{n}}}^{j+1}$, where $\alpha = \{ \alpha_{j} \}_{j\in \omega}$ and each $\alpha_{j}$ is a sequence of length $j+1$ whose components are functions from $[r_{n+j}(A)]^{n}$ to $2$, such that $D$ is a diagonalization of $\{ A_{\alpha \restriction j} \}_{j\in\omega}$, that is, $D/d \subseteq A_{\alpha \restriction (d+1)}$ for all $d\in D$. Furthermore, as in the proof of Theorem \ref{Ramsey_ideal_theorem}, we may obtain this set $D$ satisfying $D/d \cap r_{n+d}(A) = \emptyset$ for all $d\in D$.
    
    \medskip

    Given any $m\in \omega$ and $t\in [D]^{n}$ such that $m \leq m_{t}$, where $m_{t}=\max t$, we get $t\in [r_{n+m_{t}}(A)]^{n}$, and moreover, for each $k\in D/m_{t}$ it follows that $k \in A_{\alpha\restriction (m_{t}+1)} \setminus r_{n+m_{t}}(A)$; accordingly, we infer that $\pi_{m} (t\cup \{k\}) = f^{\alpha \restriction m_{t}}_{\alpha_{m_{t}} (m)} (t)$ whenever $k\in D/m_{t}$, which implies that the coloring $\pi_{m}$ is constant on the set $\{ t\cup \{k\} : k\in D/m_{t}\}$. 
    
    \medskip
    
    Taking this into account, for each $m \in\omega$ we define the coloring $\phi_{m} : [D]^{n} \rightarrow 2$ by setting $\phi_{m}(t) = f^{\alpha \restriction m_{t}}_{\alpha_{m_{t}} (m)} (t)$ for every $t\in [D]^{n}$ such that $m\leq m_{t}$, and assigning values arbitrarily otherwise. By inductive hypothesis, there exists $H \in \mathcal{I}^{+} \!\restriction\! D \subseteq \mathcal{I}^{+} \!\restriction\! A$ that is almost homogeneous for each $\phi_{m}$; in consequence, for every $m\in\omega$ there is $h_{m} \in \omega$ such that $H/h_{m}$ is homogeneous for $\phi_{m}$, and we can assume that $m\leq h_{m}$ for each $m\in\omega$. 

    \medskip    

    Finally, we assert that $H$ is also an almost homogeneous set for each $\pi_{m}$. Indeed, given any $m\in\omega$, if $h\in [H/h_{m}]^{n+1}$ then $h= t\cup\{k\}$ for some $t\in [H/h_{m}]^{n}$ and $k\in H/ m_{t}$, thus $\pi_{m}(h) = \pi_{m} (t\cup\{k\}) = f^{\alpha \restriction m_{t}}_{\alpha_{m_{t}} (m)} (t) = \phi_{m}(t)$, and hence $H/h_{m}$ is also homogeneous for $\pi_{m}$.

    \medskip
    
    $[\text{(b)} \Longrightarrow \text{(c)}].$ Assume that the ideal $\mathcal{I}$ satisfies the affirmation stated in $(b)$. Let $n\in\omega$ with $n \geq 2$, let $\{k_{m}\}_{m\in\omega} \subseteq \omega$ be a sequence such that each $k_{m} \geq 2$, and for every $m\in\omega$, let $\pi_{m} : [A]^{n} \rightarrow k_{m}$ be a coloring of $[A]^{n}$, where $A \in \mathcal{I}^{+}$. Next, for each $m\in\omega$ and each $i<k_{m}$, consider the coloring $\pi_{m}^{i} : [A]^{n} \rightarrow 2$ defined as follows: for $a\in [A]^{n}$, let $\pi_{m}^{i} (a) =0$ if and only if $\pi_{m} (a) \neq i$. 

    \medskip

    Let $H\in \mathcal{I}^{+} \!\restriction\! A$ be such that, for all $m\in\omega$ and $i<k_{m}$, the set $H$ is almost homogeneous for $\pi_{m}^{i}$. Therefore, for every $m\in \omega$, there exists $\ell_{m} \in\omega$ such that $H/\ell_{m}$ is homogeneous for each coloring that belongs to the finite collection $\{ \pi_{m}^{i}: [A]^{n} \rightarrow 2 \}_{i<k_{m}}$. Now, pick any $a \in [H/\ell_{m}]^{n}$, and let $j<k_{m}$ be such that $\pi_{m} (a) =j$. Then, we have $\pi_{m}^{j} (a) =1$, and thus $\pi_{m}^{j} \,\text{''}\, [H/\ell_{m}]^{n} = \{1\}$, which implies that $\pi_{m} \,\text{''}\, [H/\ell_{m}]^{n} = \{j\}$. Consequently, we conclude that $H$ is almost homogeneous for each $\pi_{m}$. 
\end{proof}

\smallskip

The previous theorem naturally leads to the following question, which is currently open, asking whether the simultaneous almost homogeneity property remains valid when the dimensions of the colorings are allowed to vary.

\begin{question} \label{question_almost_homogeneous}
Let $\mathcal{I}$ be a Ramsey ideal, and let $\{n_{m}\}_{m\in\omega} \subseteq \omega$ be any sequence such that each $n_{m} \geq 2$. If for every $m\in\omega$ we take any coloring $\pi_m : [\omega]^{n_{m}} \rightarrow 2$, is there an $H\in \mathcal{I}^{+}$ such that $H$ is an almost homogeneous set for each $\pi_m$?
\end{question}

\subsection{Some partition properties}

We will review now some high-dimensional partition properties related to ideals, as well as their connection with the ideal Ramsey property, complementing the study of the lowest-dimensional partition notions for ideals initiated in \cite[Subsection 4.5]{HMTU2017} and \cite[Section 2.5]{Meza}.

\begin{definition}
Let $\mathcal{I}$ be an ideal on $\omega$ and fix $n\in \omega$ with $n \geq 2$. The ideal $\mathcal I$ satisfies:

\setlist{nolistsep}
\begin{itemize}
\setlength{\itemsep}{0pt} 

\item $\mathcal{I}^+ \longrightarrow (\omega, \mathcal{I}^+)^{n}_{2}$ if for every $A\in \mathcal{I}^+$ and every coloring $\pi:[A]^{n} \rightarrow 2$, either there exists an infinite set $B \in [A]^{\omega}$ such that $\pi \,\text{''}\, [B]^{n} = \{0\}$, or there exists an $\mathcal{I}$-positive set $B \in \mathcal{I}^{+} \!\restriction\! A$ such that $\pi \,\text{''}\, [B]^{n} = \{1\}$. 

\item $\mathcal{I}^+ \longrightarrow (<\omega, \mathcal{I}^+)^{n}_{2}$ if for every $A\in \mathcal{I}^+$ and every coloring $\pi:[A]^{n} \rightarrow 2$, either for each $n<m<\omega$ there exists a finite set $B \in [A]^{m}$ of size $m$ such that $\pi \,\text{''}\, [B]^{n} = \{0\}$, or there exists an $\mathcal{I}$-positive set $B \in \mathcal{I}^{+} \!\restriction\! A$ such that $\pi \,\text{''}\, [B]^{n} = \{1\}$.

\item $\mathcal{I}^+ \longrightarrow (m, \mathcal{I}^+)^{n}_{2}$, where $n<m<\omega$, if for every $A\in \mathcal{I}^+$ and every coloring $\pi:[A]^{n} \rightarrow 2$, either there exists a finite set $B \in [A]^{m}$ of size $m$ such that $\pi \,\text{''}\, [B]^{n} = \{0\}$, or there exists an $\mathcal{I}$-positive set $B \in \mathcal{I}^{+} \!\restriction\! A$ such that $\pi \,\text{''}\, [B]^{n} = \{1\}$.

\end{itemize}
\end{definition}

Clearly, for any ideal $\mathcal{I}$ on $\omega$ and every $n,m\in \omega$ with $m>n\geq2$, the following implications are straightforward:

\kern-0.5em
\begin{equation*}
   \mathcal{I}^+ \longrightarrow (\mathcal{I}^+)^{n}_{2} \;\;\;\Longrightarrow\;\;\; 
    \mathcal{I}^+ \longrightarrow (\omega, \mathcal{I}^+)^{n}_{2}
    \;\;\;\Longrightarrow\;\;\;
    \mathcal{I}^+ \longrightarrow (<\omega, \mathcal{I}^+)^{n}_{2}
    \;\;\;\Longrightarrow\;\;\; 
    \mathcal{I}^+ \longrightarrow (m, \mathcal{I}^+)^{n}_{2}
\end{equation*}

\kern-0.5em
Nevertheless, in \cite[Subsection 4.5]{HMTU2017} it was shown that, for the case $n=2$, these implications do not reverse, so that:

\kern-0.5em
\begin{equation*}
    \mathcal{I}^+ \longrightarrow (m, \mathcal{I}^+)^{2}_{2}
    \;\;\;\centernot\Longrightarrow\;\;\;
    \mathcal{I}^+ \longrightarrow (<\omega, \mathcal{I}^+)^{2}_{2}
    \;\;\;\centernot\Longrightarrow\;\;\;
    \mathcal{I}^+ \longrightarrow (\omega, \mathcal{I}^+)^{2}_{2}
    \;\;\;\centernot\Longrightarrow\;\;\;
    \mathcal{I}^+ \longrightarrow (\mathcal{I}^+)^{2}_{2}  
\end{equation*}

In \cite[Subsection 5.2]{HMTU2017}, the authors suggest analyzing these partition properties for Ramsey-type ideals by considering colorings of higher-dimensional $n$-tuples for $n \geq 3$, and determining whether they exhibit analogous behavior to the known case $n = 2$. We will answer this question in Theorem \ref{equiv_partition_properties}, by proving that for every $n \geq 3$ all these partition properties are equivalent to the ideal Ramsey property $\mathcal{I}^+ \longrightarrow (\mathcal{I}^+)^{2}_{2}$. To do this, we apply Theorem \ref{Ramsey_ideal_theorem}, along with the following proposition, established by Baumgartner and Taylor in \cite[Theorem 2.1, (iii) $\Rightarrow$ (i)]{BT} in the context of ultrafilters, whose argument carries over naturally to the setting of ideals. For the sake of completeness, we include its proof here.

\begin{proposition}[\cite{BT}] \label{Theorem_BT}
    Let $\mathcal{I}$ be an ideal on $\omega$. If $\mathcal{I}^+ \longrightarrow (4, \mathcal{I}^+)^{3}_{2}$ holds, then $\mathcal{I}$ is Ramsey.
\end{proposition}

\begin{proof}
    Suppose that the ideal $\mathcal{I}$ satisfies $\mathcal{I}^+ \longrightarrow (4, \mathcal{I}^+)^{3}_{2}$, and let $\pi: [A]^{2} \rightarrow 2$ be a coloring with $A \in\mathcal{I}^{+}$. Define $\phi: [A]^{3} \rightarrow 2$ as follows: for every $\{a,b,c\}\in [A]^{3}$ with $a<b<c$, let $\phi (\{a,b,c\}) = 0$ if and only if $\pi (\{a,b\})=0$ and $\pi (\{b,c\})=1$. Notice that $\phi \,\text{''}\, [F]^{3} \neq \{0\}$ for all finite set $F\in [A]^{4}$ of size $4$, thus there must exist an $\mathcal{I}$-positive set $B\in \mathcal{I}^{+} \!\restriction\! A$ such that $\phi \,\text{''}\, [B]^{3} = \{1\}$. Now, if $\pi(t)=1$ for each $t \in [B]^{2}$, then obviously $\pi \,\text{''}\, [B]^{2} = \{1\}$. Otherwise, if $\pi(t)=0$ for some $t\in [B]^{2}$, then it is easy to deduce that $\pi \,\text{''}\, [B/t]^{2} = \{0\}$. Therefore the ideal $\mathcal{I}$ is Ramsey.
\end{proof}

\begin{theorem} \label{equiv_partition_properties}
    Let $\mathcal I$ be an ideal on $\omega$. Then, the following statements are equivalent:
\setlist{nolistsep}
\begin{itemize}
\setlength{\itemsep}{0pt}
\item[(a)] $\mathcal I$ is Ramsey.
\item[(b)] $\mathcal{I}^+ \longrightarrow (\omega,\mathcal{I}^+)^{n}_{2}$ holds for all $n\geq 3$.
\item[(c)] $\mathcal{I}^+ \longrightarrow (<\omega,\mathcal{I}^+)^{n}_{2}$ holds for all $n\geq 3$.
\item[(d)] $\mathcal{I}^+ \longrightarrow (m,\mathcal{I}^+)^{n}_{2}$ holds for all $m>n\geq 3$.
\end{itemize}
\end{theorem}

\begin{proof}
    $[\text{(b)} \Longrightarrow \text{(c)} \Longrightarrow \text{(d)}].$ Straightforward.

    \medskip
    
    $[\text{(d)} \Longrightarrow \text{(a)}].$ Suppose that $\mathcal{I}^+ \longrightarrow (m,\mathcal{I}^+)^{n}_{2}$ holds for all $m>n\geq 3$. So, in particular $\mathcal{I}^+ \longrightarrow (4,\mathcal{I}^+)^{3}_{2}$ holds, and therefore the ideal $\mathcal{I}$ is Ramsey according to Proposition \ref{Theorem_BT}.

    \medskip
    
    $[\text{(a)} \Longrightarrow \text{(b)}].$ Suppose that $\mathcal{I}$ is Ramsey. Then, by virtue of Theorem \ref{Ramsey_ideal_theorem} we deduce that $\mathcal{I}^+ \longrightarrow (\mathcal{I}^+)^{n}_{2}$ holds for all $n\geq 3$, and hence we conclude that $\mathcal{I}^+ \longrightarrow (\omega,\mathcal{I}^+)^{n}_{2}$ also holds for all $n\geq 3$.
\end{proof}

\subsection{Some convergence properties}

We now turn our attention to some convergence properties associated with ideals, which extend to higher dimensions the notions of $\textrm{h-FinBW}$ and $\textrm{h-Mon}$ for ideals introduced in \cite{FMRS07, FMRS11}. 

\medskip

First, we will work with the notion of convergence related to strengthenings of sequential compactness, as researched in \cite{KuSze}. Let $\mathcal{X}$ be a topological space, and fix $A\in [\omega]^{\omega}$ and $n\in\omega$ with $n\geq 1$. Any function $f: [A]^{n} \rightarrow \mathcal{X}$ will be called a sequence of elements of $\mathcal{X}$, and for each $B\in [A]^{\omega}$, the restriction $f \!\restriction\! [B]^{n}$ will be called a subsequence of $f$. Now, following \cite[Section 1]{KuSze}, we say that a sequence $f: [A]^{n} \rightarrow \mathcal{X}$ converges to $p\in \mathcal{X}$, if for every open neighborhood $U$ of $p$ there exists $m\in\omega$ such that $f\,\text{''}\, [A/m]^{n} \subseteq U$. Note that the conventional notion of convergence of sequences $f:A \rightarrow \mathcal{X}$ corresponds to the case $n=1$. 

\medskip

Furthermore, we will consider a strict partial order $\triangleleft^{*}$ on $[\omega]^{<\omega}$ that was studied in \cite[Section II.2]{Todorcevic2005}. Let us first recall that the shift relation $\triangleleft$ on $[\omega]^{<\omega}$ is defined as follows: given $s,t \in [\omega]^{<\omega}$, let $s\triangleleft t$ if and only if $(s\setminus \{\min s\}) \sqsubseteq t$. Now, let $\triangleleft^{*}$ be the transitive closure of the shift relation $\triangleleft$ on $[\omega]^{<\omega}$, that is, for $s,t\in [\omega]^{<\omega}$, we have $s \triangleleft^{*} t$ if for some $k\in \omega$ with $k\geq 1$, there are $r_0, r_1, \ldots, r_k \in [\omega]^{<\omega}$ such that $s = r_0 \triangleleft r_1 \triangleleft \dots \triangleleft r_k = t$. 

\medskip

Here, we will use the shift relation $\triangleleft$ and its transitive closure $\triangleleft^{*}$ restricted to sets of the form $[A]^n$, where $A\in [\omega]^{\omega}$. Thus, given $s,t\in [A]^{n}$, we have $s\triangleleft t$ whenever $t$ is obtained by removing the first element of $s$ and adding a new element on top of the maximum of $s$. Additionally, note that for the case $n=1$, the strict order relation $\triangleleft^{*}$ on $[\omega]^{1}$ coincides with the usual order of $\omega$. So, a sequence of real numbers $f: [A]^{n} \rightarrow \mathbb{R}$ is said to be monotone if, for every $s,t \in [A]^{n}$, one of the following always holds: either $s \triangleleft^{*} t$ implies $f(s) \leq f(t)$, or $s \triangleleft^{*} t$ implies $f(s) \geq f(t)$.

\begin{definition}
    Let $\mathcal{I}$ be an ideal on $\omega$ and fix $n\in \omega$ with $n \geq 1$. We say that:
    \setlist{nolistsep}
    \begin{itemize}
    \setlength{\itemsep}{0pt} 
    \item $\mathcal{I}$ is $\textrm{h-Mon}_{n}$ if for every $A\in\mathcal{I}^{+}$ and every sequence of real numbers $f:[A]^{n} \rightarrow \mathbb{R}$, there exists some $B\in \mathcal{I}^{+} \!\restriction\! A$ such that the subsequence $f \!\restriction\! [B]^{n}$ is monotone. 
    \item $\mathcal{I}$ is $\textrm{h-FinBW}_{n}$ if for every $A\in\mathcal{I}^{+}$ and every bounded sequence of reals numbers $f: [A]^{n} \rightarrow \mathbb{R}$, there exists some $B\in \mathcal{I}^{+} \!\restriction\! A$ such that the subsequence $f \!\restriction\! [B]^{n}$ is convergent.
    \end{itemize}
\end{definition}

Keeping in mind the previously introduced notions for ideals, it is clear that $\textrm{h-Mon}_{1}$ and $\textrm{h-Mon}$ coincide; similarly, $\textrm{h-FinBW}_{1}$ and $\textrm{h-FinBW}$ also coincide. Moreover, for every $m,n \in\omega$ with $m\geq n \geq 1$, we have that $\textrm{h-Mon}_{m}$ implies $\textrm{h-Mon}_{n}$, and $\textrm{h-FinBW}_{m}$ implies $\textrm{h-FinBW}_{n}$. On the other hand, note that $\textrm{h-Mon}_{n}$ implies $\textrm{h-FinBW}_{n}$, since for each $n\geq 1$ and each $A\in [\omega]^{\omega}$, every bounded monotone sequence of real numbers $f:[A]^{n} \rightarrow \mathbb{R}$ is convergent, as stated in the following result.

\begin{lemma} \label{h-Mon_n implies h-FinBW_n}
    Let $\mathcal{I}$ be an ideal on $\omega$ and let $n\in\omega$ with $n\geq 1$. If $\mathcal{I}$ is $\textrm{h-Mon}_{n}$, then it is $\textrm{h-FinBW}_{n}$.
\end{lemma}

\begin{proof}
    Fix $n\in\omega$ with $n\geq 1$ and suppose that the ideal $\mathcal{I}$ is $\textrm{h-Mon}_{n}$. Let $A\in\mathcal{I}^{+}$ and let $f:[A]^{n} \rightarrow \mathbb{R}$ be a bounded sequence of real numbers, then there exists $B\in\mathcal{I}^{+} \!\restriction\! A$ such that the bounded subsequence $f \!\restriction\! [B]^{n}$ is monotone.

    \medskip

    We claim that $f \!\restriction\! [B]^{n}$ is a convergent sequence. Indeed, suppose that for every $s,t \in [B]^{n}$, if $s \triangleleft^{*} t$ then $f(s) \leq f(t)$. Let $q\in\mathbb{R}$ be given by $q= \sup f \text{\,''\,} [B]^{n}$; then, for every $\varepsilon >0$ there is $s_{\varepsilon}\in [B]^{n}$ such that $q-\varepsilon <f(s_{\varepsilon}) \leq q$. Now, notice that for each $t\in [B/\max s_{\varepsilon}]^{n}$ we have $s_{\varepsilon} \triangleleft^{*} t$ and hence $f(s_{\varepsilon}) \leq f(t) \leq q$. Thus, we infer that $f \text{\,''\,} [B/\max s_{\varepsilon}]^{n} \subseteq (q-\varepsilon, q+\varepsilon)$, which shows that $q$ is the limit of $f \!\restriction\! [B]^{n}$. On the other hand,  suppose that for every $s,t \in [B]^{n}$, if $s \triangleleft^{*} t$ then $f(s) \geq f(t)$. In this case, let $p\in\mathbb{R}$ be given by $p= \inf f \text{\,''\,} [B]^{n}$; then, an analogous argument shows that $p$ is the limit of $f \!\restriction\! [B]^{n}$. Therefore, we conclude that the ideal $\mathcal{I}$ is $\textrm{h-FinBW}_{n}$.
\end{proof}

\smallskip

Regarding the property $\textrm{h-FinBW}_{n}$ for ideals, it is natural to consider sequences indexed by sets of the form $[A]^n$ in arbitrary Hausdorff spaces, not only in compact subspaces of the real line, as was established for the property $\textrm{h-FinBW}$ in \cite[Subsection 2.3]{FMRS07} and \cite[Subsection 2.2]{FMRS12}.

\medskip

More precisely, let $n\in \omega$ with $n\geq 1$, let $\mathcal{I}$ be an ideal on $\omega$, and let $\mathcal{X}$ be a Hausdorff space. We say that the pair $(\mathcal{X}, \mathcal{I})$ is $\textrm{h-FinBW}_{n}$ if for every $A\in\mathcal{I}^{+}$ and every sequence $f: [A]^{n} \rightarrow \mathcal{X}$, there exists some $B\in \mathcal{I}^{+} \!\restriction\! A$ such that the subsequence $f \!\restriction\! [B]^{n}$ is convergent.

\medskip

As in the case of $\textrm{h-FinBW}$, it is easy to check that if $(\mathcal{X},\mathcal{I})$ is $\textrm{h-FinBW}_{n}$ for any $n\geq 1$, then so is $(\mathcal{Y},\mathcal{I})$ when $\mathcal{Y}$ is a continuous image of $\mathcal{X}$ or a closed subspace of $\mathcal{X}$. 
With this in mind, and using the well-known fact that every uncountable compact metrizable space is a continuous image of $2^{\omega}$ and contains a closed subspace homeomorphic to $2^{\omega}$, we obtain the following result.

\begin{fact} \label{fact_(X,I)}
    Let $\mathcal{I}$ be an ideal on $\omega$. Then, for all $n\in \omega$ with $n \geq 1$, the following statements are equivalent:
\setlist{nolistsep}
\begin{itemize}
\setlength{\itemsep}{0pt} 
\item[(a)] $\mathcal I$ is $\textrm{h-FinBW}_{n}$.
\item[(b)] $(2^{\omega},\mathcal I)$ is $\textrm{h-FinBW}_{n}$.
\item[(c)] $(\mathcal{X},\mathcal I)$ is $\textrm{h-FinBW}_{n}$ for some uncountable compact metrizable space $\mathcal{X}$.
\item[(d)] $(\mathcal{X},\mathcal I)$ is $\textrm{h-FinBW}_{n}$ for every uncountable compact metrizable space $\mathcal{X}$.
\end{itemize}
\end{fact}

\begin{proof}
    Straightforward. 
\end{proof}

\smallskip

As an application of the previous result, we obtain the following characterization of the property $\textrm{h-FinBW}_{n}$ for ideals, which extends Lemma \ref{h-FinBW Lemma} to higher dimensions.

\begin{proposition} \label{h-FinBW_n Lemma}
Let $\mathcal{I}$ be an ideal on $\omega$. Then, for all $n\in \omega$ with $n \geq 1$, the following statements are equivalent:
\setlist{nolistsep}
\begin{itemize}
\setlength{\itemsep}{0pt} 
\item[(a)] $\mathcal I$ is $\textrm{h-FinBW}_{n}$.
\item[(b)] For every $A\in\mathcal{I}^{+}$ and every countable collection $\{ f_{m} : [A]^{n} \rightarrow  \mathcal{X} \}_{m\in\omega}$ of sequences in $\mathcal{X}$ indexed by $[A]^{n}$, where $\mathcal{X}$ is any uncountable compact metrizable space, there exists $B\in \mathcal{I}^{+} \!\restriction\! A$ such that each subsequence $f_{m} \!\restriction\! [B]^{n}$ is convergent.   
\end{itemize}
\end{proposition}

\begin{proof}
    $[\text{(b)} \Longrightarrow \text{(a)}].$ Straightforward.

    \medskip

    $[\text{(a)} \Longrightarrow \text{(b)}].$ Suppose that $\mathcal{I}$ is $\textrm{h-FinBW}_n$, and let $\mathcal{X}$ be an uncountable compact metrizable space, so that the uncountable space $\mathcal{X}^{\omega}$, endowed with the product topology, is also compact and metrizable. Given $A\in\mathcal{I}^{+}$ and any countable collection $\{ f_{m} : [A]^{n} \rightarrow  \mathcal{X} \}_{m\in\omega}$ of sequences in $\mathcal{X}$, define the sequence $\varphi: [A]^{n} \rightarrow  \mathcal{X}^{\omega}$ in the product space $\mathcal{X}^{\omega}$ by setting $\varphi(a) = \{f_{m}(a)\}_{m\in\omega}$ for each $a\in [A]^{n}$. As the ideal $\mathcal{I}$ is $\textrm{h-FinBW}_n$, then there exist $B\in \mathcal{I}^{+} \!\restriction\! A$ and $\alpha \in \mathcal{X}^{\omega}$ such that the subsequence $\varphi \!\restriction\! [B]^{n}$ converges to $\alpha$.

    \medskip

    We claim that for each $m\in\omega$ the subsequence $f_{m} \!\restriction\! [B]^{n}$ converges to $\alpha(m)$. Indeed, given any open subset $U$ of $\mathcal{X}$ such that $\alpha(m) \in U$, consider the open subset $V = \prod_{i\in\omega} V_{i}$ of $\mathcal{X}^{\omega}$ determined by $V_{m} = U$ and $V_{i} = \mathcal{X}$ for $i\neq m$, so that $\alpha\in V$. Since $\varphi \!\restriction\! [B]^{n}$ converges to $\alpha$, then there is $j\in\omega$ such that $\varphi \text{\,''\,} [B/j]^{n}\subseteq V$; thus, for every $a\in [B/j]^{n}$ we have $\varphi(a) \in V$ and hence $f_{m}(a)\in U$, which implies that $f_{m} \text{\,''\,} [B/j]^{n}\subseteq U$. Therefore, we conclude that $\alpha(m)$ is the limit of the subsequence $f_{m} \!\restriction\! [B]^{n}$.     
\end{proof}

\smallskip

In what follows, we will show that an ideal is $\textrm{h-Mon}$ if and only if it is $\textrm{h-Mon}_{n}$ for all $n \geq 1$; likewise, an ideal is $\textrm{h-FinBW}_{2}$ if and only if it is $\textrm{h-FinBW}_{n}$ for all $n \geq 2$.

\begin{proposition} \label{Ramsey=h-Mon_n}
    Let $\mathcal I$ be an ideal on $\omega$. Then, the following statements are equivalent:
\setlist{nolistsep}
\begin{itemize}
\setlength{\itemsep}{0pt} 
\item[(a)] $\mathcal I$ is Ramsey.
\item[(b)] $\mathcal I$ is $\textrm{h-Mon}_{n}$ for all $n\geq 1$.
\end{itemize}
\end{proposition}

\begin{proof}
    $[\text{(b)} \Longrightarrow \text{(a)}].$ This fact follows directly from Proposition \ref{Ramsey = (q^{+})+(h-FinBW)}.

    \medskip

    $[\text{(a)} \Longrightarrow \text{(b)}].$ Suppose that $\mathcal{I}$ is Ramsey, then by Theorem \ref{Ramsey_ideal_theorem}, it follows that $\mathcal{I}^+ \longrightarrow (\mathcal{I}^+)^{n+1}_{2}$ holds for all $n\geq 1$. Now, fix $n\geq 1$ and consider a sequence of real numbers $f:[A]^{n} \rightarrow \mathbb{R}$, with $A\in\mathcal{I}^{+}$. 

    \medskip
    
    For each $s\in [A]^{n+1}$, define $\underline{s}, \overline{s} \in [A]^{n}$ by $\underline{s} = s \setminus \{\max s\}$ and $\overline{s} = s \setminus \{\min s\}$, thus we have that $\underline{s} \triangleleft \overline{s}$. Taking this into account, define the coloring $\pi: [A]^{n+1} \rightarrow 2$ by $\pi(s) = 0$ if and only if $f(\underline{s}) \leq f(\overline{s})$. Since $\mathcal{I}^+ \longrightarrow (\mathcal{I}^+)^{n+1}_{2}$ holds, there exists $B\in \mathcal{I}^{+} \!\restriction\! A$ such that $B$ is homogeneous for $\pi$; thus, either $\pi \text{\,''\,} [B]^{n+1} = \{0\}$ or $\pi \text{\,''\,} [B]^{n+1} = \{1\}$. Finally, conclude that $f$ is monotone on $[B]^{n}$.

    \medskip

    Indeed, let $s,t\in [B]^{n}$ be such that $s \triangleleft^{*} t$, then there exist $k\geq 1$ and $r_0, r_1, \ldots, r_k \in [B]^{n}$ such that $s = r_0 \triangleleft r_1 \triangleleft \dots \triangleleft r_k = t$. For each $1\leq i\leq k$, let $s_i \in [B]^{n+1}$ be given by $s_{i} = r_{i-1} \cup r_{i}$, and note that $\underline{s_{i}} = r_{i-1}$ and $\overline{s_{i}} = r_{i}$. If $\pi \text{\,''\,} [B]^{n+1} = \{0\}$, then $f(r_{i-1}) = f(\underline{s_{i}}) \leq f(\overline{s_{i}}) = f(r_{i})$ for each $1 \leq i \leq k$, which implies that $f(s) \leq f(t)$. Analogously, if $\pi \text{\,''\,} [B]^{n+1} = \{1\}$, then it follows that $f(s) \geq f(t)$. Therefore, we conclude that $\mathcal{I}$ is $\textrm{h-Mon}_{n}$.
\end{proof}

\begin{proposition} \label{Ramsey=h-FinBW_2}
    Let $\mathcal I$ be an ideal on $\omega$. Then, the following statements are equivalent:
\setlist{nolistsep}
\begin{itemize}
\setlength{\itemsep}{0pt} 
\item[(a)] $\mathcal I$ is Ramsey.
\item[(b)] $\mathcal I$ is $\textrm{h-FinBW}_{n}$ for all $n\geq 2$.
\end{itemize}
\end{proposition}

\begin{proof}
$[\text{(a)} \Longrightarrow \text{(b)}].$ This fact is easily deduced applying Proposition \ref{Ramsey=h-Mon_n} and Lemma \ref{h-Mon_n implies h-FinBW_n}.

\medskip

$[\text{(b)} \Longrightarrow \text{(a)}].$ Suppose that $\mathcal{I}$ is $\textrm{h-FinBW}_{n}$ for all $n\geq 2$, in particular $\mathcal{I}$ is $\textrm{h-FinBW}_{2}$. Let $A\in\mathcal{I}^{+}$ and let $\pi : [A]^{2} \rightarrow 2$ be a coloring of $[A]^{2}$; then, $\pi$ can be viewed as a bounded sequence $\pi :[A]^{2} \rightarrow \mathbb{R}$ with image contained in the discrete set $\{0,1\}$. Since $\mathcal{I}$ is $\textrm{h-FinBW}_{2}$, there exist $p\in 2$ and $B\in \mathcal{I}^{+} \!\restriction\! A$ such that the subsequence $\pi \!\restriction\! [B]^{2}$ converges to $p$. Thus, for any $0<\varepsilon<1$, there is $k_{\varepsilon}\in\omega$ such that $\pi \text{\,''\,} [B/k_{\varepsilon}]^{2} \subseteq (p-\varepsilon, p+\varepsilon)$ and hence $\pi \text{\,''\,} [B/k_{\varepsilon}]^{2} = \{p\}$, which means that $B/k_{\varepsilon} \in \mathcal{I}^{+} \!\restriction\! A$ is homogeneous for $\pi$. Therefore, we conclude that $\mathcal{I}$ is Ramsey.
\end{proof}

\smallskip

Subtle modifications of the properties $\textrm{h-FinBW}_{n}$ and $\textrm{h-Mon}_{n}$ play an important role in the study of non-hereditary versions of the Ramsey property for ideals, as we will show in the final section.

\medskip

The following open question may be viewed as a topological counterpart of Question~\ref{question_almost_homogeneous}, formulated in terms of simultaneous ideal convergence of sequences of varying dimensions.

\begin{question}
Let $\mathcal{X}$ be an uncountable compact metrizable space, let $\mathcal{I}$ be a Ramsey ideal, and let $\{m_i\}_{i\in\omega}$ be any sequence of integers such that each $m_i \geq 1$. If for every $i\in\omega$ we take any sequence $f_i : [\omega]^{m_i} \rightarrow \mathcal{X}$, is there a set $B\in \mathcal{I}^{+}$ such that each subsequence $f_{i} \!\restriction\! [B]^{m_i}$ is convergent?
\end{question}

\subsection{Ideals related to canonical colorings}

Building on the combinatorial properties of Ramsey ideals developed earlier, we will now establish an ideal version of the Erd\H{o}s--Rado canonization theorem.

\medskip

Let $r\in \omega$ with $r\geq 1$, and for each $x\in [\omega]^{r}$, let $x = \{x_i : i<r\}$ be its increasing enumeration. So, for every $\Gamma \subseteq r$, define $x \!\restriction\! \Gamma =\{x_i : i\in \Gamma\}$; then, it is clear that $\wp(x)= \{ x\!\restriction\!\Gamma : \Gamma \subseteq r\}$.

\begin{definition}
Given $r\in\omega$ with $r\geq 1$ and $A\in [\omega]^{\omega}$, we say that a coloring $f: [A]^{r}\rightarrow \omega$ is canonical on $C$, where $C\in [A]^{\omega}$, if there exists $\Gamma\subseteq r$ such that for all $x,y\in [C]^{r}$, we have $f(x) = f(y)$ if and only if $x \!\restriction\! \Gamma = y \!\restriction\! \Gamma$. Moreover, in this case, the set $\Gamma$ is unique and we say that $\Gamma$ canonizes $f$ on $C$.
\end{definition}

With this terminology in place, the Erd\H{o}s--Rado canonization theorem extends the Ramsey theorem by asserting that every coloring of $[A]^r$ becomes canonical on some infinite subset of $A$.

\begin{proposition}[\cite{Erdos-Rado}] \label{ER}
For every $r\in\omega$ with $r\geq 1$ and every coloring $f: [A]^{r}\rightarrow \omega$, where $A\in [\omega]^{\omega}$, there exists $C\in [A]^{\omega}$ such that $f$ is canonical on $C$.
\end{proposition}

As a further consequence of Theorem \ref{Ramsey_ideal_theorem}, we now present a canonization result for Ramsey ideals that localizes the Erd\H{o}s--Rado canonization theorem stated in Proposition \ref{ER}. To do this, we introduce some useful notation.

\medskip

Given $s,t\in [\omega]^{<\omega}$ with $s \subseteq t$, let $\Gamma_{s}^{t}$ be the unique subset of $|t|$ such that $s = t\!\restriction\!\Gamma_{s}^{t}$. Also, for $r\in \omega$ with $r \geq 1$ and $B \in [\omega]^{\omega}$, denote by $\equiv$ the equivalence relation on $[B]^{r} \times [B]^{r}$ defined by $(x,y) \equiv (u,v)$ if and only if $\Gamma_{x}^{x \cup y} = \Gamma_{u}^{u \cup v}$ and $\Gamma_{y}^{x \cup y} = \Gamma_{v}^{u \cup v}$. In other words, $(x,y) \equiv (u,v)$ means that the elements of $x$ and $y$ are interlaced in exactly the same way as the elements of $u$ and $v$. Moreover, note that if $(x,y) \equiv (u,v)$ then necessarily $|x\cup y| = |u\cup v|$.

\medskip

Let $\mathcal{I}$ be an ideal on $\omega$, and fix $r\in \omega$ with $r \geq 1$. The ideal $\mathcal I$ satisfies

\kern-0.5em
\begin{equation*}
    \mathcal{I}^+ \longrightarrow *(\mathcal{I}^+)^{r}
\end{equation*}

\kern-0.5em
if, for every $A\in \mathcal{I}^+$ and every coloring $f:[A]^{r} \rightarrow \omega$, there exists an $\mathcal{I}$-positive set $C\in \mathcal{I}^+ \!\restriction\! A$ such that $f$ is canonical on $C$.

\begin{theorem} \label{ideal_canonization_theorem}
    Let $\mathcal I$ be an ideal on $\omega$. Then, the following statements are equivalent:
\setlist{nolistsep}
\begin{itemize}
\setlength{\itemsep}{0pt}
\item[(a)] $\mathcal I$ is Ramsey.
\item[(b)] $\mathcal{I}^+ \longrightarrow *(\mathcal{I}^+)^{r}$ holds for all $r\geq 1$.
\end{itemize}
\end{theorem}

\begin{proof}
$[\text{(a)} \Longrightarrow \text{(b)}].$ Suppose that $\mathcal{I}$ is Ramsey, so we will prove by induction on $r$ that it satisfies $\mathcal{I}^+ \longrightarrow *(\mathcal{I}^+)^{r}$ for all $r\geq 1$.

\medskip

For the base case $r=1$, let $A\in \mathcal{I}^{+}$ and let $f: [A]^{1} \rightarrow \omega$ be a coloring. Next, take the coloring $\pi : [A]^{2} \rightarrow 2$ given by $\pi (\{n,m\}) = 0$ if and only if $f(\{n\}) \neq f(\{m\})$. Since $\mathcal{I}$ is Ramsey, there exists $C\in \mathcal{I}^{+}\!\restriction\! A$ that is homogeneous for $\pi$. Notice that if $\pi \,\text{''}\, [C]^{2} = \{0\}$, then $\Gamma= \emptyset$ canonizes $f$ on $C$; otherwise, if $\pi \,\text{''}\, [C]^{2} = \{1\}$, then $\Gamma= 1$ canonizes $f$ on $C$. Thus, $\mathcal{I}^+ \longrightarrow *(\mathcal{I}^+)^{1}$ holds.

\medskip

Inductively, suppose that $\mathcal{I}^+ \longrightarrow *(\mathcal{I}^+)^{r-1}$ holds for a fixed $r>1$, and consider a coloring $f: [A]^{r} \rightarrow \omega$ with $A \in\mathcal{I}^{+}$. Furthermore, fix an enumeration $\{ e_{i} : i<\ell\}$, where $\ell\in\omega$, of all equivalence relations on the finite set $[2r]^{r}$.  

\medskip

Now, given any $s\in [A]^{2r}$, consider the equivalence relation $\sim_{s}$ on $[2r]^{r}$ defined as follows: for $\Gamma,\Upsilon \in [2r]^{r}$, put $\Gamma \sim_{s} \Upsilon$ if and only if $f (s \!\restriction\! \Gamma) = f (s \!\restriction\! \Upsilon)$. Hence, $\sim_{s}$ belongs to the set $\{ e_{i} : i<\ell\}$. 

\medskip

Keeping this in mind, define the coloring $\pi : [A]^{2r} \rightarrow \ell$ as follows: for each $s\in [A]^{2r}$, assign $\pi(s) = i$ if and only if $\sim_{s} = e_{i}$. Consequently, since the ideal $\mathcal{I}$ is Ramsey, by Theorem \ref{Ramsey_ideal_theorem} we conclude that there exists $B\in \mathcal{I}^{+} \!\restriction\! A$ that is homogeneous for $\pi$, so that $\pi \,\text{''}\, [B]^{2r} = \{i\}$ for some $i<\ell$. With this in place, denote by $\sim$ the equivalence relation $e_{i}$ on $[2r]^{r}$, for which $\sim_{s}$ coincides with $\sim$ for each $s\in [B]^{2r}$. 

\medskip

On the one hand, suppose that the equivalence relation $\sim$ is the equality relation $=$. Let $x,y\in [B]^{r}$ be such that $f(x)=f(y)$; then, by taking $s\in [B]^{2r}$ with $x\cup y \sqsubseteq s$, we conclude that $\Gamma_{x}^{s} = \Gamma_{y}^{s}$, and hence $x = y$. As a result, $f \!\restriction\! [B]^{r}$ is an injective function. Similarly, one can show that the converse implication also holds. Therefore, in this case we obtain that $\Gamma = r$ canonizes $f$ on $B$.

\medskip

On the other hand, suppose that the equivalence relation $\sim$ is different from the equality relation $=$, in which case there exist $p,q \in [B]^{r}$ such that $p\neq q$ and $f(p)=f(q)$. Enumerate $B$ in increasing order as $B = \{ b_{j} : j<\omega\}$, and assume that $p = \{ b_{n_{i}} : i<r \}$ and $q = \{ b_{m_{i}} : i<r \}$, both listed in increasing order. Since $p \neq q$, we can fix some $k<r$ such that $b_{m_{k}} \in q \setminus p$; in particular, $p$ and $q$ differ at their $k$-th element.

\medskip

\textit{Claim 1.} For every $x,y,u,v \in [B]^{r}$, if $f(x) = f(y)$ and $(x,y) \equiv (u,v)$, then $f(u) = f(v)$.

\medskip

Indeed, $(x,y) \equiv (u,v)$ implies that $\Gamma_{x}^{x \cup y} = \Gamma_{u}^{u \cup v}$ and $\Gamma_{y}^{x \cup y} = \Gamma_{v}^{u \cup v}$. Now, taking $s,t \in [B]^{2r}$ such that $x\cup y \sqsubseteq s$ and $u\cup v \sqsubseteq t$, we obtain $\Gamma_{x}^{s} = \Gamma_{u}^{t}$ and $\Gamma_{y}^{s} = \Gamma_{v}^{t}$. Thus, if $f(x) = f(y)$, then $\Gamma_{x}^{s} \sim \Gamma_{y}^{s}$, so $\Gamma_{u}^{t} \sim \Gamma_{v}^{t}$, and hence $f(u) = f(v)$.

\medskip

\textit{Claim 2.} Let $x,y \in [B]^{r}$ with increasing enumerations $x= \{ b_{\xi_{i}} : i<r\}$ and $y= \{ b_{\zeta_{i}} : i<r\}$. If $x$ and $y$ differ only at their $k$-th element, that is $x \setminus \{b_{\xi_{k}}\} = y \setminus \{b_{\zeta_{k}}\}$ and $b_{\xi_{k}} < b_{\zeta_{k}}$, then $f(x) = f(y)$. 

\medskip

Indeed, based on the increasing enumeration of $p$ and $q$, define $p^{\prime}, q^{\prime} \in [B]^{r}$ by $p^{\prime} = \{ b_{2n_{i}} : i<r \}$ and $q^{\prime} = \{ b_{2m_{i}} : i<r \}$, and then define $q^{*} \in [B]^{r}$ by $q^{*} = (q^{\prime} \setminus \{ b_{2m_{k}} \}) \cup \{ b_{2m_{k}+1} \}$. Therefore, $q^{\prime}$ and $q^{*}$ differ only at their $k$-th element, and hence $(x,y) \equiv (q^{\prime},q^{*})$; moreover, it is clear that $(p^{\prime}, q^{\prime}) \equiv(p,q) \equiv (p^{\prime},q^{*})$. So, since $f(p)=f(q)$, we obtain $f(q^{\prime}) = f(p^{\prime}) = f(q^{*})$, and thus we also conclude that $f(x) = f(y)$.

\medskip

To continue, write $B$ as the finite union $B=B_{0} \cup B_{1}$ with $B_{0} =\{b_{2j} : j<\omega\}$ and $B_{1} =\{b_{2j+1} : j<\omega\}$, so at least one of these sets belongs to $\mathcal{I}^{+}$, and without loss of generality we may assume that $B_{0} \in \mathcal{I}^{+}$. 

\medskip

Next, define the coloring $g: [B_{0}]^{r-1} \rightarrow \omega$ as follows: for each $z\in [B_{0}]^{r-1}$, let $g(z) = f(z \cup \{n\})$, where $n\in B \setminus z$ is the element that occupies the $k$-th position of $z \cup \{n\}$ in its increasing enumeration. It should be noted that the coloring $g$ is well-defined due to the previous claim. By inductive hypothesis, $\mathcal{I}^+ \longrightarrow *(\mathcal{I}^+)^{r-1}$ holds, so there exist $C\in \mathcal{I}^+ \!\restriction\! B_{0} \subseteq \mathcal{I}^+ \!\restriction\! A$ and $\Lambda \subseteq r-1$ such that $g(z) = g(w)$ if and only if $z \!\restriction\! \Lambda = w \!\restriction\! \Lambda$ for all $z,w \in [C]^{r-1}$.

\medskip

Now, define $\Gamma \subseteq r$ by $\Gamma = \{ i<r : (i<k \wedge i\in\Lambda ) \vee ( i\geq k \wedge i-1 \in\Lambda) \}$, thus $k\notin \Gamma$, and for every $z\in [C]^{r-1}$ we have $z \!\restriction\! \Lambda = (z\cup\{n\}) \!\restriction\! \Gamma$ whenever $n\in B \setminus z$ corresponds to the $k$-th element of $z\cup\{n\}$.

\medskip

Finally, we affirm that $\Gamma$ canonizes $f$ on $C$. Indeed, given $x,y\in [C]^{r}$, both listed in increasing order, let $n_{x} \in x$ and $n_{y} \in y$ denote their respective $k$-th elements. Then, we have:

 \kern-0.5em
\begin{equation*}
    f(x)=f(y) \iff g(x \setminus \{n_{x}\}) = g(y \setminus \{n_{y}\}) \iff (x \setminus \{n_{x}\}) \!\restriction\! \Lambda = (y \setminus \{n_{y}\}) \!\restriction\! \Lambda \iff x \!\restriction\! \Gamma = y \!\restriction\! \Gamma.
\end{equation*}

\kern-0.5em
Therefore, we infer that the coloring $f$ is canonical on $C$; as a result, we conclude that $\mathcal{I}^+ \longrightarrow *(\mathcal{I}^+)^{r}$ also holds.

\medskip

$[\text{(b)} \Longrightarrow \text{(a)}].$ Suppose that $\mathcal{I}^+ \longrightarrow *(\mathcal{I}^+)^{r}$ holds for all $r\geq 1$, and consider a coloring $\pi :[A]^{2} \rightarrow 2$ with $A\in\mathcal{I}^{+}$. Since $\mathcal{I}^+ \longrightarrow *(\mathcal{I}^+)^{2}$ holds, there exists $B \in \mathcal{I}^{+} \!\restriction\! A$ such that $\pi$ is canonical on $B$. Let $\Gamma \subseteq 2$ be the set that canonizes $\pi$ on $B$; then, for all $x,y \in [B]^{2}$, we have $\pi(x) = \pi(y)$ if and only if $x \!\restriction\! \Gamma = y \!\restriction\! \Gamma$.

\medskip

Notice that the case $\Gamma = 2$ is impossible because $\pi$ takes at most two values, so we must have $\Gamma \neq 2$. Moreover, if $\Gamma = \emptyset$, then $\pi(x)=\pi(y)$ for all $x,y\in [B]^{2}$, in which case $B$ is homogeneous for $\pi$. In the remaining cases, $\Gamma=\{0\}$ or $\Gamma=\{1\}$, we can find an $\mathcal{I}$-positive subset of $B$ that is homogeneous for $\pi$.

\medskip

Indeed, suppose that $\Gamma = \{0\}$, then $\pi(x) = \pi(y)$ if and only if $\min x = \min y$ for all $x,y\in [B]^{2}$. Therefore, for each $n\in B$, the coloring $\pi$ takes a constant value on the set of pairs $\{ \{n,m\} : m\in B/n \}$. So, for each $i\in 2$, define the set $B_{i} = \{ n\in B : (\forall\, m\in B/n) ( \pi( \{n,m\}) =i )\}$, then $\pi \,\text{''}\, [B_{i}]^{2} = \{i\}$. Since $B=B_{0} \cup B_{1}$, either $B_{0} \in \mathcal{I}^{+} \!\restriction\! A$ or $B_{1} \in \mathcal{I}^{+} \!\restriction\! A $, in any case $B_{i}$ is homogeneous for $\pi$. Similarly, the case $\Gamma=\{1\}$ is handled in an analogous way.
\end{proof}

\section{An ideal version of Galvin's lemma}
\label{An ideal version of Galvin's lemma}

In this section, inspired by the Galvin lemma on finite partitions of the family of all finite subsets of the natural numbers, we introduce the notion of a Galvin ideal and investigate its main combinatorial properties. In particular, we show that Galvin ideals constitute an intermediate notion between Ramsey and semiselective ideals by constructing both a Ramsey tall ideal that is not Galvin and a Galvin tall ideal that is not semiselective. We also study ideals associated with colorings of barriers, leading to an ideal version of the Nash-Williams theorem.

\subsection{Galvin ideals}

Given a family $\mathcal{F}\subseteq [\omega]^{<\omega}$, we denote by
$\textrm{hom}(\mathcal{F})$ the collection of all sets $B\in [\omega]^{\omega}$ satisfying
either one of the following two alternatives:
\setlist{nolistsep}
\begin{itemize}
\setlength{\itemsep}{0pt}
\item[(I)] For every $C\in [B]^{\omega}$ there is $s\in
\mathcal{F}$ such that $s\sqsubset C$.
\item[(II)] $[B]^{<\omega} \cap \mathcal{F}=\emptyset$.
\end{itemize}

\medskip

Note that for every family $\mathcal{F} \subseteq [\omega]^{<\omega}$, the collection of its homogeneous sets $\textrm{hom}(\mathcal{F}) \subseteq [\omega]^{\omega}$ is a coanalytic tall family that is closed under taking infinite subsets.

\medskip

Using the notation introduced above, the Galvin lemma, which generalizes the Ramsey theorem, can be stated as follows.

\begin{proposition}[\cite{Galvin}] \label{Galvin_Lemma}
For every $\mathcal{F}\subseteq
[\omega]^{<\omega}$ and every $A\in [\omega]^{\omega}$ there exists $B\in [A]^{\omega}$ such that $B\in \hom(\mathcal{F})$.
\end{proposition}

With this in mind, we now introduce the combinatorial notion of a Galvin ideal, which may be viewed as an ideal-theoretic counterpart of the Galvin lemma.

\begin{definition}
An ideal $\mathcal{I}$ on $\omega$ is \textit{Galvin} if for every $\mathcal{F} \subseteq
[\omega]^{<\omega}$ and every $A\in \mathcal{I}^{+}$ there exists $B\in \mathcal{I}^{+} \!\restriction\! A$
such that $B\in \textrm{hom} (\mathcal{F})$.
\end{definition}

Thus, given a Galvin ideal $\mathcal{I}$, for every family $\mathcal{F} \subseteq [\omega]^{<\omega}$ the collection of its homogeneous sets $\textrm{hom}(\mathcal{F})$ is a dense-open set on $(\mathcal{I}^{+}, \subseteq)$.

\medskip

The connection between Galvin ideals and semiselective ideals was established in \cite[Theorem 2.2]{Farah1998} and \cite[Lemma 7.12]{Todorcevic2010} through the semiselective Galvin lemma, which asserts the following.

\begin{proposition}[\cite{Farah1998, Todorcevic2010}] \label{Semiselective_Galvin_Lemma}
Let $\mathcal{I}$ be a semiselective ideal on $\omega$. Then, for every $\mathcal{F}\subseteq
[\omega]^{<\omega}$ and every $A\in \mathcal{I}^{+}$ there exists $B\in \mathcal{I}^{+} \!\restriction\! A$ such that $B\in \hom(\mathcal{F})$.
\end{proposition}

It is straightforward to see that the notion of a Galvin ideal is a combinatorial property lying between semiselective ideals and Ramsey ideals. Indeed, every semiselective ideal is Galvin by Proposition \ref{Semiselective_Galvin_Lemma}. Moreover, every Galvin ideal is Ramsey, since the Galvin lemma implies the Ramsey theorem. For the sake of completeness, we present a sketch of the proof of this last fact.

\begin{proposition}
    Let $\mathcal{I}$ be a Galvin ideal on $\omega$. Then, for every $A\in\mathcal{I}^{+}$ and every coloring $\pi : [A]^{2} \rightarrow 2$ there exists $B\in\mathcal{I}^{+} \!\restriction\! A$ such that $B$ is homogeneous for $\pi$.
\end{proposition}

\begin{proof}
    Given $A\in \mathcal{I}^{+}$ and a coloring $\pi : [A]^{2} \rightarrow 2$, define the family $\mathcal{F} \subseteq [\omega]^{<\omega}$ by $\mathcal{F} = \pi^{-1} (\{0\})$. Since the ideal $\mathcal{I}$ is Galvin, there exists $B\in \mathcal{I}^{+} \!\restriction\! A$ such that $B \in \textrm{hom} (\mathcal{F})$. If every infinite subset of $B$ has an initial segment in $\mathcal{F}$, then $\pi \,\text{''}\, [B]^{2} = \{0\}$; otherwise, if no finite subset of $B$ belongs to $\mathcal{F}$, then $\pi \,\text{''}\, [B]^{2} = \{1\}$. Consequently, $B$ is homogeneous for $\pi$.
\end{proof}

\smallskip

Given a family $\mathcal{F} \subseteq [\omega]^{<\omega}$, for each $n\in \omega$ define $\mathcal{F}_{(n)} \subseteq [\omega]^{<\omega}$ by

\kern-0.5em
\begin{equation*}
    \mathcal{F}_{(n)} = \{ t\in [\omega/n]^{<\omega} : \{n\}\cup t \in \mathcal{F} \}.
\end{equation*}
 
\kern-0.5em
The next proposition states that an ideal $\mathcal{I}$ on $\omega$ is Galvin exactly when, for every family $\mathcal{F} \subseteq [\omega]^{<\omega}$, the collection $\{\textrm{hom}(\mathcal F_{(n)}) \}_{n\in\omega}$ has a dense set of diagonalizations in $(\mathcal{I}^+, \subseteq)$. Thus, this combinatorial characterization of the Galvin property for an ideal $\mathcal{I}$ can be viewed as a weakening of semiselectivity, since it involves only certain countable collections of dense-open subsets of $(\mathcal{I}^{+}, \subseteq)$.

\begin{proposition} \label{Galvin = version with diagonalizations}
Let $\mathcal I$ be an ideal on $\omega$. Then, the following statements are equivalent:
\setlist{nolistsep}
\begin{itemize}
\setlength{\itemsep}{0pt}
\item[(a)] $\mathcal{I}$ is Galvin.
\item[(b)] For every $\mathcal{F} \subseteq [\omega]^{\omega}$ and every $A\in\mathcal{I}^{+}$ there exists $B\in \mathcal{I}^{+} \!\restriction\! A$ such that $B/n \in \textrm{hom}(\mathcal{F}_{(n)})$ for all $n\in B$. 
\end{itemize}
\end{proposition}

\begin{proof}
$[\text{(a)} \Longrightarrow \text{(b)}].$ Suppose that $\mathcal{I}$ is Galvin, then for every $\mathcal{F} \subseteq [\omega]^{\omega}$ and every $A\in\mathcal{I}^{+}$ there exists $B\in \mathcal{I}^{+} \!\restriction\! A$ such that $B\in \textrm{hom}(\mathcal{F})$. On the one hand, if every infinite subset of $B$ has an initial segment in $\mathcal{F}$, then for each $n\in B$ and each $C\in [B/n]^{\omega}$ there is $s\in \mathcal{F}$ such that $s \sqsubset \{n\} \cup C$, thus $s=\{n\} \cup t$ and $t \sqsubset C$ for some $t\in \mathcal{F}_{(n)}$, which implies that $B/n \in \textrm{hom}(\mathcal{F}_{(n)})$. On the other hand, if $B$ does not contain finite subsets in $\mathcal{F}$, then for each $n\in B$ and each $t\in [B/n]^{<\omega}$ we have $\{n\} \cup t \notin \mathcal{F}$, thus $t\notin \mathcal{F}_{(n)}$ and hence $[B/n]^{<\omega} \cap \mathcal{F}_{(n)} = \emptyset$, which implies that $B/n \in \textrm{hom}(\mathcal{F}_{(n)})$.

\medskip

$[\text{(b)} \Longrightarrow \text{(a)}].$ Given $\mathcal F \subseteq [\omega]^{<\omega}$ and $A\in \mathcal I^+$, there exists $B\in \mathcal{I}^{+} \!\restriction\! A$ such that $B/n \in \textrm{hom}(\mathcal{F}_{(n)})$ for all $n\in B$. Thus, for each $n\in B$ there are two possibilities: either every infinite subset of $B/n$ has an initial segment in 
$\mathcal F_{(n)}$, or $B/n$ does not contain finite subsets in $\mathcal F_{(n)}$. With this in mind, split $B = B_{0} \cup B_{1}$ into two disjoint sets, where $B_{0} = \{ n\in B : (\forall\, C \in [B/n]^{\omega}) (\exists\, t\in \mathcal{F}_{(n)}) ( t \sqsubset C) \}$ and $B_{1} = \{ n\in B : [B/n]^{<\omega} \cap \mathcal{F}_{(n)} = \emptyset \}$, then at least one of these sets belongs to $\mathcal{I}^{+}$. In either case, one readily verifies that $B_{i} \in \textrm{hom}(\mathcal{F})$ for each $i\in 2$. Therefore, we conclude that $\mathcal{I}$ is Galvin.
\end{proof}

\smallskip

We continue by presenting a topological characterization of Galvin ideals in terms of the local Ramsey property relative to a given ideal.

\begin{definition}
Given an ideal $\mathcal{I}$ on $\omega$, we say that a subset $\mathcal{X} \subseteq [\omega]^{\omega}$ is $\mathcal{I}^{+}$-Ramsey if for every $A\in \mathcal{I}^{+}$ there exists $B\in \mathcal{I}^{+} \!\restriction\! A$ such that either $[B]^{\omega} \subseteq \mathcal{X}$ or $[B]^{\omega} \cap \mathcal{X} = \emptyset$. 
\end{definition}

It is easy to verify that for every ideal $\mathcal{I}$ on $\omega$, the collection of all $\mathcal{I}^{+}$-Ramsey sets forms an algebra of subsets of $[\omega]^{\omega}$. Moreover, it follows from \cite[Lemma 3.1]{Farah1998} and \cite[Corollary 7.17]{Todorcevic2010} that the collection of all $\mathcal{I}^{+}$-Ramsey sets forms a $\sigma$-algebra whenever the ideal $\mathcal{I}$ is semiselective. 

\medskip 

For each $s\in [\omega]^{<\omega}$,  let $[s]= \{ A\in [\omega]^{\omega} : s\sqsubset A\}$. The collection $\{ [s]  : s\in [\omega]^{<\omega} \}$ forms a countable basis of closed-open sets for a Polish topology over $[\omega]^{\omega}$, called the \textit{metrizable topology}, which also is obtained by identifying $[\omega]^{\omega}$ as a subspace of the Cantor space $2^{\omega}$. Thus, a subset $\mathcal{O}$ of $[\omega]^{\omega}$ is open if and only if there exists a family $\mathcal{F} \subseteq [\omega]^{<\omega}$ such that $\mathcal{O} = \bigcup_{s\in \mathcal{F}} [s]$. Consequently, for any $X\in[\omega]^{\omega}$, we have $X\in \mathcal{O}$ if and only if $s\sqsubset X$ for some $s\in \mathcal{F}$.

\begin{proposition} \label{Galvin = closed sets are Ramsey}
Let $\mathcal I$ be an ideal on $\omega$. Then, the following statements are equivalent:
\setlist{nolistsep}
\begin{itemize}
\setlength{\itemsep}{0pt}
\item[(a)] $\mathcal I$ is Galvin.
\item[(b)] Every open subset of $[\omega]^{\omega}$, with respect to the metrizable topology, is $\mathcal{I}^{+}$-Ramsey.
\item[(c)] Every closed subset of $[\omega]^{\omega}$, with respect to the metrizable topology, is $\mathcal{I}^{+}$-Ramsey.
\end{itemize}
\end{proposition}

\begin{proof}
$[\text{(b)} \Longleftrightarrow \text{(c)}].$ Straightforward.

\medskip

$[\text{(a)} \Longrightarrow \text{(b)}].$ Suppose that the ideal $\mathcal{I}$ is Galvin. Let $\mathcal{O}$ be an open subset of $[\omega]^{\omega}$ and consider a family $\mathcal{F}\subseteq [\omega]^{<\omega}$ such that $\mathcal{O} = \bigcup_{s\in \mathcal{F}} [s]$. Since $\mathcal{I}$ is Galvin, for every $A\in \mathcal{I}^{+}$ there exists $B\in \mathcal{I}^{+} \!\restriction\! A$ such that $B\in \textrm{hom}(\mathcal{F})$. On the one hand, if for every $C\in [B]^{\omega}$ there is $s\in\mathcal{F}$ such that $s \sqsubset C$, then we deduce that $[B]^{\omega} \subseteq \mathcal{O}$. On the other hand, if $[B]^{<\omega} \cap \mathcal{F} =\emptyset$, then we conclude that $[B]^{\omega} \cap \mathcal{O}= \emptyset$. Therefore, the open set $\mathcal{O}$ is $\mathcal{I}^{+}$-Ramsey.

\medskip 

$[\text{(b)} \Longrightarrow \text{(a)}].$ Suppose that every open subset of $[\omega]^{\omega}$ is $\mathcal{I}^{+}$-Ramsey. Let $\mathcal{F} \subseteq [\omega]^{\omega}$ and consider the open set $\mathcal{O} = \bigcup_{s\in \mathcal{F}} [s]$. Then, for every $A\in\mathcal{I}^{+}$ there exists $B\in \mathcal{I}^{+} \!\restriction\! A$ such that either $[B]^{\omega} \subseteq \mathcal{O}$ or $[B]^{\omega} \cap \mathcal{O} = \emptyset$. In the first case, $[B]^{\omega} \subseteq \mathcal{O}$ implies that for every $C\in [B]^{\omega}$ there is $s\in\mathcal{F}$ such that $s \sqsubset C$. In the second case, $[B]^{\omega} \cap \mathcal{O} = \emptyset$ implies that $[B]^{<\omega} \cap \mathcal{F} =\emptyset$. Thus, in either case we infer that $B\in \textrm{hom}(\mathcal{F})$. Therefore, the ideal $\mathcal{I}$ is Galvin. 
\end{proof}

\smallskip

The previous proposition naturally leads to the following question, whose answer is presently unknown, asking whether the characterization of Galvin ideals in terms of the local Ramsey property of open and closed sets can be extended to the broader class of Borel sets.

\begin{question} \label{Question: Galvin = Borel sets are Ramsey}
    Let $\mathcal{I}$ be a Galvin ideal on $\omega$. Is every Borel subset of $[\omega]^{\omega}$, with respect to the metrizable topology, an $\mathcal{I}^{+}$-Ramsey set?
\end{question}

The relationships and implications among the combinatorial notions of selectivity and semiselectivity as well as the Ramsey and Galvin properties for ideals introduced previously are summarized in the following diagram:

\kern-0.5em
\begin{equation*}
\text{Selective} 
\;\;\;\Longrightarrow\;\;\; 
\text{Semiselective} 
\;\;\;\Longrightarrow\;\;\;
\text{Galvin}
\;\;\;\Longrightarrow\;\;\;
\text{Ramsey}
\end{equation*}

\kern-0.5em
It is known that the first arrow is strict (see \cite[Example 2.1]{Farah1998}). In the next two subsections, we will show that both the
second arrow and the last arrow are also
strict.

\medskip 

An important question that could be closely associated with our work was asked  by Hru\v{s}\'{a}k in \cite[Question 5.19]{Hrusak2011} and \cite[Question 6.4]{Hrusak2017}: Is there a Borel (or even analytic) Ramsey tall ideal?

\medskip

In the case of Galvin ideals, we claim that there are no analytic Galvin tall ideals. To prove this, we use a result of Greb\'ik and Vidny\'anszky concerning analytic tall ideals containing $F_{\sigma}$ tall ideals (see \cite[Theorem 0.1]{Gre-Vid}), as well as a result of Mazur concerning $F_{\sigma}$ ideals generated by closed hereditary families (see \cite[Lemma 1.2]{Mazur}).

\begin{proposition}[\cite{Gre-Vid}] \label{GV}
    Every analytic tall ideal contains an $F_{\sigma}$ tall ideal.
\end{proposition}

\begin{proposition}[\cite{Mazur}] \label{Mazur-lemma}
    Every $F_{\sigma}$ tall ideal is generated by a closed hereditary tall family.
\end{proposition}

At this point, it is worth recalling that a family $\mathcal{K} \subseteq 2^{\omega}$ of subsets of $\omega$ is said to be \textit{hereditary} if, whenever $A \in \mathcal{K}$ and $B \subseteq A$, then $B \in \mathcal{K}$.

\begin{lemma} \label{Galvin_non-analytic_tall_lemma}
No analytic tall ideal can be extended to a Galvin tall ideal.
\end{lemma}

\begin{proof}
Let $\mathcal{I}$ be an analytic tall ideal on $\omega$, and let $\mathcal{J}$ any tall ideal on $\omega$ such that $\mathcal{I} \subseteq \mathcal{J}$. By Proposition \ref{GV}, there exists an $F_\sigma$ tall ideal $\mathcal{G}$ on $\omega$ contained in the analytic tall ideal $\mathcal{I}$, so that $\mathcal{G} \subseteq \mathcal{I} \subseteq \mathcal{J}$. By Proposition \ref{Mazur-lemma}, the $F_\sigma$ tall ideal $\mathcal{G}$ is generated by a closed hereditary tall family $\mathcal{K} \subseteq 2^{\omega}$. Therefore, for every $A\in \mathcal{J}^{+}$ we have that $[A]^{\omega} \not \subseteq \mathcal{K}$ and $[A]^{\omega} \cap \mathcal{K} \neq \emptyset$, which means that the closed family $\mathcal{K}$ (restricted to $[\omega]^{\omega}$) is not a $\mathcal{J}^{+}$-Ramsey set, and hence the ideal $\mathcal{J}$ is not Galvin, according to Proposition \ref{Galvin = closed sets are Ramsey}.
\end{proof}

\begin{proposition} \label{no analytic Galvin tall}
There are no analytic Galvin tall ideals.
\end{proposition}

\begin{proof}
    This result follows directly from Lemma \ref{Galvin_non-analytic_tall_lemma}.
\end{proof}

\smallskip

Regarding the descriptive complexity of Galvin ideals, it should be pointed out that \cite[Theorem 4.24]{GrebikUzca2019} provides a $\mathbf{\Pi}^{1}_{2}$ Galvin tall ideal on $\omega$. Nevertheless, it remains open whether a coanalytic Galvin tall ideal exists.

\subsection{A Ramsey non-Galvin tall ideal}

We recall the construction of a coanalytic Ramsey tall ideal from \cite[Theorem 4.7]{HMTU2017} in order to prove that it is not Galvin. To begin with, we present some preliminary results needed for the proof of Theorem \ref{ramsey-no-galvin}.

\medskip 

As noted in Proposition~\ref{GV}, it was shown in \cite[Theorem 0.1]{Gre-Vid} that every analytic tall ideal $\mathcal{I}$ on $\omega$ contains a closed hereditary tall family $\mathcal{K}$, and hence the ideal generated by $\mathcal{K}$ is an $F_{\sigma}$ tall ideal contained in $\mathcal{I}$. However, a closer analysis of the proof of \cite[Theorem 0.1]{Gre-Vid} reveals that the authors in fact establish the following more general result.

\begin{proposition}[\cite{Gre-Vid}] \label{GV-thm}
Let $\mathcal{A} \subseteq 2^{\omega}$ be an analytic hereditary family such that $[\omega]^{<\omega} \subseteq \mathcal{A}$. Then, there exists a closed hereditary family $\mathcal{K} \subseteq \mathcal{A}$ satisfying that for every infinite set $A\in\mathcal{A}$ there is an infinite subset $B \subseteq A$ such that $B\in \mathcal{K}$.
\end{proposition}

On the other hand, the following two lemmas are essentially drawn from \cite[Subsection 4.3]{HMTU2017}, which are indispensable for the construction of a Ramsey non-Galvin ideal, as will be seen in the proof of Theorem \ref{ramsey-no-galvin}.

\begin{lemma}[\cite{HMTU2017}]   \label{efective-Ramsey} 
There exists a Borel map $F: [\omega]^{\omega}\times 2^{[\omega]^2} \rightarrow
[\omega]^{\omega}$ such that, for every $A \in [\omega]^{\omega}$ and every coloring $\pi : [\omega]^{2} \rightarrow 2$, the infinite set $F(A,\pi)$ is a homogeneous subset of $A$ for $\pi$.
\end{lemma}

\begin{lemma}[\cite{HMTU2017}]  \label{psi-rho} 
There exists a continuous function
$\psi:[\omega]^{\omega}\times 2^\omega \rightarrow [\omega]^{\omega}$ such that,
for every $A\in [\omega]^{\omega}$, the collection $\{ \psi(A,x) : x\in 2^{\omega} \}$ is an almost disjoint family of infinite subsets of $A$. Moreover, there exists a continuous function $\eta:[\omega]^{\omega}\rightarrow [\omega]^{\omega}$ such that $\eta(A)\subseteq A$ and $\eta(A)\cap \psi(A,x)=\emptyset$ for all $x\in 2^\omega$ and all $A\in [\omega]^{\omega}$.
\end{lemma}

We now present the Ramsey ideal constructed in \cite[Theorem 4.7]{HMTU2017} with the aim of showing that it is not Galvin. It is worth mentioning that the construction of this ideal is based on the key idea used in \cite[Example 2.2]{Farah1998}, in which a Ramsey non-semiselective ideal is exhibited.

\begin{theorem} \label{ramsey-no-galvin} 
There exists a Ramsey tall ideal which is not Galvin.
\end{theorem}

\begin{proof} 
We begin by setting a continuous surjection $\varphi : 2^{\omega} \to 2^{[\omega]^2}$, so that its range is the collection of all colorings of $[\omega]^2$ into two colors. Moreover, let $F : [\omega]^{\omega} \times 2^{[\omega]^2} \rightarrow [\omega]^{\omega}$ be the Borel map given by Lemma \ref{efective-Ramsey}, and let also $\psi : [\omega]^{\omega} \times 2^{\omega} \rightarrow [\omega]^{\omega}$ and $\eta : [\omega]^{\omega} \rightarrow [\omega]^{\omega}$ be the continuous functions provided by Lemma \ref{psi-rho}.

\medskip 

Recursively construct a family $\{ A_{u} : u\in (2^{\omega})^{<\omega} \} \subseteq \wp([\omega]^{\omega})$ of infinite subsets of $\omega$, with $A_{\emptyset} = \omega$, as follows: Let $u\in (2^{\omega})^{<\omega}$ and suppose that $A_{u}$ has been defined; then, for each $x\in 2^{\omega}$, define $A_{u^{\smallfrown}x}$ by  

\kern-0.5em
\begin{equation*}
  A_{u^{\smallfrown}x} = F(\psi(A_s,x),\varphi(x)). 
\end{equation*}

\kern-0.5em
Thus, for every $u\in (2^{\omega})^{<\omega}$, the collection $\{ A_{u^{\smallfrown}x} : x\in 2^\omega \}$ is an almost disjoint family of subsets of $A_u$ with the property that, for each $x\in 2^{\omega}$, the infinite set $A_{u^{\smallfrown}x}$ is a homogeneous subset of $A_{s}$ for the coloring $\varphi(x)$.

\medskip

Next, consider the family $\{ \mathcal{C}_{n} : n\in\omega \}$ of analytic subsets of $[\omega]^{\omega}$, defined recursively by $\mathcal{C}_{0} = \{\omega\}$ and 

\kern-0.5em
\begin{equation*}
 \mathcal{C}_{n+1} = \{F(\psi(A,x),\varphi(x)):
(A,x)\in \mathcal{C}_n \times 2^\omega\}  =
\{A_{u^{\smallfrown} x}: |u|=n \,\wedge\, x\in 2^\omega\}.
\end{equation*}

\kern-0.5em
In view of the above, let $\mathcal{I}^{+}$ be the coideal on $\omega$ defined by

\kern-0.5em
\begin{equation*}
\mathcal{I}^{+} = \{ X \in [\omega]^{\omega} : (\exists\, n\in \omega) (\exists\, C\in
\mathcal{C}_n) (C \subseteq^{*} X) \}. 
\end{equation*} 

\kern-0.5em
Finally, conclude that $\mathcal{I} = \mathcal{P}(\omega)\setminus\mathcal{I}^{+}$ is
a coanalytic Ramsey tall ideal on $\omega$ (see \cite[Theorem 4.7]{HMTU2017}).  

\medskip 

To continue, consider the collection $\mathcal{A} = \{ \eta(A) : ( \exists\, n\in\omega )(A\in \mathcal{C}_{n}) \}$, which is an analytic subset of $[\omega]^{\omega}$. Moreover, notice that $\mathcal{A} \subseteq \mathcal{I}$ and every set from $\mathcal{I}^{+}$ contains a subset from $\mathcal{A}$. 

\medskip 

Now, let $\mathcal{I}_{\mathcal{A}}$ be the ideal generated by $\mathcal{A}$, then $\mathcal{I}_{\mathcal{A}}$ is an analytic ideal on $\omega$ such that $\mathcal{I}_{\mathcal{A}} \subseteq \mathcal{I}$. By Proposition \ref{GV-thm}, there exists a closed hereditary family $\mathcal{K} \subseteq \mathcal{I}_{\mathcal{A}} \subseteq \mathcal{I}$ such that every infinite set from $\mathcal{A}$ contains an infinite subset from $\mathcal{K}$. 

\medskip 

We claim that the closed family $\mathcal{K}$ (restricted to $[\omega]^{\omega}$) is not a $\mathcal{I}^{+}$-Ramsey set. Indeed, for every $X\in\mathcal{I}^{+}$ we have that $[X]^{\omega} \not\subseteq \mathcal{K}$ due to $\mathcal{K} \subseteq \mathcal{I}$; nevertheless, for every $X\in\mathcal{I}^{+}$ there is $Y\in [X]^{\omega}$ such that $Y\in \mathcal{A}$, then there is also $Z\in [Y]^{\omega} \subseteq [X]^{\omega}$ such that $Z \in \mathcal{K}$, which implies that $[X]^{\omega} \cap \mathcal{K} \neq \emptyset$. 

\medskip 

Therefore, by virtue of Proposition \ref{Galvin = closed sets are Ramsey}, we finally conclude that the coanalytic Ramsey tall ideal $\mathcal{I}$ is not Galvin. 
\end{proof}

\smallskip 

In the previous theorem, Proposition \ref{GV-thm} was used to obtain an example of a Ramsey ideal that is not Galvin. Next, we provide another application of Proposition \ref{GV-thm}.

\medskip

A simple way to obtain tall families is as follows: Given a hereditary family $\mathcal{A} \subseteq 2^{\omega}$, the orthogonal of $\mathcal{A}$ is the family $\mathcal{A}^{\perp} \subseteq 2^{\omega}$ defined by $\mathcal{A}^{\perp} = \{B\subseteq \omega: (\forall\, A\in \mathcal{A} ) (|A\cap B| <\omega) \}$. With this in mind, observe that for any hereditary family $\mathcal{A}$, the family $\mathcal{A} \cup \mathcal{A}^\perp$ is tall.

\medskip 

In \cite[Question 10.13 and Example 10.14(b)]{Uzc}, it is asked whether, for a simply definable hereditary family $\mathcal{A}$, the tall family $\mathcal{A} \cup \mathcal{A}^{\perp}$ necessarily contains a set of the form $\textrm{hom}(\mathcal{F})$ for some $\mathcal{F} \subseteq [\omega]^{<\omega}$. The following result provides a partial answer of this question.

\begin{proposition}
 For every analytic ideal $\mathcal{I}$ on $\omega$, there exists $\mathcal{F}\subseteq [\omega]^{<\omega}$ such that $\textrm{hom}(\mathcal{F})\subseteq \mathcal{I} \cup \mathcal{I}^{\perp}$. In particular, the ideal generated by 
 $\mathcal{I}\cup \mathcal{I}^{\perp}$ is tall but not Galvin.
\end{proposition}

\begin{proof}
Let $\mathcal{I}$ be an analytic ideal on $\omega$. By Proposition \ref{GV-thm}, there exists a closed hereditary family $\mathcal{K}\subseteq \mathcal{I}$ such that every infinite set from $\mathcal{I}$ contains an infinite subset from $\mathcal{K}$, in which case $\mathcal{K}^\perp=\mathcal{I}^\perp$ and hence $\mathcal{K}\cup \mathcal{K}^\perp \subseteq \mathcal{I} \cup \mathcal{I}^\perp$. With this in mind, let $\mathcal{F} \subseteq [\omega]^{<\omega}$ be defined by 

\kern-0.5em
\begin{equation*}
\mathcal{F} = \{ s\in [\omega]^{<\omega} : (\forall\, A\in \mathcal{K}) (s\not\subseteq A) \}. 
\end{equation*} 

On the one hand, let us check that $\{ H\in [\omega]^{\omega} : [H]^{<\omega} \cap \mathcal{F} = \emptyset\} \subseteq \mathcal{K}$. Indeed, let $H\in [\omega]^{\omega}$ be such that it contains no finite subsets in $\mathcal{F}$. For each $n\in\omega$, define $s_{n}\in [H]^{<\omega}$ by $s_{n} = \{m\in H : m<n \}$. As $s_{n} \notin \mathcal{F}$, there is $A_{n}\in \mathcal{K}$ such that $s_{n} \subseteq A_{n}$. By compactness of $\mathcal{K}$, there are $A\in \mathcal{K}$ and a strictly increasing sequence $\{ n_{k} \}_{k\in\omega} \subseteq \omega$ such that $\lim_{k\to\infty} A_{n_{k}} = A$. As $s_{n} \subseteq s_{n+1}$ for each $n\in\omega$, it follows that $s_{n_{k}} \subseteq A$ for all $k\in\omega$, and hence $H \subseteq A$. Finally, since $\mathcal{K}$ is hereditary, we deduce that $H\in\mathcal{K}$.

\medskip

On the other hand, let us verify that $\{ H\in [\omega]^{\omega} : (\forall\, X\in [H]^{\omega}) (\exists\, s\in\f) (s \sqsubset X) \} \subseteq \mathcal{K}^\perp$. Indeed, let $H\in [\omega]^{\omega}$ be such that $H \notin \mathcal{K}^\perp$. Then, there exists $B\in\mathcal{K}$ such that $B\cap H$ is infinite. Thus, for each $s \sqsubset B\cap H$ we have that $s\subseteq B$ and hence $s\notin \mathcal{F}$. Consequently, $B\cap H$ has no initial segments in $\mathcal{F}$. 

\medskip

Therefore, we conclude that the family $\mathcal{F}$ satisfies $\textrm{hom}(\mathcal{F}) \subseteq \mathcal{K}\cup \mathcal{K}^\perp$, and thus $\textrm{hom}(\mathcal{F}) \subseteq \mathcal{I}\cup \mathcal{I}^\perp$.
\end{proof}

\subsection{A Galvin non-semiselective tall ideal}

We review the construction of a Ramsey tall ideal that does not satisfy property $p^{+}$, as presented in \cite[Claim 5.9]{Hrusak2011} (see also \cite[Example 3.1]{HMTU2017}), in order to show that it is in fact a Galvin ideal.

\medskip

Let $\mathcal{M}$ and $\mathcal{N}$ be maximal almost disjoint families on $\omega$. We say that $\mathcal{N}$ refines $\mathcal{M}$ if each $N\in \mathcal{N}$ is almost contained in some $M\in\mathcal{M}$. Then, for every $M\in\mathcal{M}$ it holds that $\mathcal{N}\cap [M]^{\omega}$ is a maximal almost disjoint family of subsets of $M$.

\begin{proposition} \label{Galvin non-selective}
    There exists a Galvin tall ideal which does not satisfy property $p^{+}$.
\end{proposition}

\begin{proof}
    Let $\{ \mathcal{M}_{n} : n\in\omega \}$ be a countable collection of maximal almost disjoint families on $\omega$ such that $\mathcal{M}_{n+1}$ refines $\mathcal{M}_{n}$ for each $n\in\omega$. Moreover, let $\mathcal{I}_{\mathcal{M}_{n}}$ be the selective tall ideal generated by $\mathcal{M}_{n}$, so that $\mathcal{I}_{\mathcal{M}_{n}} = \{ A \subseteq \omega : | \{ M\in\mathcal{A}_{n} : |A\cap M| =\omega \} | < \omega \}$.

    \medskip
        
    In view of the above, let $\mathcal{I}$ be the ideal on $\omega$ defined by
    
    \kern-0.5em
    \begin{equation*}
    \textstyle \mathcal{I} = \bigcap_{n\in\omega} \mathcal{I}_{\mathcal{M}_{n}}. 
    \end{equation*}
    
    \kern-0.5em
    We claim that the ideal $\mathcal{I}$ is Galvin and tall. On the one hand, it is clear that $\mathcal{I}$ is tall, as it is a countable intersection of tall ideals. On the other hand, to see that $\mathcal{I}$ is Galvin, we use the fact that each selective ideal $\mathcal{I}_{\mathcal{M}_{n}}$ is itself Galvin. Indeed, given $\mathcal{F} \subseteq [\omega]^{<\omega}$ and $X\in\mathcal{I}^{+}$, we have that $X\in \bigcup_{n\in\omega} \mathcal{I}_{\mathcal{M}_{n}}^{+}$, then there is $n\in\omega$ such that $X\in \mathcal{I}_{\mathcal{M}_{n}}^{+}$. Since $\mathcal{I}_{\mathcal{M}_{n}}^{+}$ is Galvin, there exists $Y\in \mathcal{I}_{\mathcal{M}_{n}}^{+} \!\restriction\! X \subseteq \mathcal{I}^{+} \!\restriction\! X$ such that $Y\in\textrm{hom}(\mathcal{F})$. Thus, we infer that $\mathcal{I}$ is Galvin.

    \medskip

    Finally, we assert that $\mathcal{I}$ fails property $p^{+}$. Indeed, since each $\mathcal{M}_{n+1}$ refines $\mathcal{M}_{n}$, it follows that $\mathcal{M}_{n} \subseteq \mathcal{I}_{\mathcal{M}_{n+1}}^{+} \subseteq \mathcal{I}^{+}$ for every $n\in\omega$. Due to this, we choose a sequence $\{M_{n}\}_{n\in\omega} \subseteq \mathcal{I}^{+}$ such that $M_{n+1} \subseteq M_{n}$ and $M_{n} \in \mathcal{M}_{n} \subseteq \mathcal{I}_{\mathcal{M}_{n}}$ for each $n\in\omega$. Now, let $A\in [\omega]^{\omega}$ be such that $A \subseteq^{*} M_{n}$ for all $n\in\omega$, then $A\in \mathcal{I}_{\mathcal{M}_{n}}$ for all $n\in\omega$, and hence $A\in\mathcal{I}$. Therefore, we conclude that the Galvin tall ideal $\mathcal{I}$ does not satisfy property $p^{+}$. 
\end{proof}

\smallskip

Let us recall the following special kind of maximal almost disjoint family (see, for instance, \cite{Guzman}). A maximal almost disjoint family $\mathcal{M}$ on $\omega$ is completely separable if each $X \in \mathcal{I}_{\mathcal{M}}^{+}$ contains some $M \in \mathcal{M}$. It is known that, consistently, a completely separable maximal almost disjoint family exists; in particular, this holds under the assumption that $\mathfrak{c} < \aleph_{\omega}$. However, whether such maximal almost disjoint families exist in $\mathrm{ZFC}$ remains an open problem.

\medskip

In the setting of Proposition \ref{Galvin non-selective}, we claim that if $\{ \mathcal{M}_{n} : n\in\omega \}$ is a countable collection of completely separable maximal almost disjoint families on $\omega$ with each $\mathcal{M}_{n+1}$ refining $\mathcal{M}_{n}$, then the Galvin tall ideal $\mathcal{I}= \bigcap_{n\in\omega} \mathcal{I}_{\mathcal{M}_{n}}$ is not semiselective. 

\medskip

Indeed, for each $n\in\omega$ let $\mathcal{M}_{n}^{*} = \{ Y\in \mathcal{I}^{+} : (\exists\, M\in\mathcal{M}_{n}) (Y \subseteq^{*} M) \}$, and observe that each $\mathcal{M}_{n}^{*}$ is a dense-open set on $(\mathcal{I}^{+}, \subseteq^{*})$ such that $X \not\subseteq^{*} Y$ for every $X\in \mathcal{I}_{\mathcal{M}_{n}}^{+}$ and every $Y\in\mathcal{M}_{n}^{*}$. Thus, the ideal $\mathcal{I}= \bigcap_{n\in\omega} \mathcal{I}_{\mathcal{M}_{n}}$ does not satisfy property $p^{w}$.

\medskip

Next, we construct in $\mathrm{ZFC}$ a Galvin ideal that is not semiselective, following the main method from the constructions of Ramsey ideals in \cite[Example 2.2]{Farah1998} and \cite[Theorem 4.7]{HMTU2017}, together with the key idea underlying the construction of a $\mathbf{\Pi}^{1}_{2}$ Galvin ideal in \cite[Theorem 4.24]{GrebikUzca2019}.

\begin{theorem} \label{galvin-no-semiselective} 
There exists a Galvin tall ideal which is not semiselective.
\end{theorem}

\begin{proof} 
Fix an enumeration $\wp([\omega]^{<\omega}) = \{ \mathcal{F}_{\xi} : \xi \in 2^{\omega} \}$ of all subsets of $[\omega]^{<\omega}$, and take a function $\eta: [\omega]^{\omega} \rightarrow [\omega]^{\omega}$ such that $\eta(A) \subseteq A$ and $|A \setminus\eta(A)|=\omega$ for all $A\in [\omega]^{\omega}$. 

\medskip

We recursively construct a family $\{ A_{u} : u\in (2^{\omega})^{<\omega} \} \subseteq [\omega]^{\omega}$, with $A_{\emptyset} =\omega$, such that for every $u\in (2^{\omega})^{<\omega}$, the collection $\{ A_{u^{\smallfrown}\xi} : \xi\in2^{\omega} \} \subseteq [A_{u} \setminus \eta(A_{u})]^{\omega}$ is an almost disjoint family of subsets of $A_{u} \setminus \eta(A_{u})$ such that $A_{u^{\smallfrown}\xi} \in \textrm{hom}(\mathcal{F}_{\xi})$ for each $\xi \in 2^{\omega}$.

\medskip

Indeed, let $u\in (2^{\omega})^{<\omega}$ and suppose that $A_{u}$ has been defined. Fix an almost disjoint family $\{ B_{\xi}^{u} : \xi\in2^{\omega} \} \subseteq [A_{u} \setminus \eta(A_{u})]^{\omega}$ of subsets of $A_{u}$ that are disjoint from $\eta(A_{u})$. By the Galvin lemma (see Proposition \ref{Galvin_Lemma}), for each $\xi\in 2^{\omega}$ there exists $A_{u^{\smallfrown}\xi} \in [B_{\xi}^{u}]^{\omega} \subseteq [A_{u} \setminus \eta(A_{u})]^{\omega}$ such that $A_{u^{\smallfrown}\xi} \in \textrm{hom} (\mathcal{F}_{\xi})$. Moreover, the family $\{ A_{u^{\smallfrown}\xi} : \xi\in2^{\omega} \} \subseteq [A_{u} \setminus \eta(A_{u})]^{\omega}$ is almost disjoint. 

\medskip

In view of the above, we now consider the coideal $\mathcal{I}^{+}$ on $\omega$ defined by

\kern-0.5em
\begin{equation*}
\mathcal{I}^{+} = \{ X\in [\omega]^{\omega} : (\exists\, u\in (2^{\omega})^{<\omega}) (A_{u}\subseteq^{*} X)\}. 
\end{equation*} 

\kern-0.5em
To see that $\mathcal{I}^{+}$ is in fact a coideal, note first that $\{ A_{u} : u\in (2^{\omega})^{<\omega} \} \subseteq \mathcal{I}^{+}$ and that $(\mathcal{I}^{+}, \subseteq^{*})$ is closed upward. Now, let $X,Y \in [\omega]^{\omega}$ be such that $X \cup Y \in \mathcal{I}^{+}$, then $A_{u} \subseteq^{*} X\cup Y$ for some $u \in (2^{\omega})^{<\omega}$. Let $\xi\in 2^{\omega}$ be such that $\mathcal{F}_{\xi} = [X]^{2}$, then $A_{u^{\smallfrown}\xi} \in \textrm{hom}(\mathcal{F}_{\xi})$. If every infinite subset of $A_{u^{\smallfrown}\xi}$ has an initial segment in $[X]^{2}$, then $A_{u^{\smallfrown}\xi} \subseteq^{*} X$, but if $A_{u^{\smallfrown}\xi}$ does not contain finite subsets in $[X]^{2}$, then $A_{u^{\smallfrown}\xi} \subseteq^{*} Y$. Thus, either $X\in \mathcal{I}^{+}$ or $Y\in \mathcal{I}^{+}$.

\medskip

By construction, the ideal $\mathcal{I} = \mathcal{P}(\omega)\setminus\mathcal{I}^{+}$ is Galvin and tall. On the one hand, note that $\mathcal{I}$ is tall, since $\{ \eta(A_{u}) : u\in (2^{\omega})^{<\omega} \} \subseteq \mathcal{I}$, and for every $X\in\mathcal{I}^{+}$ there exists $u\in (2^{\omega})^{<\omega}$ such that $\eta(A_{u}) \subseteq A_{u} \subseteq^{*} X$. On the other hand, observe that $\mathcal{I}$ is Galvin, since for every $\mathcal{F} \subseteq [\omega]^{<\omega}$ and every $X\in\mathcal{I}^{+}$ there exist $\xi\in2^{\omega}$ and $u\in(2^{\omega})^{<\omega}$ such that $\mathcal{F} = \mathcal{F}_{\xi}$ and $A_{u} \subseteq^{*} X$, which implies that $A_{u^{\smallfrown}\xi} \in \textrm{hom}(\mathcal{F})$.

\medskip

Finally, we verify that the ideal $\mathcal{I} = \mathcal{P}(\omega)\setminus\mathcal{I}^{+}$ is not semiselective. Indeed, consider the sequence of dense-open sets $\{\mathcal{D}_{n}\}_{n\in\omega}$ on $(\mathcal{I}^{+}, \subseteq^{*})$ defined by 

\kern-0.5em
\begin{equation*}
    \mathcal{D}_{n} = \{ Y\in\mathcal{I}^{+} : (\exists\, u\in (2^{\omega})^{n}) (Y \subseteq^{*} A_{u}) \}.
\end{equation*}

\kern-0.5em
Given any $X \in \mathcal{I}^{+}$, there exists $u \in (2^{\omega})^{<\omega}$ such that $A_u \subseteq^{*} X$, and we may assume that $|X \setminus A_u| = \omega$. Now, let $n=|u|$ and let $Y\in \mathcal{D}_{n}$; then, there exists $v\in (2^{\omega})^{n}$ such that $Y \subseteq^{*} A_{v}$. Since $|u|=|v|$, it follows that $|X \setminus Y|= \omega$ and hence $X \not\subseteq^{*} Y$. As a result, we deduce that $\mathcal{I}$ does not satisfy property $p^{w}$. Therefore, by virtue of Proposition \ref{semiselective=(p^w)+(q)}, we finally conclude that the Galvin tall ideal $\mathcal{I}$ is not semiselective.   
\end{proof}

\smallskip

At present, we do not know whether the Galvin ideal $\mathcal{I}$ constructed in the previous theorem has the property that every Borel or analytic subset of $[\omega]^{\omega}$ is $\mathcal{I}^{+}$-Ramsey. However, it is worth noting that the example of a Ramsey non-semiselective ideal $\mathcal{J}$ given in \cite[Example 2.2]{Farah1998} satisfies the additional property that every analytic subset of $[\omega]^{\omega}$ is $\mathcal{J}^{+}$-Ramsey. By Proposition \ref{Galvin = closed sets are Ramsey}, it follows that such an ideal $\mathcal{J}$ is also Galvin. This observation motivates the following question, which is closely related to Question \ref{Question: Galvin = Borel sets are Ramsey}.

\begin{question} \label{Question: Galvin = analytic sets are Ramsey}
    Let $\mathcal{I}$ be a Galvin ideal on $\omega$. Is every analytic subset of $[\omega]^{\omega}$, with respect to the metrizable topology, an $\mathcal{I}^{+}$-Ramsey set?
\end{question}

The following result establishes that every analytic set $\mathcal{A} \subseteq [\omega]^{\omega}$, for which $(\mathcal{A}, \subseteq)$ is closed downward, is $\mathcal{I}^{+}$-Ramsey for any Galvin ideal $\mathcal{I}$.

\begin{proposition}
    Let $\mathcal{I}$ be a Galvin ideal on $\omega$, and let $\mathcal{A}$ be an analytic subset of $[\omega]^{\omega}$ with respect to the metrizable topology. If $\mathcal{A}$ is closed under taking infinite subsets, then $\mathcal{A}$ is $\mathcal{I}^{+}$-Ramsey.
\end{proposition}

\begin{proof}
    Let $\mathcal{I}$ be a Galvin ideal on $\omega$, and let $\mathcal{A} \subseteq [\omega]^{\omega}$ be an analytic set that is closed under taking infinite subsets. By Proposition \ref{GV-thm}, there exists a closed hereditary family $\mathcal{K} \subseteq \mathcal{A} \cup [\omega]^{<\omega}$ satisfying that for every $A\in\mathcal{A}$ there is $B\in[A]^{\omega}$ such that $B\in\mathcal{K}$.

    \medskip

    Since $\mathcal{I}$ is Galvin, the closed family $\mathcal{K}$ (restricted to $[\omega]^{\omega}$) is $\mathcal{I}^{+}$-Ramsey, according to Proposition \ref{Galvin = closed sets are Ramsey}. Thus, for every $X\in \mathcal{I}^{+}$ there exists $Y\in\mathcal{I}^{+} \!\restriction\! X$ such that either $[Y]^{\omega} \subseteq \mathcal{K}$ or $[Y]^{\omega} \cap \mathcal{K} = \emptyset$. On the one hand, $[Y]^{\omega} \subseteq \mathcal{K}$ implies that $[Y]^{\omega} \subseteq \mathcal{A}$, since $\mathcal{K} \subseteq \mathcal{A} \cup [\omega]^{<\omega}$. On the other hand, $[Y]^{\omega} \cap \mathcal{K} = \emptyset$ implies that $[Y]^{\omega} \cap \mathcal{A} = \emptyset$, since each member of $\mathcal{A}$ contains an infinite subset belonging to $\mathcal{K}$. As a result, we conclude that $\mathcal{A}$ is $\mathcal{I}^{+}$-Ramsey.
\end{proof}

\subsection{Ideals related to colorings of barriers}

One can weaken the notion of a Galvin ideal in a natural way via an ideal version of the Nash-Williams theorem for barriers. First of all, we present the concepts of barriers and fronts introduced in \cite{NashWilliams} (see \cite[Definition II.2.1]{Todorcevic2005} and \cite[Definition 1.20]{Todorcevic2010}).

\begin{definition}
    Let $\mathcal{B} \subseteq [M]^{<\omega}$ be a family of nonempty finite subsets of an infinite set $M \in [\omega]^{\omega}$. The family $\mathcal{B}$ is called a \textit{barrier} on $M$ if it satisfies the following conditions:
\setlist{nolistsep}
\begin{enumerate}
\setlength{\itemsep}{0pt} 
        \item[i.] $\mathcal{B}$ is an antichain with respect to $\subseteq$, that is,  for all $s,t \in \mathcal{B}$ with $s \neq t$, we have $s \not\subseteq t$.
        \item[ii.] For every $A\in [M]^{\omega}$ there exists $s\in \mathcal{B}$ such that $s \sqsubset A$.
\end{enumerate}
Moreover, if the relation $\subseteq$ is replaced by $\sqsubseteq$ in the first item, then the family $\mathcal{B}$ is called a \textit{front} on $M$.
\end{definition}

Clearly, every barrier is itself a front. Moreover, every front on an infinite set $M \subseteq\omega$ contains a barrier on some infinite subset $N \subseteq M$. Furthermore, every barrier $\mathcal{B} \subseteq [M]^{<\omega}$ is a maximal antichain with respect to both $\sqsubseteq$ and $\subseteq$. Finally, observe that for each $M \in [\omega]^{\omega}$ and each $n \in \omega$, the family $[M]^{n}$ is a barrier on $M$.

\medskip

In what follows, we work with barriers rather than fronts, following the treatment in \cite[Sections II.2 and II.3]{Todorcevic2005} and \cite[Section 1.3]{Todorcevic2010}. 

\medskip

Given $M\in[\omega]^{\omega}$ and a barrier $\mathcal{B}$ on $M$, for each $n\in M$ let $\mathcal{B}_{(n)}$ be the barrier on $M/n$ defined by

 \kern-0.5em
\begin{equation*}
    \mathcal{B}_{(n)} = \{ t\in [M/n]^{<\omega} : \{n\}\cup t \in \mathcal{B} \}.
\end{equation*}

\kern-0.5em
For instance, the \textit{Schreier barrier} $\mathcal{S}$ is the barrier on $\omega$ defined by $\mathcal{S} = \{ s\in [\omega]^{<\omega} : |s|= 1 +\min s \}$, and for each $n\in\omega$ we have $\mathcal{S}_{(n)} = [\omega/n]^{n}$. 

\medskip

On the other hand, given a barrier $\mathcal{B}$ on $M \in [\omega]^{\omega}$ and $N \in [M]^{\omega}$, the restriction of $\mathcal{B}$ to $N$ is the barrier on $N$ defined by $\mathcal{B} \!\restriction\! N = \mathcal{B} \cap [N]^{<\omega}$. 

\medskip

We now present the notion of the rank of a barrier (see \cite[Definition 1.24 and Lemma 1.25]{Todorcevic2010}). Let $\mathcal{B}$ be a barrier on $M \in [\omega]^{\omega}$, and define $T(\mathcal{B}) \subseteq [M]^{<\omega}$ by

\kern-0.5em
\begin{equation*}
    T(\mathcal{B}) = \{ s\in [M]^{<\omega} : (\exists\, t\in \mathcal{B}) (s \sqsubseteq t) \} = \{ s\in [M]^{<\omega} : (\exists\, t\in \mathcal{B}) (s \subseteq t) \}.
\end{equation*}

\kern-0.5em
The family $T(\mathcal{B})$, ordered by the relation $\sqsubseteq$, is a well-founded tree on $M$, so it has no infinite branches. Let $\rho_{T(\mathcal{B})} : T(\mathcal{B}) \rightarrow \omega_{1}$ be the strictly decreasing map defined recursively by 

\kern-0.5em
\begin{equation*}
    \rho_{T(\mathcal{B})} (s) = \sup \{ \rho_{T(\mathcal{B})} (r) +1 : r\in T(\mathcal{B}) \wedge s\sqsubset r \}.
\end{equation*}

\kern-0.5em
The rank of the barrier $\mathcal{B}$ on $M$, denoted by $\textrm{rk}_{M}(\mathcal{B})$, is defined as the countable ordinal 

\kern-0.5em
\begin{equation*}
    \textrm{rk}_{M}(\mathcal{B}) = \rho_{T(\mathcal{B})} (\emptyset).
\end{equation*}

\kern-0.5em
Furthermore, the rank of the barrier $\mathcal{B}$ on $M$ admits a recursive characterization in terms of the ranks of the barriers $\mathcal{B}_{(n)}$ on $M/n$, as follows:

\kern-0.5em
\begin{equation*}
    \textrm{rk}_{M}(\mathcal{B}) = \sup \{ \textrm{rk}_{M/n}(\mathcal{B}_{(n)}) +1 : n\in M \}.
\end{equation*}

\kern-0.5em
We shall suppress the index $M$ in the notation for the rank of a barrier $\mathcal{B}$ on $M$, since it will be clear from the context; thus, we write $\textrm{rk}(\mathcal{B})$ instead of $\textrm{rk}_{M}(\mathcal{B})$.

\medskip

Observe that for any barrier $\mathcal{B}$ on $M\in [\omega]^{\omega}$, we have $\textrm{rk}(\mathcal{B}) \geq \textrm{rk}(\mathcal{B} \!\restriction\! N)$ for every $N \in [M]^{\omega}$. However, it is often convenient to work with barriers whose rank is preserved under taking restrictions. This motivates the notion of a uniform barrier (see \cite[Definition II.3.1 and Lemma II.3.3]{Todorcevic2005}), defined by induction on the rank as follows.

\begin{definition}
    Let $M\in [\omega]^{\omega}$ and let $\mathcal{B}$ be a barrier on $M$ with $\textrm{rk} (\mathcal{B}) = \alpha$. We say that $\mathcal{B}$ is \textit{uniform} provided that:
    \setlist{nolistsep}
    \begin{itemize}
    \setlength{\itemsep}{0pt} 
    \item If $\alpha = 1$, then $\mathcal{B} = [M]^{1}$.
    \item If $\alpha$ is a successor ordinal with $\alpha = \beta +1$ and $\beta \geq 1$, then for each $n\in M$ the barrier $\mathcal{B}_{(n)}$ on $M/n$ is uniform with $\textrm{rk} (\mathcal{B}_{(n)}) = \beta$.
    \item If $\alpha$ is a limit ordinal, then there exists an increasing sequence $\{\beta_{n}\}_{n\in M}$ with $\sup_{n\in M} \beta_{n} = \alpha$ such that for each $n\in M$ the barrier $\mathcal{B}_{(n)}$ on $M/n$ is uniform with $\textrm{rk} (\mathcal{B}_{(n)}) = \beta_{n}$. 
\end{itemize}    
\end{definition}

The invariance of the rank under restriction for a uniform barrier is the key feature of uniformity, meaning that for every uniform barrier $\mathcal{B}$ on $M\in [\omega]^{\omega}$ and every $N\in [M]^{\omega}$ we have $\textrm{rk}(\mathcal{B}) = \textrm{rk}(\mathcal{B} \!\restriction\! N)$.

\medskip

It is worth noting that the uniform barriers on $\omega$ of finite rank are essentially the families $[\omega]^n$ for $n \in \omega$ with $n \geq 1$. Likewise, observe that the Schreier barrier is a uniform barrier of rank $\omega$. In fact, it is well-known that for every countable ordinal $1 \leq \alpha < \omega_1$ there exists a uniform barrier $\mathcal{B}$ on $\omega$ such that $\textrm{rk}(\mathcal{B}) = \alpha$.

\medskip

As an additional remark, the barrier $[M]^{1}$ is referred to as the trivial barrier on $M \in [\omega]^{\omega}$. In what follows, we will implicitly assume that all barriers $\mathcal{B}$ are nontrivial, and thus $\textrm{rk}(\mathcal{B}) \geq 2$.

\medskip

Having briefly reviewed some notions related to barriers, we now state the Nash-Williams theorem
concerning finite colorings of barriers.

\begin{proposition}[\cite{NashWilliams}] \label{NashWilliams_Theorem}
   Let $\mathcal{B}$ be a barrier on $\omega$. For every $A\in [\omega]^{\omega}$, every $k\in\omega$ with $k \geq 2$, and every coloring $\pi: \mathcal{B}\!\restriction\!A \rightarrow k$, there is some $H\in [A]^{\omega}$ such that $\pi$ is constant on $\mathcal{B} \!\restriction\! H$.
\end{proposition}

Given a barrier $\mathcal{B}$ on $\omega$ and a coloring $\pi : \mathcal{B} \rightarrow k$, a set $H \in [\omega]^{\omega}$ is said to be homogeneous for $\pi$ if $\pi$ is constant on $\mathcal{B}\!\restriction\! H$. We denote by $\textrm{hom}(\pi)$ the collection of all infinite homogeneous sets for $\pi$, and observe that $\textrm{hom}(\pi) \subseteq [\omega]^{\omega}$ is a closed tall family that is closed under taking infinite subsets.

\medskip

Let $\mathcal{I}$ be an ideal on $\omega$, let $\mathcal{B}$ a barrier on $\omega$, and fix $k\in \omega$ with $k \geq 2$. The ideal $\mathcal I$ satisfies

 \kern-0.5em
\begin{equation*}
    \mathcal{I}^+ \longrightarrow (\mathcal{I}^+)^{\mathcal{B}}_{k}
\end{equation*}

 \kern-0.5em
if, for every $A\in \mathcal{I}^+$ and every coloring $\pi: \mathcal{B}\!\restriction\! A \rightarrow k$ of the barrier $\mathcal{B}\!\restriction\! A$ into $k$ colors, there exists some $H\in \mathcal{I}^+ \!\restriction\! A$ such that $H$ is homogeneous for $\pi$.

\medskip

The combinatorial notion of a Ramsey ideal can be naturally extended by requiring the existence of positive homogeneous sets for finite colorings of an arbitrary barrier $\mathcal{B}$ on $\omega$.

\begin{definition}
Let $\mathcal{B}$ be a barrier on $\omega$. An ideal $\mathcal{I}$ on $\omega$ is \textit{$\mathcal{B}$-Ramsey} if it satisfies $\mathcal{I}^+ \longrightarrow (\mathcal{I}^+)^{\mathcal{B}}_{2}$.
\end{definition}

By adapting the argument from the proof of Fact \ref{fact_colors} to the context of barriers, it follows that, given a barrier $\mathcal{B}$ on $\omega$, an ideal $\mathcal{I}$ on $\omega$ is $\mathcal{B}$-Ramsey if and only if $\mathcal{I}^+ \longrightarrow (\mathcal{I}^+)^{\mathcal{B}}_{k}$ holds for all $k \geq 2$.

\medskip

Clearly, an ideal is Ramsey if it is $[\omega]^{2}$-Ramsey; moreover, by Theorem \ref{Ramsey_ideal_theorem}, an ideal is Ramsey if it is $[\omega]^{n}$-Ramsey for each $n \in \omega$ with $n\geq 2$. Furthermore, every Galvin ideal is $\mathcal{B}$-Ramsey for all barriers $\mathcal{B}$ on $\omega$, since the Galvin lemma implies the Nash-Williams theorem. For the sake of completeness, we present a sketch of the proof of this last fact.

\begin{proposition} \label{Galvin implies B-Ramsey}
    Let $\mathcal{I}$ be a Galvin ideal on $\omega$. Then, $\mathcal{I}^+ \longrightarrow (\mathcal{I}^+)^{\mathcal{B}}_{2}$ holds for all barriers $\mathcal{B}$ on $\omega$.
\end{proposition}

\begin{proof}
    Let $\mathcal{B}$ be a barrier on $\omega$. Given $A\in \mathcal{I}^{+}$ and a coloring $\pi : \mathcal{B} \!\restriction\! A \rightarrow 2$, consider the family $\mathcal{F} \subseteq [\omega]^{<\omega}$ defined by $\mathcal{F} = \{ b \in \mathcal{B} \!\restriction\! A : \pi(b)=0  \}$. As the ideal $\mathcal{I}$ is Galvin, there exists $H\in \mathcal{I}^{+} \!\restriction\! A$ such that $H \in \textrm{hom} (\mathcal{F})$; in fact, we have that $H \in \textrm{hom} (\mathcal{\pi})$. Indeed, if every infinite subset of $H$ has an initial segment in $\mathcal{F}$, then $\pi \,\text{''}\, \mathcal{B} \!\restriction\! H = \{0\}$; on the other hand, if no finite subset of $B$ belongs to $\mathcal{F}$, then $\pi \,\text{''}\, \mathcal{B} \!\restriction\! H = \{1\}$. Therefore, $\mathcal{I}^+ \longrightarrow (\mathcal{I}^+)^{\mathcal{B}}_{2}$ holds for all barriers $\mathcal{B}$ on $\omega$.
\end{proof}

\smallskip

Given a barrier $\mathcal{B}$ on $M \in [\omega]^{\omega}$, any coloring $\pi : \mathcal{B} \rightarrow 2$ naturally induces a countable collection of colorings $\{ \pi_{(n)} : \mathcal{B}_{(n)} \rightarrow 2 \}_{n\in M}$ defined by $\pi_{(n)} (t) = \pi ( \{n\}\cup t )$ for each $n\in M$ and each $t\in \mathcal{B}_{(n)}$.

\begin{proposition} \label{B-Ramsey = version with diagonalizations}
Let $\mathcal I$ be an ideal on $\omega$, and let $\mathcal{B}$ be a barrier on $\omega$. Then, the following statements are equivalent:
\setlist{nolistsep}
\begin{itemize}
\setlength{\itemsep}{0pt}
\item[(a)] $\mathcal{I}$ is $\mathcal{B}$-Ramsey.
\item[(b)] For every coloring $\pi : \mathcal{B} \rightarrow 2$ and every $M\in\mathcal{I}^{+}$ there exists $H\in \mathcal{I}^{+} \!\restriction\! M$ such that $H/n \in \textrm{hom}(\pi_{(n)})$ for all $n\in H$. 
\end{itemize}
\end{proposition}

\begin{proof}
    $[\text{(a)} \Longrightarrow \text{(b)}].$ Suppose that $\mathcal{I}$ is $\mathcal{B}$-Ramsey, then for every coloring $\pi : \mathcal{B} \rightarrow 2$ and every $M\in\mathcal{I}^{+}$ there exists $H\in \mathcal{I}^{+} \!\restriction\! M$ such that $H \in \textrm{hom}(\pi)$, so that $\pi \,\text{''}\, \mathcal{B}\!\restriction\!H =\{i\}$ for some $i\in 2$. Consequently, for each $n\in H$ and each $t\in \mathcal{B}_{(n)} \!\restriction\! (H/n)$ we have that $\{n\}\cup t \in \mathcal{B} \!\restriction\! H$ and $\pi_{(n)} (t) = \pi (\{n\}\cup t) = i$, which implies that $\pi_{(n)} \,\text{''}\, \mathcal{B}_{(n)}\!\restriction\!(H/n) =\{i\}$. Therefore, $H/n \in \textrm{hom}(\pi_{(n)})$ for all $n\in H$. 

    \medskip

    $[\text{(b)} \Longrightarrow \text{(a)}].$  Let $\pi : \mathcal{B} \rightarrow 2$ be a coloring and let $M\in\mathcal{I}^{+}$, then there exists $H\in \mathcal{I}^{+} \!\restriction\! M$ such that $H/n \in \textrm{hom}(\pi_{(n)})$ for all $n\in H$. Thus, for each $n\in H$ there is $j_{n}\in 2$ such that  $\pi_{(n)} \,\text{''}\, \mathcal{B}_{(n)}\!\restriction\!(H/n) =\{j_{n}\}$. Partition $H=H_{0} \cup H_{1}$ into two disjoint sets, where $H_{i} = \{ n\in H : j_{n} = i\}$, then $H_{i} \in \mathcal{I}^{+} \!\restriction\! M$ for some $i\in 2$. Finally, observe that $H_{i} \in \textrm{hom} (\pi)$, since for each $s\in \mathcal{B} \!\restriction\! H_{i}$ there are $n\in H_{i}$ and $t\in \mathcal{B}_{(n)} \!\restriction\! (H_{i}/n)$ such that $s=\{n\}\cup t$, thus $\pi(s) = \pi (\{n\}\cup t) = \pi_{(n)}(t)=i$ and hence $\pi \,\text{''}\, \mathcal{B}\!\restriction\!H_{i} =\{i\}$. Therefore,  we conclude that $\mathcal{I}$ is $\mathcal{B}$-Ramsey.   
\end{proof}

\smallskip

Next, it should be noted that a slight modification of the arguments used in the construction of a Ramsey non-Galvin ideal in Theorem \ref{ramsey-no-galvin} also works to construct a $\mathcal{B}$-Ramsey non-Galvin ideal for an arbitrary barrier $\mathcal{B}$. This is made possible by the following extension of Lemma \ref{efective-Ramsey}, proposed in \cite[Corollary 3.8]{GrebikUzca2019}, which states that there is a Borel way to choose a homogeneous set for every coloring of a barrier.

\begin{lemma}[\cite{GrebikUzca2019}]
\label{efective-NW} 
Let $\mathcal{B}$ be a barrier on $\omega$. There
exists a Borel map $G: [\omega]^{\omega} \times 2^{\mathcal{B}} \rightarrow
[\omega]^{\omega}$ such that, for every $A \in [\omega]^{\omega}$ and every coloring $\pi : \mathcal{B} \rightarrow 2$, the infinite set $G(A,\pi)$ is a homogeneous subset of $A$ for $\pi$.
\end{lemma}

\begin{theorem} \label{B-ramsey-no-galvin}
   Let $\mathcal{B}$ be a barrier on $\omega$. There exists a $\mathcal{B}$-Ramsey tall ideal which is not Galvin. 
\end{theorem}

\begin{proof}
    The construction of this ideal closely parallels that of the Ramsey non-Galvin tall ideal presented in Theorem \ref{ramsey-no-galvin}, with only minor modifications. Indeed, fix a continuous surjection $\phi : 2^{\omega} \to 2^{\mathcal{B}}$, whose range is the collection of all colorings of $\mathcal{B}$ into two colors. By following exactly the same construction and arguments as in Theorem \ref{ramsey-no-galvin}, but using $\phi$ in place of $\varphi$ and replacing $F$ with the Borel map $G : [\omega]^{\omega} \times 2^{\mathcal{B}} \rightarrow [\omega]^{\omega}$ given by Lemma \ref{efective-NW}, we obtain a coanalytic $\mathcal{B}$-Ramsey tall ideal $\mathcal{I}$ on $\omega$ that is not Galvin.
\end{proof}

\smallskip

We continue by showing in Proposition \ref{rk(c)<rk(B)} that for any countable ordinal $2 \leq \alpha < \omega_1$, if an ideal is $\mathcal{B}$-Ramsey for some uniform barrier $\mathcal{B}$ of rank $\alpha$, then it is also $\mathcal{C}$-Ramsey for every barrier $\mathcal{C}$ of rank less than $\alpha$. To achieve this, we consider the following useful quasi-order relation on the class of all barriers, implicitly introduced and applied in \cite[Lemma II.3.12 and Corollary II.3.16]{Todorcevic2005}.

\begin{definition}
    Let $M\in [\omega]^{\omega}$, and let $\mathcal{B}$ and $\mathcal{C}$ be barriers on $M$. We write $\mathcal{C} \sqsubseteq \mathcal{B}$ if for every $b\in \mathcal{B}$ there exists $c\in \mathcal{C}$ such that $c\sqsubseteq b$.
\end{definition}

\begin{fact}
    Let $M\in [\omega]^{\omega}$, and let $\mathcal{B}$ and $\mathcal{C}$ be barriers on $M$ such that $\mathcal{C} \sqsubseteq \mathcal{B}$. Then, the following holds:
    \setlist{nolistsep}
    \begin{itemize}
    \setlength{\itemsep}{0pt}
    \item[(1)] For every $b\in \mathcal{B}$ there exists a unique $c_{b}\in \mathcal{C}$ such that $c_{b}\sqsubseteq b$.
    \item[(2)] For every $c\in \mathcal{C}$ there exists $b\in\mathcal{B}$ such that $c\sqsubseteq b$.
    \item[(3)] $\mathcal{C}\!\restriction\!N \sqsubseteq \mathcal{B}\!\restriction\!N$ for all $N\in [M]^{\omega}$.
    \item[(4)] $\textrm{rk}(\mathcal{C}) \leq \textrm{rk}(\mathcal{B})$.
\end{itemize}
\end{fact}

\begin{proof}
    Straightforward.
\end{proof}

\smallskip

The ideal Ramsey property associated with barriers is preserved downward under the relation $\sqsubseteq$. More precisely, the $\mathcal{B}$-Ramsey property of an ideal transfers to every barrier below $\mathcal{B}$ with respect to $\sqsubseteq$.

\begin{lemma} \label{Lemma_Barriers_1}
    Let $\mathcal{I}$ be an ideal on $\omega$, and let $\mathcal{B}$ and $\mathcal{C}$ be barriers on $\omega$ such that $\mathcal{C} \sqsubseteq \mathcal{B}$. If $\mathcal{I}$ is $\mathcal{B}$-Ramsey, then $\mathcal{I}$ is $\mathcal{C}$-Ramsey.
\end{lemma}
    
\begin{proof}
Suppose that the ideal $\mathcal{I}$ is $\mathcal{B}$-Ramsey, and let $\pi: \mathcal{C}\!\restriction\! A \rightarrow 2$ be a coloring, where $A\in \mathcal{I}^{+}$. Since $\mathcal{C}\!\restriction\! A \sqsubseteq \mathcal{B}\!\restriction\! A$, for every $b\in \mathcal{B}\!\restriction\! A$ there exists a unique $c_{b}\in \mathcal{C}\!\restriction\! A$ such that $c_{b}\sqsubseteq b$. Now, define the coloring $\hat{\pi}: \mathcal{B}\!\restriction\! A \rightarrow 2$ by $\hat{\pi} (b) = \pi(c_{b})$. As $\mathcal{I}$ is $\mathcal{B}$-Ramsey, there exists $H\in\mathcal{I}^{+} \!\restriction\! A$ that is homogeneous for $\hat{\pi}$, then $\hat{\pi} \,\text{''}\, \mathcal{B}\!\restriction\!H =\{i\}$ for some $i\in 2$. 

\smallskip

We claim that $H$ is also homogeneous for $\pi$. Indeed, for each $c\in \mathcal{C} \!\restriction\! H$ there exists $b\in \mathcal{B} \!\restriction\! H$ such that $c \sqsubseteq b$, then $c=c_{b}$ and hence $\pi(c) = \pi (c_{b}) = \hat{\pi} (b) =i$, which implies that $\pi \,\text{''}\, \mathcal{C}\!\restriction\!H =\{i\}$. Therefore, we conclude that the ideal $\mathcal{I}$ is $\mathcal{C}$-Ramsey.
\end{proof}

\smallskip

The next result describes how the ideal Ramsey property related to a uniform barrier interacts with the quasi-order $\sqsubseteq$ and the associated rank structure of barriers.

\begin{lemma}  \label{Lemma_Barriers_2}
    Let $\mathcal{I}$ be an ideal on $\omega$, and let $\mathcal{B}$ and $\mathcal{C}$ be barriers on $\omega$ such that $\mathcal{B}$ is uniform and $\textrm{rk}(\mathcal{C}) < \textrm{rk}(\mathcal{B})$. If $\mathcal{I}$ is $\mathcal{B}$-Ramsey, then for every $M\in\mathcal{I}^{+}$ there exists $N\in \mathcal{I}^{+} \!\restriction\! M$ such that $\mathcal{C} \!\restriction\! N \sqsubseteq \mathcal{B} \!\restriction\! N$.
\end{lemma}

\begin{proof}
    We begin by considering the coloring $\pi: \mathcal{B} \rightarrow 2$ defined as follows: for $b\in\mathcal{B}$, let $\pi(b)=0$ if and only if $b\subseteq c$ for some $c\in\mathcal{C}$. Since the ideal $\mathcal{I}$ is $\mathcal{B}$-Ramsey, it follows that for every $M\in\mathcal{I}^{+}$ there exists $N\in \mathcal{I}^{+} \!\restriction\! M$ such that $\pi$ is constant on $\mathcal{B}\!\restriction\!N$.

    \smallskip

    We claim that $\pi \,\text{''}\, \mathcal{B}\!\restriction\!N =\{1\}$. Suppose, toward a contradiction, that $\pi \,\text{''}\, \mathcal{B}\!\restriction\!N =\{0\}$. Since $\mathcal{B}\!\restriction\!N$ is a barrier, for every $c\in \mathcal{C}\!\restriction\!N$ there exists $b\in \mathcal{B}\!\restriction\!N$ such that either $c\sqsubseteq b$ or $b\sqsubseteq c$. As $\pi(b)=0$, there exists $c^{\prime}\in \mathcal{C}$ such that $b \subseteq c^{\prime}$. Since $\mathcal{C}$ is a barrier, it follows that $c \not\sqsubset b$, and therefore $b \sqsubseteq c$. As a result, we obtain that $\mathcal{B} \!\restriction\! N \sqsubseteq \mathcal{C} \!\restriction\! N$, which implies that $\textrm{rk}(\mathcal{B}\!\restriction\!N) \leq \textrm{rk}(\mathcal{C}\!\restriction\!N) \leq \textrm{rk}(\mathcal{C})$. By uniformity of the barrier $\mathcal{B}$, we have $\textrm{rk}(\mathcal{B}) = \textrm{rk}(\mathcal{B}\!\restriction\!N)$, and hence $\textrm{rk}(\mathcal{B}) \leq \textrm{rk}(\mathcal{C})$, a contradiction.

    \smallskip

    Finally, since $\mathcal{C}\!\restriction\!N$ is a barrier, for every $b\in \mathcal{B}\!\restriction\!N$ there exists $c\in \mathcal{C}\!\restriction\!N$ such that either $b\sqsubseteq c$ or $c\sqsubseteq b$. As $\pi(b)=1$, it follows that $b \not\subseteq c^{\prime}$ for every $c^{\prime}\in\mathcal{C}$. In particular, we have $b \not\sqsubseteq c$, and therefore $c \sqsubset b$. Consequently, we conclude that $\mathcal{C} \!\restriction\! N \sqsubseteq \mathcal{B} \!\restriction\! N$.  
\end{proof}

\smallskip

As a consequence of the previous lemmas, the ideal Ramsey property associated with a uniform barrier is preserved under passage to barriers of smaller rank, as made precise in the following proposition.

\begin{proposition} \label{rk(c)<rk(B)}
     Let $\mathcal{I}$ be an ideal on $\omega$, and let $\mathcal{B}$ be a uniform barrier on $\omega$. If $\mathcal{I}$ is $\mathcal{B}$-Ramsey, then $\mathcal{I}$ is $\mathcal{C}$-Ramsey for every barrier $\mathcal{C}$ on $\omega$ satisfying $\textrm{rk}(\mathcal{C}) < \textrm{rk}(\mathcal{B})$. 
\end{proposition}

\begin{proof}
    Suppose that $\mathcal{I}$ is a $\mathcal{B}$-Ramsey ideal on $\omega$ for some uniform barrier $\mathcal{B}$ on $\omega$. Let $\mathcal{C}$ be a barrier on $\omega$ with $\textrm{rk}(\mathcal{C}) < \textrm{rk}(\mathcal{B})$, and let $\pi : \mathcal{C} \!\restriction\! M \rightarrow 2$ be a coloring, where $M\in\mathcal{I}^{+}$.

    \smallskip

    By Lemma \ref{Lemma_Barriers_2}, there exists $N\in \mathcal{I}^{+} \!\restriction\! M$ such that $\mathcal{C} \!\restriction\! N \sqsubseteq \mathcal{B} \!\restriction\! N$. Since $\mathcal{I} \!\restriction\! N$ is $(\mathcal{B}\!\restriction\!N)$-Ramsey, it follows from Lemma \ref{Lemma_Barriers_1} that $\mathcal{I} \!\restriction\! N$ is also $(\mathcal{C}\!\restriction\!N)$-Ramsey. Thus, there exists $H\in \mathcal{I}^{+} \!\restriction\! N \subseteq \mathcal{I}^{+} \!\restriction\! M$ such that $H$ is homogeneous for $\pi \!\restriction\! (\mathcal{C}\!\restriction\!N)$, which implies that $\pi$ is constant on $\mathcal{C}\!\restriction\!H$. Therefore, we conclude that the ideal $\mathcal{I}$ is $\mathcal{C}$-Ramsey.
\end{proof}

\smallskip

As expected, the ideal Ramsey property relative to an arbitrary barrier strengthens the classical ideal Ramsey property, as shown in the following proposition.

\begin{proposition} \label{B-Ramsey implies Ramsey}
    Let $\mathcal{I}$ be an ideal on $\omega$, and let $\mathcal{B}$ be a uniform barrier on $\omega$. If $\mathcal{I}$ is $\mathcal{B}$-Ramsey, then $\mathcal{I}$ is Ramsey.
\end{proposition}

\begin{proof}
    Let $\mathcal{I}$ be a $\mathcal{B}$-Ramsey ideal for some uniform barrier $\mathcal{B}$ on $\omega$. If $\textrm{rk}(\mathcal{B}) = 2$ then $\mathcal{B} = [\omega]^{2}$ and hence $\mathcal{I}$ is Ramsey. If $\textrm{rk}(\mathcal{B}) > 2$, then $\mathcal{I}$ is Ramsey by Proposition \ref{rk(c)<rk(B)}.
\end{proof}

\smallskip

We conclude by emphasizing that the study of ideals associated with colorings of barriers deserves further investigation. In this framework, we pose several interconnected open questions concerning $\mathcal{B}$-Ramsey ideals for barriers $\mathcal{B}$.

\medskip

On the one hand, Proposition \ref{rk(c)<rk(B)} states that if an ideal is $\mathcal{B}$-Ramsey for a uniform barrier $\mathcal{B}$, then it is $\mathcal{C}$-Ramsey for every barrier $\mathcal{C}$ whose rank is strictly less than the rank of $\mathcal{B}$. However, it is unclear what happens when $\mathcal{B}$ and $\mathcal{C}$ have the same rank. Likewise, it is not known whether the converse of Proposition \ref{rk(c)<rk(B)} holds for uniform barriers. This motivates the following two open questions.

\begin{question} \label{Q1-barriers}
    Let $\mathcal{B}$ be a uniform barrier on $\omega$ with $\textrm{rk}(\mathcal{B}) \geq \omega$. Does every $\mathcal{B}$-Ramsey ideal $\mathcal{I}$ on $\omega$ have to be $\mathcal{C}$-Ramsey for all barriers $\mathcal{C}$ on $\omega$ satisfying $\mathrm{rk}(\mathcal{C}) = \mathrm{rk}(\mathcal{B})$?
\end{question}

\begin{question} \label{Q2-barriers}
    Do there exist uniform barriers $\mathcal{B}$ and $\mathcal{C}$ on $\omega$, with $\textrm{rk}(\mathcal{C}) < \textrm{rk}(\mathcal{B})$, and an ideal $\mathcal{I}$ on $\omega$ such that $\mathcal{I}$ is $\mathcal{C}$-Ramsey but fails to be $\mathcal{B}$-Ramsey?
\end{question}

Let $\mathcal{B}$ be a uniform barrier on $\omega$, and let $\mathcal{E}^{\mathcal{B}} \subseteq [\omega]^{<\omega}$ be defined by

\kern-0.5em
\begin{equation*}
     \mathcal{E}^{\mathcal{B}} = \{ \{n\} \cup t : n\in\omega \,\wedge\, t\in \mathcal{B}\!\restriction\!(\omega/n) \}.
\end{equation*}

\kern-0.5em
The family $\mathcal{E}^{\mathcal{B}}$ is a uniform barrier on $\omega$ with $\textrm{rk}(\mathcal{E}^{\mathcal{B}}) = \textrm{rk} (\mathcal{B})+1$; moreover, for each $n\in\omega$ we have $\mathcal{E}^{\mathcal{B}}_{(n)} = \mathcal{B}\!\restriction\!(\omega/n)$. In particular, note that if $\mathcal{B} = [\omega]^{n}$ for some $n\in\omega$ with $n\geq 2$, then $\mathcal{E}^{\mathcal{B}} = [\omega]^{n+1}$.

\medskip

With this in mind, we now restrict our attention to the canonical successor barrier $\mathcal{E}^{\mathcal{B}}$ of a uniform barrier $\mathcal{B}$, leading naturally to the following special case of Question \ref{Q2-barriers}.

\begin{question} \label{Q3-barriers}
    Let $\mathcal{B}$ be a uniform barrier on $\omega$ with $\textrm{rk}(\mathcal{B}) \geq \omega$. Does every $\mathcal{B}$-Ramsey ideal $\mathcal{I}$ on $\omega$ have to be $\mathcal{E}^{\mathcal{B}}$-Ramsey? 
\end{question}

On the other hand, Proposition~\ref{B-Ramsey implies Ramsey} shows that whenever a barrier $\mathcal{B}$ has infinite rank, any $\mathcal{B}$-Ramsey ideal is itself Ramsey. However, the converse implication is not known in general, which leads to the following question in particular.

\begin{question} \label{Q4-barriers}
    Does there exist a Ramsey tall ideal that is not $\mathcal{S}$-Ramsey, where $\mathcal{S}$ is the Schreier barrier?
\end{question} 

Observe that the previous question is closely related to Question \ref{question_almost_homogeneous}, since every coloring of the Schreier barrier induces a sequence of colorings of barriers of finite rank of the type considered there. Naturally, the notion of an almost homogeneous set for colorings of barriers of finite rank, introduced in Definition \ref{def_almosthomog}, can be extended to arbitrary barriers.

\begin{definition}
    Let $\mathcal{B}$ be a barrier on $\omega$, let $A\in [\omega]^{\omega}$, and let $\pi: \mathcal{B} \!\restriction\! A \rightarrow k$ be a finite coloring of $\mathcal{B} \!\restriction\! A$. An infinite subset $H$ of $A$ is almost homogeneous for $\pi$ if there is $m\in\omega$ such that $\mathcal{B} \!\restriction\! (H/m)$ is homogeneous for $\pi$.
\end{definition}

Furthermore, in the same spirit as Questions \ref{question_almost_homogeneous} and \ref{Q4-barriers}, one may ask whether an analogue of Theorem \ref{countable_colorings_of_n-sets} holds for uniform barriers of infinite rank.

\begin{question}
Let $\mathcal{B}$ be a uniform barrier on $\omega$ with $\textrm{rk}(\mathcal{B}) \geq \omega$. Given $A\in\mathcal{I}^{+}$ together with a countable collection $\{ \pi_{m}: \mathcal{B} \!\restriction\!A \rightarrow 2 \}_{m\in\omega}$ of colorings of $\mathcal{B} \!\restriction\! A$, does there exist $H \in \mathcal{I}^{+}\!\restriction\! A$ that is almost homogeneous for each $\pi_m$?
\end{question}

We close this section by introducing the combinatorial notion of a Nash-Williams ideal, which naturally corresponds to the ideal-theoretic version of the Nash-Williams theorem on colorings of barriers.

\begin{definition}
    Let $\mathcal{I}$ be an ideal on $\omega$. We say that $\mathcal{I}$ is Nash-Williams if it is $\mathcal{B}$-Ramsey for every barrier $\mathcal{B}$ on $\omega$.
\end{definition}

In view of the previous definition, it is useful to consider the following characterization of Nash-Williams ideals in terms of uniform barriers of all possible ranks.

\begin{proposition} \label{omega_1 barriers}
    Let $\mathcal{I}$ be an ideal on $\omega$, and for each countable ordinal $2\leq \alpha <\omega_{1}$, let $\mathcal{B}_{\alpha}$ be a uniform barrier of rank $\alpha$. Then, the following statements are equivalent:
\setlist{nolistsep}
\begin{itemize}
\setlength{\itemsep}{0pt}
\item[(a)] $\mathcal{I}$ is Nash-Williams.
\item[(b)] $\mathcal{I}$ is $\mathcal{B}_{\alpha}$-Ramsey for every $2\leq\alpha<\omega_{1}$.
\end{itemize}
\end{proposition}

\begin{proof}
    This fact follows directly from Proposition \ref{rk(c)<rk(B)}. 
\end{proof}

\smallskip

By virtue of Propositions \ref{Galvin implies B-Ramsey} and \ref{B-Ramsey implies Ramsey}, it follows that the notion of a Nash-Williams ideal determines an intermediate combinatorial property between Galvin ideals and Ramsey ideals. Nevertheless, at present we do not know whether Nash-Williams ideals coincide with either of these two classes of ideals.

\medskip

Since every Galvin ideal is Nash-Williams, the Galvin ideal constructed in Theorem \ref{galvin-no-semiselective} provides, in particular, an example of a Nash-Williams ideal that is not semiselective. We refine this construction to obtain an optimal example of a Nash-Williams non-semiselective ideal, directly using the Nash-Williams theorem on colorings of barriers stated in Proposition \ref{NashWilliams_Theorem}. However, it is not known whether the resulting ideal fails to be Galvin.

\begin{theorem} \label{NW-no-semiselective} 
There exists a Nash-Williams tall ideal which is not semiselective.
\end{theorem}

\begin{proof}
    The construction of this ideal follows a closely similar line of argument to that of the Galvin non-semiselective tall ideal presented in Theorem \ref{galvin-no-semiselective}, with appropriate modifications. Indeed, for each countable ordinal $2\leq \alpha <\omega_{1}$, let $\mathcal{B}_{\alpha}$ be a uniform barrier with $\textrm{rk}(\mathcal{B}_{\alpha}) = \alpha$. Fix an enumeration 
    
    \kern-0.5em
    \begin{equation*}
       \textstyle \bigcup_{\alpha} 2^{\mathcal{B}_{\alpha}} = \{ \pi_{\xi} : \xi \in 2^{\omega} \}
    \end{equation*}

    \kern-0.5em
    of all colorings from each of the barriers $\mathcal{B}_{\alpha}$ into two colors. Moreover, let $\eta: [\omega]^{\omega} \rightarrow [\omega]^{\omega}$ be a function such that $\eta(A) \subseteq A$ and $|A \setminus \eta(A)|=\omega$ for all $A\in [\omega]^{\omega}$. 

    \medskip

    Reasoning analogously to the proof of Theorem \ref{galvin-no-semiselective}, but applying the Nash-Williams theorem (see Proposition \ref{NashWilliams_Theorem}) in the argument, it is possible to recursively construct a family $\{ A_{u} : u\in (2^{\omega})^{<\omega} \} \subseteq [\omega]^{\omega}$, with $A_{\emptyset} =\omega$, such that for every $u\in (2^{\omega})^{<\omega}$, the collection $\{ A_{u^{\smallfrown}\xi} : \xi\in2^{\omega} \} \subseteq [A_{u} \setminus \eta(A_{u})]^{\omega}$ is an almost disjoint family of subsets of $A_{u} \setminus \eta(A_{u})$ such that $A_{u^{\smallfrown}\xi} \in \textrm{hom}(\pi_{\xi})$ for each $\xi \in 2^{\omega}$.

\medskip

From this construction, consider the coideal $\mathcal{I}^{+}$ on $\omega$ defined by

\kern-0.5em
\begin{equation*}
\mathcal{I}^{+} = \{ X\in [\omega]^{\omega} : (\exists\, u\in (2^{\omega})^{<\omega}) (A_{u}\subseteq^{*} X)\}, 
\end{equation*} 

\kern-0.5em
and conclude that $\mathcal{I} = \mathcal{P}(\omega)\setminus\mathcal{I}^{+}$ is a $\mathcal{B}_{\alpha}$-Ramsey non-semiselective tall ideal for each $2\leq \alpha <\omega_{1}$. Therefore, by Proposition \ref{omega_1 barriers}, it follows that $\mathcal{I}$ is a Nash-Williams non-semiselective tall ideal.
\end{proof}

\smallskip

It is worth noting that the construction of the Nash-Williams non-semiselective tall ideal given in the previous theorem cannot be adapted to produce a coanalytic ideal. Consequently, unlike in Theorem \ref{ramsey-no-galvin}, we cannot conclude that the resulting ideal is non-Galvin. The main obstruction lies in the fact that there is no Borel collection of uniform barriers $\{ \mathcal{B}_{\alpha} : 2 \leq \alpha < \omega{_1} \} \subseteq 2^{[\omega]^{<\omega}}$ such that $\textrm{rk}(\mathcal{B}_{\alpha}) = \alpha$, as will be shown in Proposition \ref{no-Borel-barriers}.

\medskip

Let $\mathcal{T}$ be the family of all trees on $\omega$, that is, all $T \subseteq [\omega]^{<\omega}$ such that for every $s, t \in [\omega]^{<\omega}$, if $s \sqsubseteq t$ and $t \in T$ then $s \in T$. Let $\mathcal{WFT}$ be the family of all well-founded trees on $\omega$, that is, trees $T \in \mathcal{T}$ such that there is no $A \in [\omega]^\omega$ with all of its initial segments belonging to $T$. It is well-known that $\mathcal{WFT}$ is a ${\mathbf\Pi}^{1}_{1}$-complete subset of $2^{[\omega]^{<\omega}}$ (see, for example, \cite[Theorem 27.1 and Subsection 32.B]{Kechris94}). 

\medskip

Furthermore, let $\mathfrak{F}$ be the collection of all fronts on $\omega$, namely, all $\mathcal{B} \subseteq [\omega]^{<\omega}$ that are antichains with respect to $\sqsubseteq$ and satisfy the property that every infinite subset of $\omega$ has an initial segment in $\mathcal{B}$. Note that $\mathfrak{F}$ is a coanalytic subset of $2^{[\omega]^{<\omega}}$.

\medskip

Although we have defined the notion of rank here only for barriers, it extends naturally to all fronts $\mathcal{B}$ on $\omega$ via the well-founded tree $T(\mathcal{B})$ obtained by taking the downward closure of $\mathcal{B}$ with respect to $\sqsubseteq$.

\begin{proposition} \label{complexity_fronts}
The collection $\mathfrak{F}$ of all fronts on $\omega$ is a ${\mathbf\Pi}^{1}_{1}$-complete subset of $2^{[\omega]^{<\omega}}$.
\end{proposition}

\begin{proof}
Let $\mathfrak{F}$ be the collection of all fronts on $\omega$, which is a coanalytic subset of $2^{[\omega]^{<\omega}}$.  To prove that $\mathfrak{F}$ is ${\mathbf\Pi}^{1}_{1}$-complete, it suffices to show  that $\mathcal{WFT}$ is Wadge reducible to $\mathfrak{F}$ (see, for example, \cite[Definition 21.13 and Remark following Theorem 27.4]{Kechris94}). 

\medskip

Consider the function $F: \mathcal{T} \rightarrow 2^{[\omega]^{<\omega}}$ defined by

\kern-0.5em
\begin{equation*}
F(T) =  \{a\cup \{n\} \in [\omega]^{<\omega} : a<n \;\wedge\; a\in T \;\wedge\; a\cup \{n\}\not\in T \},
\end{equation*}

\kern-0.5em
and observe that $F$ is a continuous map such that $F(T)$ is an antichain with respect to $\sqsubseteq$ for every tree $T \in \mathcal{T}$. Next, we show that $F$ is the required reduction.

\medskip

On the one hand, let $T \in \mathcal{WFT}$. Since $T$ is a well-founded tree on $\omega$, for every $A\in [\omega]^{\omega}$ there exists a maximal initial segment $a\sqsubset A$ such that $a\in T$. Thus, if $n\in\omega$ satisfies $a<n$ and $a\cup\{n\} \sqsubset A$, then necessarily $a\cup\{n\} \notin T$, and hence $a\cup\{n\}\in F(T)$. Consequently, $F(T)$ is a front on $\omega$.

\medskip

On the other hand, let $T \in \mathcal{T} \setminus \mathcal{WFT}$. Since $T$ is a tree on $\omega$ that is not well-founded, there exists an infinite branch $A\in[\omega]^{\omega}$ through $T$, then $a\in T$ for every initial segment $a\sqsubset A$, and thus $a\notin F(T)$ for every $a \sqsubset A$. Consequently, $F(T)$ is not a front on $\omega$.

\medskip

Therefore, we conclude that $T\in \mathcal{WTF}$ if and only if $F(T) \in \mathfrak{F}$, which implies that the collection $\mathfrak{F}$ of all fronts on $\omega$ is a ${\mathbf\Pi}^{1}_{1}$-complete set.
\end{proof}

\begin{proposition} \label{no-Borel-barriers}
    There is no Borel collection of uniform barriers $\{ \mathcal{B}_{\alpha} : 2 \leq \alpha < \omega{_1} \} \subseteq 2^{[\omega]^{<\omega}}$ such that $\textrm{rk}(\mathcal{B}_{\alpha}) = \alpha$.
\end{proposition}

\begin{proof}
It is a classical fact that the usual rank $\rho$ on $\mathcal{WFT}$ is a coanalytic rank (see, for example, \cite[Exercise 34.6(i)]{Kechris94}). We have defined the rank of a front $\mathcal{B}$ on $\omega$ by $\textrm{rk}(\mathcal{B}) = \rho_{T(\mathcal{B})} (\emptyset)$, where $T(\mathcal{B})$ is the well-founded tree on $\omega$ associated with $\mathcal{B}$. Since the function $\mathcal{B}\mapsto T(\mathcal{B})$ is a Borel map from $\mathfrak{F}$ to $\mathcal{WFT}$, it follows that $\textrm{rk}$ is also a coanalytic rank on $\mathfrak{F}$. 

\medskip
    
As a consequence of Proposition \ref{complexity_fronts}, and by applying the boundedness theorem for coanalytic ranks (see, for example, \cite[Theorem 35.23]{Kechris94}), we conclude that any Borel collection of fronts or barriers is necessarily bounded in rank.    
\end{proof}

\smallskip

Finally, the following diagram summarizes the relationships and implications among the combinatorial notions of ideals discussed throughout this section:

\kern-0.5em
\begin{equation*}
\mbox{Semiselective}
\;\;\;\Longrightarrow\;\;\;
\mbox{Galvin}
\;\;\;\Longrightarrow\;\;\;
\mbox{Nash-Williams}
\;\;\;\Longrightarrow\;\;\;
\mbox{Ramsey}
\end{equation*}

\kern-0.5em
However, some natural reverse implications remain open and lead to the following two questions, which conclude this section.

\begin{question}
   Does there exist a Ramsey tall ideal which fails to be Nash-Williams?
\end{question}

\begin{question}
    Does there exist a Nash-Williams tall ideal which fails to be Galvin?
\end{question}

\section{Remarks on non-hereditary Ramsey ideals}
\label{Remarks on non-hereditary Ramsey ideals}

In this section, we discuss several remarks concerning non-hereditary Ramsey-type properties of tall ideals, in analogy with those considered in the preceding sections for the hereditary setting.

\subsection{Ideals satisfying a non-hereditary Ramsey property}

Let $\mathcal{I}$ be an ideal on $\omega$ and fix $n,k\in \omega$ with $n,k \geq 2$. The ideal $\mathcal I$ satisfies

\kern-0.5em
\begin{equation*}
    \omega \longrightarrow (\mathcal{I}^+)^{n}_{k}
\end{equation*}

\kern-0.5em
if for every coloring $\pi:[\omega]^{n} \rightarrow k$ there exists some $H\in \mathcal{I}^{+}$ such that $H$ is homogeneous for $\pi$.

\begin{fact} \label{direct_implications}
Let $\mathcal{I}$ be an ideal on $\omega$. Then, the following holds:
\begin{itemize}
\setlength{\itemsep}{0pt} 
\item[(1)] $\omega \longrightarrow (\mathcal{I}^+)^{n}_{k^{\prime}}$ implies $\omega \longrightarrow (\mathcal{I}^+)^{n}_{k}$ for all $n\geq 2$ and all $k^{\prime}\geq k \geq 2$.
\item[(2)] $\omega \longrightarrow (\mathcal{I}^+)^{n^{\prime}}_{k}$ implies $\omega \longrightarrow (\mathcal{I}^+)^{n}_{k}$ for all $k\geq 2$ and all $n^{\prime} \geq n \geq 2$.
\end{itemize}    
\end{fact}

\begin{proof}
    Straightforward.
\end{proof}

\smallskip

Recall that, given a coloring $\pi:[\omega]^{n} \rightarrow k$, we denote by $\textrm{hom}(\pi) \subseteq [\omega]^{\omega}$ the collection of all infinite homogeneous sets for $\pi$, which is a closed tall family that is closed under taking infinite subsets.

\medskip

In view of the previous observation, note that the class of non-tall ideals satisfies the combinatorial property $\omega \longrightarrow (\mathcal{I}^+)^{n}_{k}$ for all $n,k\geq 2$. Indeed, if an ideal $\mathcal{I}$ on $\omega$ is not tall, then $\textrm{hom}(\pi) \not\subseteq \mathcal{I}$ for every coloring $\pi:[\omega]^{n} \rightarrow k$, and hence there exists $H\in \mathcal{I}^{+}$ such that $H \in \textrm{hom}(\pi)$.

\medskip

It should be emphasized that one of the fundamental characterizations of ideals satisfying the partition property $\omega \longrightarrow (\mathcal{I}^+)^{2}_{2}$ is formulated in terms of the Kat\v{e}tov order and its interaction with the random graph ideal $\mathcal{R}$, as shown by Hru\v{s}\'{a}k in \cite[Proposition 3.11]{Hrusak2011} (see also \cite[Section 3, Item 6]{Hrusak2017}, \cite[Theorem 4.1]{HMTU2017}, and \cite[Theorem 2.2.2]{Meza}).

\begin{proposition} [\cite{Hrusak2011}]  \label{Ramsey(omega)_and_RandomGraph}
    Let $\mathcal I$ be an ideal on $\omega$. Then, the following statements are equivalent:
\setlist{nolistsep}
\begin{itemize}
\setlength{\itemsep}{0pt} 
\item[(a)] $\omega \longrightarrow (\mathcal{I}^+)^{2}_{2}$ holds.
\item[(b)] $\mathcal{R} \not\leq_{K} \mathcal{I}$.
\end{itemize}
\end{proposition}

We now review the non-hereditary versions of the properties $\textrm{h-Mon}$ and \textrm{h-FinBW} for ideals, originally introduced in \cite{FMRS07, FMRS11}. Given an ideal $\mathcal{I}$ on $\omega$, we say that:
\setlist{nolistsep}
\begin{itemize}
\setlength{\itemsep}{0pt} 
\item $\mathcal{I}$ is $\textrm{Mon}$ if for every sequence of real numbers $\{ x_{n} \}_{n\in \omega} \subseteq \mathbb{R}$, there exists some $B\in \mathcal{I}^{+}$ such that the subsequence $\{ x_{n} \}_{n\in B}$ is monotone.
\item $\mathcal{I}$ is $\textrm{FinBW}$ if for every bounded sequence of reals numbers $\{ x_{n} \}_{n\in\omega} \subseteq \mathbb{R}$, there exists some $B\in \mathcal{I}^{+}$ such that the subsequence $\{ x_{n} \}_{n\in B}$ is convergent.
\end{itemize}

\medskip

The properties $\textrm{Mon}$ and $\textrm{FinBW}$ for ideals, and their interaction with the property $\omega \longrightarrow (\mathcal{I}^+)^{2}_{2}$, were studied in \cite{FMRS11}. In particular, the following implications were established in \cite[Fact 3.1]{FMRS11}:

\kern-0.5em
\begin{equation*}
\omega \longrightarrow (\mathcal{I}^+)^{2}_{2} \;\;\;\Longrightarrow\;\;\; 
\textrm{Mon} 
\;\;\;\Longrightarrow\;\;\; 
\textrm{FinBW}
\end{equation*}

\kern-0.5em
However, the above implications are not reversible. In \cite[1st Remark of Section 3]{FMRS11}, it is claimed that the $F_\sigma$ tall ideal $\mathcal{I}_{1/n}$ is $\textrm{FinBW}$ but not $\textrm{Mon}$. Moreover, \cite[Theorem 1.4]{Kwela} provides an $F_{\sigma}$ tall ideal $\mathcal{WR}$ that is $\textrm{Mon}$ but $\omega \centernot\longrightarrow (\mathcal{WR}^+)^{2}_{2}$ holds.

\medskip

Naturally, we may also consider non-hereditary versions of the properties $\textrm{h-Mon}_{n}$ and $\textrm{h-FinBW}_{n}$ for each $n\geq 1$, thereby extending the properties $\textrm{Mon}$ and $\textrm{FinBW}$ to higher dimensions.

\begin{definition}
    Let $\mathcal{I}$ be an ideal on $\omega$ and fix $n\in \omega$ with $n \geq 1$. We say that:
    \setlist{nolistsep}
    \begin{itemize}
    \setlength{\itemsep}{0pt} 
    \item $\mathcal{I}$ is $\textrm{Mon}_{n}$ if for every sequence of real numbers $f:[\omega]^{n} \rightarrow \mathbb{R}$, there exists some $B\in \mathcal{I}^{+}$ such that the subsequence $f \!\restriction\! [B]^{n}$ is monotone. 
    \item $\mathcal{I}$ is $\textrm{FinBW}_{n}$ if for every bounded sequence of reals numbers $f: [\omega]^{n} \rightarrow \mathbb{R}$, there exists some $B\in \mathcal{I}^{+}$ such that the subsequence $f \!\restriction\! [B]^{n}$ is convergent.
    \end{itemize}
\end{definition}

Certainly, as in the hereditary setting, we have that the property $\textrm{FinBW}_n$ holds in an analogous form for every uncountable compact metrizable space, in the same spirit as Fact \ref{fact_(X,I)}.

\medskip

Furthermore, we also extend to higher dimensions the combinatorial notions of systems of subsets with pseudo-intersections and diagonalizations introduced in Definitions \ref{def_system_1} and \ref{def_system_2}.

\begin{definition} \label{def_systems_dim_n}
    Let $\mathcal{I}$ be an ideal on $\omega$ and fix $n\in\omega$ with $n\geq 1$. A system of subsets of $[\omega]^{n}$ is a collection $\{ \mathcal{A}_{s} : s\in \omega^{<\omega} \} \subseteq \wp(\omega)$, with $\mathcal{A}_{\emptyset} = [\omega]^{n}$, for which there exists a sequence $\{ a_{k} \}_{k\in \omega} \subseteq \omega$, with each $a_{k} \geq 2$, such that for every $s\in \omega^{<\omega}$ both $\mathcal{A}_s=\bigcup_{i<a_{|s|}} \mathcal{A}_{s^{\smallfrown} i}$ and $\mathcal{A}_{s^{\smallfrown} i} \cap \mathcal{A}_{s^{\smallfrown} j}=\emptyset$ if $i< j < a_{|s|}$.
\end{definition}

\begin{definition} \label{Def SPI_n & SD_n}
    Given an ideal $\mathcal{I}$ on $\omega$ and $n\in\omega$ with $n\geq 1$, we say that:
    
    \setlist{nolistsep}
    \begin{itemize}
    \setlength{\itemsep}{0pt} 
    \item $\mathcal{I}$ is $\textrm{SPI}_{n}(\omega)$ if for each system $\{ \mathcal{A}_{s} : s\in \omega^{<\omega} \}$ of subsets of $[\omega]^{n}$, there exist $B\in \mathcal{I}^{+}$ and $x \in \prod_{k\in\omega} a_{k}$ such that for every $k\in \omega$ there is $m_k \in \omega$ for which $[B/m_k]^{n} \subseteq \mathcal{A}_{x \restriction k}$.
    
    \item $\mathcal{I}$ is $\textrm{SD}_{n}(\omega)$ if for each system $\{ \mathcal{A}_{s} : s\in \omega^{<\omega} \}$ of subsets of $[\omega]^{n}$, there exist $D\in \mathcal{I}^{+}$ and $x \in \prod_{k\in\omega} a_{k}$ such that $[D/\ell]^{n} \subseteq \mathcal{A}_{x \restriction (\ell+1)}$ for all $\ell\in D$.
    \end{itemize}  
\end{definition}

Having introduced the preceding non-hereditary properties of ideals, we now establish the following result, which naturally generalizes \cite[Theorem 3.11]{FMRS11} to higher dimensions. It is important to note that the proofs of the implications stated below are analogous to those of similar results in the hereditary context discussed in Section \ref{An ideal version of Ramsey's theorem}. For this reason, we only sketch the main arguments in the proof of the following proposition.

\begin{proposition} \label{equivalences}
Let $\mathcal{I}$ be an ideal on $\omega$ that is $q^+$. For every $n\in\omega$ with $n\geq 1$, the following statements are equivalent:
\setlist{nolistsep}
\begin{itemize}
\setlength{\itemsep}{0pt}
\item[(a)] $\omega \longrightarrow (\mathcal{I}^+)^{n+1}_{k}$ holds for all $k\geq2$.
\item[(b)] $\mathcal{I}$ is $\textrm{Mon}_n$.
\item[(c)] $\mathcal{I}$ is $\textrm{FinBW}_n$.
\item[(d)] $\mathcal{I}$ is $\textrm{SPI}_{n}(\omega)$.
\item[(e)] $\mathcal{I}$ is $\textrm{SD}_{n}(\omega)$. 
\end{itemize}
\end{proposition}

\begin{proof}
$[\text{(a)} \Longrightarrow \text{(b)}].$ This fact is proved exactly as in the hereditary case treated in Proposition \ref{Ramsey=h-Mon_n}.

\medskip

$[\text{(b)} \Longrightarrow \text{(c)}].$ This fact is proved exactly as in the hereditary case considered in Lemma \ref{h-Mon_n implies h-FinBW_n}.

\medskip

$[\text{(c)} \Longrightarrow \text{(d)}].$ Suppose that $\mathcal{I}$ is $\textrm{FinBW}_{n}$, and let $\{ \mathcal{A}_{s} : s\in \omega^{<\omega} \}$ be a system of subsets of $[\omega]^{n}$ as in Definition  \ref{def_systems_dim_n}. Observe that for each $t\in [\omega]^{n}$ there exists a unique $f(t) \in \prod_{k\in\omega} a_{k}$ such that $t \in \bigcap_{k\in\omega} \mathcal{A}_{f(t) \restriction k}$. Applying the property $\textrm{FinBW}_{n}$ to the sequence $f : [\omega]^{n} \rightarrow \prod_{k\in\omega} a_{k}$, there exist $B\in \mathcal{I}^{+}$ and $x\in \prod_{k\in\omega} a_{k}$ such that the subsequence $f \!\restriction\! [B]^{n}$ converges to $x$. Finally, conclude that $[B/m_{k}]^{n} \subseteq \mathcal{A}_{x \restriction k}$ for all $k\in\omega$. Therefore, $\mathcal{I}$ is $\textrm{SPI}_{n}(\omega)$.

\medskip

$[\text{(d)} \Longrightarrow \text{(e)}].$ Suppose that $\mathcal{I}$ is $\textrm{SPI}_{n}(\omega)$. Let $\{ \mathcal{A}_{s} : s\in \omega^{<\omega} \}$ be a system of subsets of $[\omega]^{n}$ as in Definition  \ref{def_systems_dim_n}, and fix $B\in \mathcal{I}^{+}$ and $x \in \prod_{k\in\omega} a_{k}$ as in the first item of Definition \ref{Def SPI_n & SD_n}. Consider a strictly increasing sequence $\{m_{k}\}_{k\in\omega} \subseteq \omega$, with $m_{0}=0$, such that $[B/(m_{k+1}-1)]^{n} \subseteq \mathcal{A}_{x \restriction m_{k}}$ for all $k\in\omega$. Now, for each $k\in\omega$ define $f_{k} = B \cap \{ m_{k}, \ldots, m_{k+1}-1 \}$; thus, by property $q^{+}$, there exists $D\in \mathcal{I}^{+} \!\restriction\! B$ such that $|D\cap f_{k}| \leq 1$ for all $k\in\omega$. Next, split $D=D_{0} \cup D_{1}$ where $D_{i} = \bigcup_{j\in\omega} (D\cap f_{2j+i})$ for $i\in 2$, then at least one of these two sets belongs to $\mathcal{I}^{+}$. Finally, conclude that $[D_{i}/\ell]^{n} \subseteq \mathcal{A}_{x \restriction (\ell+1)}$ for all $\ell \in D_{i}$. Therefore, $\mathcal{I}$ is $\textrm{SD}_{n}(\omega)$.   

\medskip

$[\text{(e)} \Longrightarrow \text{(a)}].$ Suppose that $\mathcal{I}$ is $\textrm{SD}_{n}(\omega)$, and let $\pi :[\omega]^{n+1} \rightarrow k$ be a coloring. Recursively construct a system $\{ \mathcal{A}_{s} : s\in k^{<\omega} \}$ of subsets of $[\omega]^{n}$, with $\mathcal{A}_{\emptyset} = [\omega]^{n}$, as follows: Given $s\in k^{<\omega}$ for which $\mathcal{A}_{s}$ has already been constructed, for each $i\in k$ set $\mathcal{A}_{s^{\smallfrown}i} = \{ t\in\mathcal{A}_{s} : |s|<\min{t} \,\wedge\, \pi( \{ |s| \}\cup t ) =i \}$; moreover, for each $t\in[\omega]^{n}$ with $\min t \leq |s|$, include $t$ in any $\mathcal{A}_{s^{\smallfrown}i}$ whenever $t\in\mathcal{A}_{s}$.

\medskip

By property $\textrm{SD}_{n}(\omega)$, there exist $H\in \mathcal{I}^{+}$ and $x\in k^{\omega}$ such that $[H/\ell]^{n} \subseteq \mathcal{A}_{x\restriction(\ell+1)}$ for all $\ell\in H$. Now, partition $H= \bigcup_{i\in k} H_{i}$ where $H_{i} = \{ \ell\in H : x(\ell)=i \}$ for $i\in k$, then at least one $H_{i}$ belongs to $\mathcal{I}^{+}$. Finally, conclude that $H_{i} \in \textrm{hom}(\pi)$. Therefore, $\omega \longrightarrow (\mathcal{I}^+)^{n+1}_{k}$ holds for all $k\geq2$.
\end{proof}

\smallskip

It is worth making a few remarks about the proof of the previous proposition. First, observe that the assumption that the ideal $\mathcal{I}$ is $q^+$ is used only to show that, under this hypothesis, the property $\textrm{SPI}_n(\omega)$ implies $\textrm{SD}_n(\omega)$, whose proof is analogous to that used in the proof of the first implication of Proposition \ref{SD}. Moreover, an argument analogous to that used for proving the other implication of Proposition \ref{SD} shows that $\mathrm{SD}_{n}(\omega)$ implies $\omega \longrightarrow (\mathcal{I}^+)^{n+1}_{k}$ for all $k\geq 2$. On the other hand, note that the property $\textrm{Mon}_{n}$ follows from the property $\omega \longrightarrow (\mathcal{I}^+)^{n+1}_{2}$. Finally, to prove that $\textrm{FinBW}_n$ implies $\textrm{SPI}_n(\omega)$, we employ the same idea as in the proof of the first implication of Lemma \ref{SPI}; in fact, by adapting the argument of the second implication of that lemma, we conclude that the properties $\textrm{FinBW}_n$ and $\textrm{SPI}_n(\omega)$ are equivalent.

\medskip

In light of the above, we deduce that, for all $n,k\in\omega$ with $n\geq 1$ and $k \geq 2$, the following implications among combinatorial properties of ideals hold:

\kern-0.5em
\begin{equation*}
\textrm{SD}_{n}(\omega)
\;\;\;\Longrightarrow\;\;\;
\omega \longrightarrow (\mathcal{I}^+)^{n+1}_{k}
\;\;\;\Longrightarrow\;\;\;
\omega \longrightarrow (\mathcal{I}^+)^{n+1}_{2} \;\;\;\Longrightarrow\;\;\; 
\textrm{Mon}_n 
\;\;\;\Longrightarrow\;\;\; 
\textrm{FinBW}_n
\end{equation*}

\kern-0.5em
By Proposition \ref{equivalences}, under the property $q^{+}$, all these implications are reversible, and thus all these properties of ideals turn out to be equivalent.

\medskip

Regarding the convergence property $\textrm{FinBW}_{n}$, we now emphasize that, for all $n,k \geq 2$, the following implications among combinatorial properties of ideals hold:

\kern-0.5em
\begin{equation*}
\omega \longrightarrow (\mathcal{I}^+)^{n+1}_{2}
\;\;\;\Longrightarrow\;\;\;
\textrm{FinBW}_{n}
\;\;\;\Longrightarrow\;\;\;
\omega \longrightarrow (\mathcal{I}^+)^{n}_{k}
\end{equation*}

\begin{lemma} \label{FinBW_n-lemma}
    Let $\mathcal{I}$ be an ideal on $\omega$ and let $n\in\omega$ with $n\geq 2$. If $\mathcal{I}$ is $\textrm{FinBW}_{n}$, then $\omega \longrightarrow (\mathcal{I}^+)^{n}_{k}$ holds for all $k\geq 2$.
\end{lemma}

\begin{proof}
    Suppose that $\mathcal{I}$ is $\textrm{FinBW}_{n}$. Since every coloring $\pi : [\omega]^{n} \rightarrow k$ can be viewed as a bounded sequence $\pi :[\omega]^{n} \rightarrow \mathbb{R}$ with image contained in the discrete set $\{0,\ldots,k-1\}$, then there exists $B\in \mathcal{I}^{+}$ such that the subsequence $\pi \!\restriction\! [B]^{2}$ converges to some $p\in k$. It follows that $\pi \text{\,''\,} [B/m]^{2} = \{p\}$ for some $m\in\omega$ and hence $B/m \in \textrm{hom} (\pi)$. Therefore, we conclude that $\omega \longrightarrow (\mathcal{I}^+)^{n}_{k}$ holds.   
\end{proof}

\smallskip

The previous lemma implies that the class of partition properties of the form $\omega \longrightarrow (\mathcal{I}^{+})^{n}_{k}$ is governed by the lexicographic order on $\omega^{2}$ induced by the parameters $n$ and $k$, as stated in the following proposition.

\begin{proposition} \label{Ramsey(omega)_lex}
    Let $\mathcal{I}$ be an ideal on $\omega$, and let $n, n^{\prime}, k, k^{\prime} \in\omega$ with $n, n^{\prime}, k, k^{\prime} \geq 2$ be such that $(n,k) \leq_{\textrm{lex}} (n^{\prime},k^{\prime})$. If $\omega \longrightarrow (\mathcal{I}^{+})^{n^{\prime}}_{k^{\prime}}$ holds, then $\omega \longrightarrow (\mathcal{I}^{+})^{n}_{k}$ also holds. 
\end{proposition}

\begin{proof}
    Suppose that $\mathcal{I}$ satisfies $\omega \longrightarrow (\mathcal{I}^{+})^{n^{\prime}}_{k^{\prime}}$ and let $(n,k) \leq_{\textrm{lex}} (n^{\prime},k^{\prime})$. On the one hand, if $n=n^{\prime}$ and $k\leq k^{\prime}$, then Fact \ref{direct_implications} yields that $\omega \longrightarrow (\mathcal{I}^{+})^{n}_{k}$ holds. On the other hand, if $n < n^{\prime}$, then Fact \ref{direct_implications} implies that $\omega \longrightarrow (\mathcal{I}^{+})^{n+1}_{k}$ holds; now, applying the two first implications of Proposition \ref{equivalences}, we obtain that $\mathcal{I}$ is $\textrm{FinBW}_{n}$, and by virtue of Lemma \ref{FinBW_n-lemma} we conclude that $\omega \longrightarrow (\mathcal{I}^{+})^{n}_{k}$ holds. 
\end{proof}

\smallskip

We conclude by mentioning that the following questions appear to be of particular interest and remain open at present.

\begin{question}
    Is there an ideal $\mathcal{I}$ on $\omega$ such that $\omega \longrightarrow (\mathcal{I}^+)^{2}_{k}$ holds for all $k\geq 2$ but it is not $\textrm{FinBW}_{2}$?
\end{question}

\begin{question}
    Is there an ideal $\mathcal{I}$ on $\omega$ that is $\textrm{FinBW}_{2}$ but not $\textrm{Mon}_{2}$?
\end{question}

\begin{question}
    Is there an ideal $\mathcal{I}$ on $\omega$ that is $\textrm{Mon}_{2}$ but $\omega \centernot\longrightarrow (\mathcal{I}^+)^{3}_{2}$ holds?
\end{question}

\subsection{Some counterexamples}

Let us recall that, in Theorem \ref{Ramsey_ideal_theorem}, we have proved that for an ideal $\mathcal{I}$ on $\omega$, the property $\i^+ \longrightarrow (\i^+)^2_2$ is equivalent to $\mathcal{I}^+ \longrightarrow (\mathcal{I}^+)^n_k$ for every $n,k \geq 2$. In contrast, this combinatorial collapse does not occur for the non-hereditary Ramsey property.

\medskip

Indeed, in the context of ideals satisfying $\omega \longrightarrow (\mathcal{I}^+)^{2}_{2}$, \cite[Theorem 5.2]{HMTU2017} establishes that the properties $\omega \longrightarrow (\mathcal{I}^+)^{2}_{2}$ and $\omega \longrightarrow (\mathcal{I}^+)^{2}_{3}$ are not equivalent; moreover, \cite[Theorem 3.1]{PG} extends this result by proving that, for all $k\geq 2$, the properties $\omega \longrightarrow (\mathcal{I}^+)^{2}_{k}$ and $\omega \longrightarrow (\mathcal{I}^+)^{2}_{k+1}$ are not equivalent. On the other hand, \cite[Theorem 6.6]{LA-OP-UA} shows that there is no $k\geq 2$ such that the properties $\omega \longrightarrow (\mathcal{I}^+)^{2}_{k}$ and $\omega \longrightarrow (\mathcal{I}^+)^{3}_{2}$ are equivalent.

\begin{proposition}[\cite{HMTU2017, PG}] \label{counterexample1}
    Let $k\in\omega$ with $k\geq2$. There exists an $F_{\sigma}$ tall ideal $\mathcal{I}$ on $\omega$ such that $\omega \longrightarrow (\mathcal{I}^+)^{2}_{k}$ holds but $\omega \centernot\longrightarrow (\mathcal{I}^+)^{2}_{k+1}$. In fact, this ideal $\mathcal{I}$ has the form $\mathcal{I}_{\textrm{hom}(\pi)}$ for some coloring $\pi: [\omega]^{2} \rightarrow k+1$.
\end{proposition}

\begin{proposition}[\cite{LA-OP-UA}] \label{counterexample2}
   There exists an $F_{\sigma}$ tall ideal $\mathcal{I}$ on $\omega$ such that $\omega \longrightarrow (\mathcal{I}^+)^{2}_{k}$ holds for all $k\geq 2$ but $\omega \centernot\longrightarrow (\mathcal{I}^+)^{3}_{2}$. In fact, this ideal $\mathcal{I}$ has the form $\mathcal{I}_{\textrm{hom}(\pi)}$ for some coloring $\pi: [\omega]^{3} \rightarrow 2$.
\end{proposition}

Based on the preceding results, we will now present new examples of ideals that exhibit the same combinatorial properties as those described in Propositions \ref{counterexample1} and \ref{counterexample2}.

\medskip

Recall that the random graph ideal $\mathcal{R}$ admits the following higher-dimensional extensions (see \cite[Section 3]{PG}): For $n,k\in \omega$ with $n,k\geq 2$, let $\mathcal{R}^n_k = \mathcal{I}_{\textrm{hom}(u^n_k)}$, where $u^{n}_{k} : [\omega]^{n} \rightarrow k$ is a coloring satisfying the property that for every collection of pairwise disjoint finite sets $F_{0}, \ldots, F_{k-1} \subseteq [\omega]^{n-1}$, there exists $\ell\in\omega$ such that $u^{n}_{k} (x\cup \{l\}) =i$ for every $x\in F_{i}$ and every $i\in k$. In \cite[Proposition 3.4]{PG}, the following generalization of Proposition \ref{Ramsey(omega)_and_RandomGraph} is stated.


\begin{proposition}[\cite{PG}]  \label{Ramsey(omega)_and_Random^n_k}
    Let $\mathcal I$ be an ideal on $\omega$. Then, for all $n,k\in\omega$ with $n,k\geq 2$, the following statements are equivalent:
\setlist{nolistsep}
\begin{itemize}
\setlength{\itemsep}{0pt} 
\item[(a)] $\omega \longrightarrow (\mathcal{I}^+)^{n}_{k}$ holds.
\item[(b)] $\mathcal{R}^{n}_{k} \not\leq_{K} \mathcal{I}$.
\end{itemize}
\end{proposition}

Clearly, $\omega \centernot\longrightarrow ({\mathcal{R}^{n}_{k}}^{+})^{n}_{k}$ holds for all $n,k \geq 2$. Moreover, by virtue of Propositions \ref{Ramsey(omega)_lex} and \ref{Ramsey(omega)_and_Random^n_k}, it follows in particular that, for all $k\geq 2$, both $\mathcal{R}^{2}_{k+1} \leq_{K} \mathcal{R}^{2}_{k}$ and $\mathcal{R}^{3}_{2} \leq_{K} \mathcal{R}^{2}_{k}$; however, it remains to determine whether these pairs of ideals are Kat\v{e}tov equivalent. Next, we show that $\mathcal{R}^{2}_{k} \not\leq_{K} \mathcal{R}^{2}_{k+1}$ and $\mathcal{R}^{2}_{k} \not\leq_{K} \mathcal{R}^{3}_{2}$ for all $k\geq 2$.

\smallskip

\begin{proposition}
\label{newcounterexample1}
    For every $k\geq 2$, the $F_{\sigma}$ tall ideal $\mathcal{R}^{2}_{k+1}$ satisfies $\omega \longrightarrow ({\mathcal{R}^{2}_{k+1}}^{+})^{2}_{k}$ but $\omega \centernot\longrightarrow ({\mathcal{R}^{2}_{k+1}}^{+})^{2}_{k+1}$. 
\end{proposition}

\begin{proof}
   We know that $\omega \centernot\longrightarrow ({\mathcal{R}^{2}_{k+1}}^{+})^{2}_{k+1}$ holds and that $\mathcal{R}^{2}_{k+1} \leq_{K} \mathcal{R}^{2}_{k}$. Now, let $\mathcal{I}$ be the $F_{\sigma}$ tall ideal described by Proposition \ref{counterexample1}, which satisfies $\omega \longrightarrow (\mathcal{I}^+)^{2}_{k}$ but $\omega \centernot\longrightarrow (\mathcal{I}^+)^{2}_{k+1}$. By Proposition \ref{Ramsey(omega)_and_Random^n_k}, it follows that both $\mathcal{R}^{2}_{k} \not\leq_{K} \mathcal{I}$ and $\mathcal{R}^{2}_{k+1} \leq_{K} \mathcal{I}$. Consequently, we have $\mathcal{R}^{2}_{k} \not\leq_{K} \mathcal{R}^{2}_{k+1}$ and hence $\omega \longrightarrow ({\mathcal{R}^{2}_{k+1}}^+)^{2}_{k}$ holds.
\end{proof}

\begin{proposition}
\label{newcounterexample2}
    The $F_{\sigma}$ tall ideal $\mathcal{R}^{3}_{2}$ satisfies $\omega \longrightarrow ({\mathcal{R}^3_2}^+)^{2}_{k}$ for all $k\geq 2$ but $\omega \centernot\longrightarrow ({\mathcal{R}^3_2}^+)^{3}_{2}$. 
\end{proposition}

\begin{proof}
We know that $\omega \centernot\longrightarrow ({\mathcal{R}^3_2}^+)^{3}_{2}$ holds and that $\mathcal{R}^{3}_{2} \leq_{K} \mathcal{R}^{2}_{k}$ for all $k\geq 2$. Now, let $\mathcal{I}$ be the $F_{\sigma}$ tall ideal described by Proposition \ref{counterexample2}, which satisfies $\omega \longrightarrow (\mathcal{I}^+)^{2}_{k}$ for all $k\geq 2$ but $\omega \centernot\longrightarrow (\mathcal{I}^+)^{3}_{2}$. By Proposition \ref{Ramsey(omega)_and_Random^n_k}, it follows that both $\mathcal{R}^{2}_{k} \not\leq_{K} \mathcal{I}$ for all $k\geq 2$ and $\mathcal{R}^{3}_{2} \leq_{K} \mathcal{I}$. Consequently, for all $k\geq 2$ we have $\mathcal{R}^{2}_{k} \not\leq_{K} \mathcal{R}^{3}_{2}$ and hence $\omega \longrightarrow ({\mathcal{R}^3_2}^+)^{2}_{k}$ holds.    
\end{proof}

\smallskip

In what follows, we briefly examine the corresponding non-hereditary versions of the notions of Galvin ideals and Nash-Williams ideals.

\medskip

The non-hereditary Galvin property for ideals is defined as follows: An ideal $\mathcal{I}$ on $\omega$ satisfies

\kern-0.5em
\begin{equation*}
    \omega \longrightarrow (\mathcal{I}^+)^{\mathtt{G}}
\end{equation*}

\kern-0.5em
if for every $\mathcal{F} \subseteq
[\omega]^{<\omega}$ there exists $B\in \i^{+}$
such that $B\in \textrm{hom}(\mathcal{F})$.

\medskip

The non-hereditary Nash-Williams property for ideals is defined as follows: An ideal $\mathcal{I}$ on $\omega$ satisfies

\kern-0.5em
\begin{equation*}
    \omega \longrightarrow (\mathcal{I}^+)^{\mathtt{NW}}_{2}
\end{equation*}

\kern-0.5em
if for every barrier $\mathcal{B}$ on $\omega$ and every coloring $\pi : \mathcal{B} \rightarrow 2$, there exists $H\in \mathcal{I}^{+}$ such that $H \in \textrm{hom}(\pi)$.

\medskip

Clearly, for all $n \in \omega$ with $n \geq 2$, the following implications among non-hereditary Ramsey-type properties of ideals hold:

\kern-0.5em
\begin{equation*}
\omega \longrightarrow (\mathcal{I}^+)^{\mathtt{G}}
\;\;\;\Longrightarrow\;\;\; 
\omega \longrightarrow (\mathcal{I}^+)^{\mathtt{NW}}_{2}
\;\;\;\Longrightarrow\;\;\; 
\omega \longrightarrow (\mathcal{I}^+)^{n}_{2}
\end{equation*}

By Propositions \ref{counterexample2} and \ref{newcounterexample2}, we infer that $\omega \longrightarrow (\mathcal{I}^+)^{2}_{2}$ does not imply $\omega \longrightarrow (\mathcal{I}^+)^{\mathtt{NW}}_{2}$; thus, the second arrow is strict when $n=2$. We next show that the first arrow is also strict, and for this purpose, we shall use the following characterization of closed hereditary tall families, established in \cite[Proposition 2.8]{GrebikUzca2019}.

\begin{proposition}[\cite{GrebikUzca2019}]
\label{closed_hereditary_tall_families}
    For every closed hereditary tall family $\mathcal{K} \subseteq 2^{\omega}$, there exists a family $\mathcal{F} \subseteq [\omega]^{<\omega}$ such that $\textrm{hom}(\mathcal{F}) = \mathcal{K} \cap [\omega]^{\omega}$.
\end{proposition}

Analogously to Proposition \ref{no analytic Galvin tall} in the hereditary setting, we also have that there are no analytic tall ideals satisfying the non-hereditary Galvin property.

\begin{proposition} \label{analytic_tall_implies_non-Galvin(omega)}
    Every analytic tall ideal $\mathcal{I}$ on $\omega$ satisfies $\omega \centernot\longrightarrow (\mathcal{I}^+)^{\mathtt{G}}$.
\end{proposition}

\begin{proof}
  Let $\mathcal{I}$ be an analytic tall ideal on $\omega$. By Proposition \ref{GV}, there exists an $F_\sigma$ tall ideal $\mathcal{J}$ on $\omega$ contained in $\mathcal{I}$. Moreover, by Proposition \ref{Mazur-lemma}, the ideal $\mathcal{J}$ is generated by a closed hereditary tall family $\mathcal{K} \subseteq 2^{\omega}$. Furthermore, by Proposition \ref{closed_hereditary_tall_families}, there exists a family $\mathcal{F} \subseteq [\omega]^{<\omega}$ such that $\textrm{hom}(\mathcal{F}) = \mathcal{K} \cap [\omega]^{\omega}$. Therefore, we infer that $\textrm{hom}(\mathcal{F}) \subseteq \mathcal{K} \subseteq \mathcal{J} \subseteq \mathcal{I}$ and hence $\omega \centernot\longrightarrow (\mathcal{I}^+)^{\mathtt{G}}$ holds.
\end{proof}

\smallskip

Given a tall ideal $\mathcal{I}$ on $\omega$, we say that $\mathcal{I}$ has a Borel selector if there exists a Borel map $S: [\omega]^{\omega} \rightarrow [\omega]^{\omega}$ such that $S(A) \in \mathcal{I} \!\restriction\! A$ for every $A\in [\omega]^{\omega}$. In connection with this notion, it was shown in \cite[Theorem 4.18]{GrebikUzca2019} and \cite[Examples 1 and 2]{Gre-Vid} that some $F_{\sigma}$ tall ideals do not admit Borel selectors.

\begin{proposition}[\cite{GrebikUzca2019, Gre-Vid}] \label{without_Borel_selector}
    There exists an $F_{\sigma}$ tall ideal $\mathcal{I}$ on $\omega$ without a Borel selector.
\end{proposition}

We are looking for an ideal that satisfies the non-hereditary Nash-Williams property but does not satisfy the non-hereditary Nash-Williams property. We claim that the ideal $\mathcal{I}$ mentioned in Proposition \ref{without_Borel_selector} provides such an example.

\begin{theorem} 
    There exists a tall ideal $\mathcal{I}$ on $\omega$ satisfying $\omega \longrightarrow (\mathcal{I}^+)^{\mathtt{NW}}_{2}$ but $\omega \centernot\longrightarrow (\mathcal{I}^+)^{\mathtt{G}}$.
\end{theorem}

\begin{proof}
Let $\mathcal{I}$ be an $F_{\sigma}$ tall ideal on $\omega$ without a Borel selector, whose existence is guaranteed by Proposition \ref{without_Borel_selector}. Since $\mathcal{I}$ is analytic and tall, it follows from Proposition \ref{analytic_tall_implies_non-Galvin(omega)} that $\omega \centernot\longrightarrow (\mathcal{I}^+)^{\mathtt{G}}$ holds.  

\medskip

We claim that $\mathcal{I}$ satisfies $\omega \longrightarrow (\mathcal{I}^+)^{\mathtt{NW}}_{2}$. Indeed, by Lemma \ref{efective-NW}, for any barrier $\mathcal{B}$ on $\omega$ there exists a Borel map $G: [\omega]^{\omega} \times 2^{\mathcal{B}} \rightarrow [\omega]^{\omega}$ such that, for every $A \in [\omega]^{\omega}$ and every coloring $\pi : \mathcal{B} \rightarrow 2$, it holds that $G(A,\pi) \in [A]^{\omega} \cap \textrm{hom} (\pi)$. Bearing this in mind, for each coloring $\pi : \mathcal{B} \rightarrow 2$ we consider the Borel map $G_{\pi} : [\omega]^{\omega} \rightarrow [\omega]^{\omega}$ given by $G_{\pi} (A) = G(A,\pi)$. Therefore, since $\mathcal{I}$ does not have a Borel selector, we conclude that there is some $A\in \mathcal{I}^{+}$ such that $G_{\pi}(A) \in \mathcal{I}^{+} \cap \textrm{hom} (\pi)$. 
\end{proof}

\smallskip

We conclude with the following open question concerning the connection between ideals satisfying the non-hereditary Ramsey property and those satisfying the non-hereditary Nash-Williams property.

\begin{question}
    Is there an ideal $\mathcal{I}$ on $\omega$ satisfying $\omega \longrightarrow (\mathcal{I}^+)^{n}_{2}$ for all $n\geq 2$ but $\omega \centernot\longrightarrow (\mathcal{I}^+)^{\mathtt{NW}}_{2}$?
\end{question}

We also remark that the study of non-hereditary Ramsey-type properties for ideals, both for finite colorings of barriers on $\omega$ and for finite partitions of the family of all finite subsets of $\omega$, as well as their interaction with the Kat\v{e}tov order, deserves further investigation.

\subsection{Reducing the dimension}

In this subsection, we adopt the convention that pairs in $[\omega]^{2}$ and triples in $[\omega]^{3}$ are always listed in increasing order.

\medskip

We give a sufficient condition on families $\mathcal{F} \subseteq [\omega]^{3}$ ensuring the existence of a coloring $\pi : [\omega]^{2} \rightarrow 2$ with $\textrm{hom}(\pi) \subseteq \textrm{hom}(\mathcal{F})$. To this end, we introduce notation for some intervals of $[\omega]^{3}$ with respect to the lexicographic order $\leq_{\textrm{lex}}$. For each $q,i\in \omega$ with $q<i$, let

\kern-0.5em
\begin{gather*}
I(q,i)=\{ t\in [\omega]^{3} : \{q,i,i+1\} \leq_{\textrm{lex}} t <_{\textrm{lex}} \{q,i+1,i+2\} \} \text{, and} 
	\\
\textstyle
I(q)=\{ t\in [\omega]^{3} : \{q,q+1,q+2\} \leq_{\textrm{lex}} t <_{\textrm{lex}} \{q+1,q+2,q+3\} \} = \bigcup_{q<i} I(q,i).
\end{gather*}

Observe that each $I(q,i)$ has order type $\omega$, and each $I(q)$ has order type $\omega^2$. Thus, for every $q\in\omega$ and every family $\mathcal{F} \subseteq [\omega]^{3}$, we have that $\mathcal{F}\cap I(q)$ has order type less
than $\omega^2$ if and only if there is $i_0 >q$ such that $\mathcal{F}\cap I(q,i)$
is finite for all $i\geq i_0$. Keeping this in mind, define

\kern-0.5em
\begin{equation*}
    D(\f)=\{\{q,i\}\in [\omega]^{2} : |\f\cap I(q,i)| <\omega \}.    
\end{equation*}

\begin{proposition}
\label{caso-dim-3}
Let $\mathcal{F}\subseteq [\omega]^{3}$ be such that the order type of $(\mathcal{F}, \leq_{\textrm{lex}})$ is less than $\omega^3$. Then, there exists a coloring
$\pi: [\omega]^{2}\rightarrow 2$ such that $\textrm{hom}(\pi) \subseteq \textrm{hom}(\mathcal{F})$.
\end{proposition}

\begin{proof}
Let $\mathcal{F}\subseteq [\omega]^{3}$ be a family of triples whose lexicographic ordering $(\mathcal{F},\leq_{\textrm{lex}})$ has order type less than $\omega^3$, then

\kern-0.5em
\begin{equation}
\label{colas en D} 
(\exists\, q_0) (\forall\, q\geq q_0) (\exists\, i_0>q) (\forall\, i\geq i_0) (\{q,i\}\in D(\f)).
\end{equation}

Now, for each $\{q,i\}\in D(\f)$ define

\kern-0.5em
\begin{equation*}
f_q(i) = 
\begin{cases} 
\max\{j\in \omega/i : \{q,i,j\}\in\f\}  & \text{if } \;\f\cap I(q,i)\not =\emptyset \\
1 & \text{if } \; \f\cap I(q,i)=\emptyset.
\end{cases}
\end{equation*}

\kern-0.5em
By an easy induction on $q$, one can extend the definition of $f_{q}(i)$ to all $q,i \in \omega$ in such a way that it satisfies the following:

\kern-0.5em
\begin{equation} 
\label{creciente_en_q} 
\mbox{If $i\leq q$, then $f_i(q)\leq f_q(q)$}.
\end{equation}

Consider the coloring $\pi: [\omega]^{2} \rightarrow 2$ defined as follows:

\kern-0.5em
\begin{equation*}
\pi(\{n,m\})=1 \text{ if and only if } (\{n,m\}\in
D(\mathcal{F}) \,\Rightarrow\, m\leq f_n(n)).
\end{equation*}

\kern-0.5em
Observe that $\pi \text{\,''\,} [H]^{2} = \{0\}$ for every $H\in \textrm{hom}(\pi)$. Indeed, by \eqref{colas en D}, we can pick $n,m\in H$, with $n<m$ and $m > f_n (n)$, such that $\{n,m\}\in D(\mathcal{F})$, and hence $\pi(\{n,m\})=0$. 

\medskip 

Finally, we claim that $\textrm{hom}(\pi)\subseteq \textrm{hom} (\mathcal{F})$. Indeed, note that it suffices to show that $[H]^{3}\cap \mathcal{F} = \emptyset$ for all $H\in \textrm{hom}(\pi)$. Suppose, toward a contradiction, that there is some $\{q,i,j\}\in
[H]^{3}\cap \mathcal{F}$. Since $\pi \text{\,''\,} [H]^{2} = \{0\}$, we have $\pi(\{q,i\})=0$ and $\pi(\{i,j\})=0$, and hence $\{q,i\},\{i,j\}\in D(\mathcal{F})$. Now, since $\{q,i,j\}\in \mathcal{F} \cap I(q,i)$, it follows that $j\leq f_q(i)$; moreover, from \eqref{creciente_en_q} we obtain $f_q(i)\leq f_i(i)$. Thus, we deduce that $j\leq f_i(i)$ and hence $\pi(\{i,j\})=1$, a contradiction.
\end{proof}

\smallskip

It should be noted that the previous proposition provides a sufficient, but not necessary, condition for a family $\mathcal{F} \subseteq [\omega]^3$ to admit a coloring $\pi : [\omega]^2 \to 2$ such that $\textrm{hom}(\pi) \subseteq \mathrm{hom}(\mathcal{F})$. Next, we present an example showing that the previous proposition is not an equivalence.

\begin{proposition}
    There exist a family $\mathcal{F} \subseteq [\omega]^{3}$ and a coloring $\pi: [\omega]^{2} \rightarrow 2$ such that the order type of $(\mathcal{F}, \leq_{\textrm{lex}})$ is $\omega^{3}$ and $\textrm{hom}(\pi) \subseteq \mathrm{hom}(\mathcal{F})$.
\end{proposition}

\begin{proof}
Let $\omega=\bigcup_{i\in\omega} A_i$ be a partition of $\omega$ into infinitely many infinite sets such that $i\notin A_i$ for every $i\in\omega$. Consider the family $\mathcal{F} \subseteq [\omega]^{3}$ defined by 

\kern-0.5em
\begin{equation*}
    \mathcal{F} =\{\{n,m,p\} \in [\omega]^{3} : \{m,p\} \subseteq A_n \}.
\end{equation*}

\kern-0.5em
Notice that $\mathcal{F}_{(n)} = [A_n/n]^2$ for each $n\in\omega$, then the order type of each $(\mathcal{F}_{(n)}, \leq_{\textrm{lex}})$ is $\omega^2$, and thus the order type of $(\mathcal{F}, \leq_{\textrm{lex}})$ is $\omega^3$. On the other hand, the complement of $\mathcal{F}$ relative to $[\omega]^{3}$ is given by

\kern-0.5em
\begin{equation*}
\mathcal{F}^\complement = \{\{n,m,p\} \in [\omega]^{3} : m\notin A_n \,\vee\, p\notin A_n\}.
\end{equation*}

\kern-0.5em
Observe that $\{\{m,p\} \in [\omega/n]^{2}: m\in A_i \,\wedge\, p\in A_j \,\wedge\, i\neq j\}\subseteq\f_{(n)}^\complement$ for each $n\in\omega$, then the order type of each $(\mathcal{F}^{\complement}_{(n)}, \leq_{\textrm{lex}})$ is $\omega^2$, and thus the order type of $(\mathcal{F}^{\complement}, \leq_{\textrm{lex}})$ is $\omega^3$.

\medskip 

Consider now the coloring $\pi:[\omega]^2\rightarrow 2$ defined by $\pi(\{n,m\})=1$ if and only if $\{n,m\}\subseteq A_i$ for some $i\in\omega$. Thus, $H\in \textrm{hom}(\pi)$ if and only if either $H\subseteq A_i$ for some $i\in\omega$ or $|H\cap A_i|\leq 1$ for all $i\in\omega$.

\medskip 

Finally, we claim that $\textrm{hom}(\pi)\subseteq \textrm{hom} (\mathcal{F})$. Indeed, note that it suffices to show that $[H]^{3}\cap \mathcal{F}=\emptyset$ for all $H\in \textrm{hom}(\pi)$. On the one hand, let $H\in \textrm{hom} (\pi)$ be such that $H\subseteq A_{i}$ for some $i\in\omega$, then for every $\{n,m,p\} \in [H]^{3}$ we have $n\neq i$ and hence $\{n,m\} \cap A_{n} = \emptyset$, which implies that $\{n,m,p\} \in \mathcal{F}^{\complement}$. On the other hand, let $H\in \textrm{hom} (\pi)$ be such that $|H \cap A_{i}| \leq 1$ for all $i\in\omega$, then for every $\{n,m,p\} \in [H]^{3}$ we have $|\{m,p\} \cap A_{n}| \leq 1$ and hence $m\notin A_n$ or $p\notin A_n$, which implies that $\{n,m,p\} \in \mathcal{F}^{\complement}$.
\end{proof}

\smallskip

Next, we show that the collection of all families $\mathcal{F} \subseteq [\omega]^3$ admitting a coloring $\pi : [\omega]^2 \to 2$ such that $\textrm{hom}(\pi)\subseteq \textrm{hom} (\mathcal{F})$ forms an analytic subset of $2^{[\omega]^3}$. However, at present we do not know whether this collection is Borel.

\begin{proposition}
The collection $\mathcal{A} =\{ \mathcal{F} \subseteq [\omega]^{3} : (\exists\, \pi\in 2^{[\omega]^{2}}) (\textrm{hom}(\pi) \subseteq \textrm{hom}(\mathcal{F})) \}$
 is an analytic subset of $2^{[\omega]^{3}}$.
\end{proposition}

\begin{proof}
Let $R=\{(\pi,\mathcal{F})\in 2^{[\omega]^{2}}\times 2^{[\omega]^{3}} : \textrm{hom}(\pi) \not\subseteq \textrm{hom}(\mathcal{F})\}$, and observe that

\kern-0.5em
\begin{equation*}
(\pi, \mathcal{F})\in R \text{ if and only if } (\exists\, H\in [\omega]^{\omega}) (H\in \textrm{hom}(\pi) \,\wedge\, H\not\in \textrm{hom}(\mathcal{F})).
\end{equation*}

\kern-0.5em
Moreover, note that $\{ (H,\pi) : H\in \textrm{hom}(\pi) \}$ is a closed subset of $[\omega]^{\omega} \times 2^{[\omega]^2}$. Therefore, $R$ is the projection of an $F_\sigma$ set on a compact space, and hence $R$ is itself $F_\sigma$. Consequently, $\mathcal{A}$ is the projection of a $G_\delta$ set, which implies that $\mathcal{A}$ is analytic. 
\end{proof}

\smallskip

As a final remark, we mention that Proposition \ref{caso-dim-3} admits the following natural generalization: For every $n\geq 2$ and every family $\mathcal{F}\subseteq [\omega]^{n+1}$ such that the order type of $(\mathcal{F},\leq_{\textrm{lex}})$ is less than $\omega^{n+1}$, there exists a coloring
$\pi: [\omega]^{n}\rightarrow 2$ such that $\textrm{hom}(\pi) \subseteq \textrm{hom}(\mathcal{F})$. We leave the details to the interested reader.

\section*{Acknowledgments}

The first author was partially supported by the Fondo de Investigaciones de la Facultad de Ciencias de la Universidad de los Andes, grants INV-2025-216-3411 and INV-2026-230-4016. The author also thanks both the University of Toronto and York University in Canada, as well as Universidad Nebrija in Spain, for their hospitality and partial support during the preparation of this article. 

\medskip

The second author thanks both Universidad Nebrija and UNED.

\medskip

The third author was partially supported by the Universidad Industrial de Santander. 

\medskip

The authors thank Michael Hru\v{s}\'{a}k and Paul Szeptycki for several valuable suggestions and stimulating discussions concerning the topics and results presented in this article.


\kern1em
\Addresses


\begin{thebibliography}{99}

\small

\setlength{\itemsep}{0pt}

\bibitem{BT} Baumgartner, T. $\&$ Taylor, A. \textit{Partitions theorems and ultrafilters}. Transactions of the American Mathematical Society, 241: 283--309. (1978).



\bibitem{Booth} Booth, D. \textit{Ultrafilters on a countable set}. Annals of Mathematical Logic. Vol. 2(1): 1--24. (1970).


\bibitem{Cameron} Cameron, P.J. \textit{The random graph}. In: `The mathematics of Paul Erd\H{o}s II' (Eds. Graham, R. $\&$ Ne\v{s}et\v{r}il, J.). Berlin: Springer-Verlag, 333--351. (1997).


\bibitem{CGH2025} Cancino, J.; Guzm\'{a}n, O. $\&$ Hru\v{s}\'{a}k, M. \textit{Ultrafilters and the Kat\v{e}tov order}. In: `Selected topics in combinatorial analysis II' (Ed. Kuzeljevi\'{c}, B.). Belgrade: Matemati\v{c}ki Institut SANU, 137--175. (2025).

\bibitem{Ellentuck} Ellentuck, E. \textit{A new proof that analytic sets are Ramsey}. Journal of Symbolic Logic. Vol. 39(1): 163--165. (1974).

\bibitem{Erdos-Rado}  
Erd\H{o}s, P. $\&$ Rado, R. \textit{A combinatorial theorem}. Journal of the London Mathematical Society, 25(1): 249--255. (1950).

\bibitem{Erdos-Renyi} Erd\H{o}s, P. $\&$ R\'{e}nyi, A. \textit{Asymmetric graphs}. Acta Mathematica Academiae Scientiarum Hungaricae, 14: 295--315. (1963).


\bibitem{Farah1998} 
Farah, I.  \textit{Semiselective coideals}. Mathematika, 45(1): 79--103. (1998).


\bibitem{Farah02}
Farah, I. \textit{How many boolean algebras $\wp(\mathbb{N})/I$ are there?}. Illinois Journal of Mathematics, 46(4): 999--1033. (2002).


\bibitem{FMRS07}  
Filip\'{o}w, R.; Mro\.{z}ek, N.; Rec{\l}aw, I; $\&$ Szuca, P. \textit{Ideal convergence of bounded sequences}. Journal of Symbolic Logic, 72(2): 501--512. (2007). 


\bibitem{FMRS11} 
Filip\'{o}w, R.; Mro\.{z}ek, N.; Rec{\l}aw, I; $\&$ Szuca, P. \textit{Ideal version of Ramsey's theorem}. Czechoslovak Mathematical Journal, 61(2): 280--308. (2011).


\bibitem{FMRS12}

Filip\'{o}w, R.; Mro\.{z}ek, N.; Rec{\l}aw, I; $\&$ Szuca, P. \textit{$\mathcal{I}$-selection principles for sequences of functions}. Journal of Mathematical Analysis and Applications, 396(2): 680--688. (2012).


\bibitem{Galvin} 
Galvin, F. \textit{A generalization of Ramsey's theorem}. Notices of the American Mathematical Society, 15: 548. (1968).




\bibitem{GrebikHrusak}
Greb\'{i}k, J. $\&$ Hru\v{s}\'{a}k, M. \textit{No minimal tall Borel ideal in the Kat\v{e}tov order}. Fundamenta Mathematicae, 248(2): 135--145. (2020).



\bibitem{GrebikUzca2019}
Greb\'{i}k, J. $\&$  Uzc\'{a}tegui, C.
\textit{Bases and Borel selectors for tall families}. Journal of Symbolic Logic, 84(1): 359--375. (2019).



\bibitem{Gre-Vid} 
Greb\'{i}k, J. $\&$ Vidny\'{a}nszky, Z. \textit{Tall $F_\sigma$ subideals of tall analytic ideals}.
Proceedings of the American Mathematical Society, 151(9): 4043--4046. (2023).  


\bibitem{Grigorieff} 
Grigorieff, S. \textit{Combinatorics on ideals and forcing}. Annals of Mathematical Logic, 3(4): 363--394. (1971).


\bibitem{Guzman} 
Guzm\'{a}n, O. \textit{On completely separable MAD families}. Set Theory: Reals and Topology. Proceedings of the Research Institute for Mathematical Sciences of the Kyoto University, 2198: 7--28. (2021).




\bibitem{Halbeisen} Halbeisen, L. \textit{Combinatorial set theory}. (2nd ed.). London: Springer-Verlag. (2017).


\bibitem{Hrusak2011} 
Hru\v{s}\'{a}k, M. \textit{Combinatorics of filters and ideals}. Contemporary Mathematics, 533: 29--69. (2011).


\bibitem{Hrusak2017}
Hru\v{s}\'{a}k, M. \textit{Kat\v{e}tov order on Borel ideals}. Archive for Mathematical Logic, 56: 831--847. (2017)

\bibitem{HMTU2017} 
Hru\v{s}\'{a}k, M.; Meza-Alc\'{a}ntara, D.; Th\"{u}mmel, E.; $\&$ Uzc\'{a}tegui, C. \textit{Ramsey type properties of ideals}. Annals of Pure and Applied Logic, 168(11): 2022--2049. (2017). 

\bibitem{Kechris94}
Kechris, A. S. \textit{Classical descriptive set theory}. New York: Springer-Verlag. (1994).


\bibitem{KuSze} 
Kubi\'{s}, W. $\&$ Szeptycki, P. \textit{On a topological Ramsey theorem}. Canadian Mathematical Bulletin, 66(1): 156--165. (2023). 


\bibitem{Kwela} Kwela, A. \textit{A note on a new ideal}. Journal of Mathematical Analysis and Applications, 430(2): 932--949. (2015).


\bibitem{LA-OP-UA} L\'{o}pez-Abad, J.; Olmos-Prieto, V. $\&$ Uzc\'{a}tegui-Aylwin, C. \textit{$F_{\sigma}$-ideals, colorings, and representation in Banach spaces}. arXiv:2501.15643. (pre-print).


\bibitem{Mathias} 
Mathias, A.R.D. \textit{Happy families}. Annals of Mathematical Logic, 12(1): 59--111. (1977).


\bibitem{Mazur} 
Mazur, K. \textit{$F_\sigma$-ideals and $\omega_1\omega_1^*$-gaps in the Boolean algebras $\wp(\omega) / \mathcal{I}$}. Fundamenta Mathematicae, 138(2): 103--111. (1991).

\bibitem{Meza}
Meza-Alc\'{a}ntara, D. \textit{Ideals and filters on countable sets}. PhD Thesis. Morelia: Universidad Nacional Aut\'{o}noma de M\'{e}xico. (2009).




\bibitem{NashWilliams} Nash-Williams, C. \textit{On well--quasi--ordering transfinite sequences}. Mathematical Proceedings of the Cambridge Philosophical Society, 61: 33--39. (1965).

\bibitem{PG} Pelayo-G\'{o}mez, J.J. \textit{Infinite games and Ramsey properties of $F_{\sigma}$ ideals}. Journal of Symbolic Logic. (To appear). 


\bibitem{Ramsey} 
Ramsey, F.P. \textit{On a problem of formal logic}. Proceedings of the London Mathematical Society, 30(4): 264--286. (1930).




\bibitem{Todorcevic2005} 
Todor\v{c}evi\'{c}, S. \textit{High--dimensional Ramsey theory and Banach space geometry}. In: `Ramsey methods in analysis' (Eds: Argyros, S. $\&$ Todor\v{c}evi\'c, S.). Basel: Birkh\"{a}user Verlag, 121--252. (2005).


\bibitem{Todorcevic2010}
Todor\v{c}evi\'{c}, S. \textit{Introduction to Ramsey spaces}. New Jersey: Princeton University Press. (2010).


\bibitem{Todorcevic1997} 
Todor\v{c}evi\'{c}, S. \textit{Topics in topology}. Berlin: Springer-Verlag. (1997).



\bibitem{Uzc} 
Uzc\'{a}tegui-Aylwin, C. \textit{Ideals on countable sets: a survey with questions}.  Revista Intagraci\'on, temas de matem\'aticas, 37(1): 167--198. (2019).


\end{thebibliography}
\end{document}